\documentclass[12pt,twoside, reqno]{amsart}

\usepackage{amssymb}
\usepackage{amsmath}
\usepackage{mathrsfs}
\usepackage{amsthm}
\usepackage{amsfonts}
\usepackage{txfonts}
\usepackage{enumerate}
\usepackage{hyperref}
\hypersetup{hidelinks}
\usepackage{graphicx}
\usepackage{subfigure}
\usepackage{float}
\usepackage{latexsym,amssymb}
\usepackage{color}
\usepackage{graphicx}
\usepackage{pgf,tikz}
\usepackage{mathrsfs}
\usepackage{fancyhdr}
\usetikzlibrary{arrows}

\allowdisplaybreaks
\usepackage{indentfirst}
\usepackage{setspace}
\usepackage{geometry}
\newtheorem{theorem}{Theorem}[section]
\newtheorem{proposition}[theorem]{Proposition}
\newtheorem{corollary}[theorem]{Corollary}
\newtheorem{lemma}[theorem]{Lemma}
\newtheorem{remark}[theorem]{Remark}
\newtheorem{definition}[theorem]{Definition}

\numberwithin{equation}{section} 

\numberwithin{equation}{section}

\newcommand{\R}{\mathbb{R}}

\newcommand{\ben}{\begin{eqnarray*}}
\newcommand{\enn}{\end{eqnarray*}}
\newcommand{\pa}{\partial}

\newcommand{\ve}{\varepsilon}

\newcommand{\al}{\alpha}

\newcommand{\ol}{\overline}

\newcommand{\half}{\frac{1}{2}}

\newcommand{\na}{\nabla}
\newcommand{\be}{\begin{equation}}
\newcommand{\ee}{\end{equation}}
\newcommand{\ba}{\begin{aligned}}
\newcommand{\ea}{\end{aligned}}

\newcommand{\wt}{\tilde{w}}

\def\9{{\infty}}
\def\a{{\alpha}}

\def\del{{\delta}}
\def\lbb{{\lambda}}
\def\t{{\theta}}

\def\calo{{\mathcal{O}}}
\def\calp{{\mathcal{P}}}

\def\calu{{\mathcal{U}}}
\def\bbp{{\mathbb{P}}}
\def\bbr{{\mathbb{R}}}
\def\ve{{\varepsilon}}
\def\vf{{\varphi}}

\def\wt{\widetilde}

\def\ol{\overline}
\def\({\left(}
\def\){\right)}
\def\<{\langle}
\def\>{\rangle}

\def\bbx{{\mathbb{X}}}

\begin{document}

\title{Multi solitary waves to stochastic nonlinear Schr\"odinger equations}

\author{Michael R\"ockner}
\address{Fakult\"at f\"ur Mathematik,
Universit\"at Bielefeld, D-33501 Bielefeld, Germany}
\email{roeckner@math.uni-bielefeld.de}
\thanks{}

\author{Yiming Su}
\address{Department of mathematics,
Zhejiang University of Technology, 310014 Zhejiang, China}
\email{yimingsu@zjut.edu.cn}
\thanks{}

\author{Deng Zhang}
\address{School of mathematical sciences,
Shanghai Jiao Tong University, 200240 Shanghai, China }
\email{dzhang@sjtu.edu.cn}
\thanks{}

\subjclass[2010]{}


\date{}


\begin{abstract}
  In this paper, we present a pathwise construction of multi-soliton solutions for
  focusing stochastic nonlinear Schr\"odinger equations
  with linear multiplicative noise, in both the $L^2$-critical and subcritical cases.
  The constructed multi-solitons behave asymptotically as
  a sum of $K$ solitary waves,
  where $K$ is any given finite number.
  Moreover,
  the convergence rate of the remainders can be of
  either exponential or polynomial type,
  which reflects the effects of the noise in the system on the asymptotical behavior
  of the solutions.
  The major difficulty in our construction of stochastic multi-solitons
is the absence of pseudo-conformal invariance.
Unlike in the deterministic case \cite{M90,RSZ21},
the existence of stochastic multi-solitons
cannot be obtained from that of stochastic multi-bubble blow-up solutions
in \cite{RSZ21,SZ20}.
  Our proof is mainly based on the rescaling approach in \cite{HRZ18},
  relying on two types of Doss-Sussman transforms,
  and on the modulation method in \cite{CF20,MM06},
  in which the crucial ingredient is the monotonicity of
  the Lyapunov type functional constructed
  by Martel, Merle and Tsai \cite{MMT06}.
  In our stochastic case,
  this functional depends on the Brownian paths in the noise.
\end{abstract}

\maketitle
\begin{spacing}{1.02}
\end{spacing}

\section{Introduction and formulation of main results}  \label{Sec-Intro-Main}

\subsection{Introduction} \label{Subsec-Intro}

In this paper we consider the following type of
focusing stochastic nonlinear Schr\"odinger equations
(SNLS for short)
with linear multiplicative noise:
\begin{equation}  \label{equa-X-rough}
\left\{ \begin{aligned}
     &  dX(t) = i\Delta X(t)dt + i|X(t)|^{p-1}X(t) dt - \mu(t) X(t) dt + \sum\limits_{k=1}^N X(t) G_k(t)dB_k(t),  \\
     &  X(T_0)= X_0 \in H^1(\bbr^d).
\end{aligned} \right.
\end{equation}
Here, $1<p\leq 1+\frac 4d$, $d\geq 1$, $T_0\geq 0$,
$\{B_k\}$ are the standard $N$-dimensional real valued Brownian motions
on a normal stochastic basis $(\Omega, \mathscr{F}, \{\mathscr{F}_t\}, \bbp)$,
$G_k(t,x)=  i\phi_k(x)g_k(t),\ \ x\in \bbr^d,\ t\geq 0$,
$\{\phi_k\} \subseteq C_b^\9(\bbr^d, \bbr)$,
$\{g_k\} \subseteq C^\a(\bbr^+, \bbr)$,
$\a\in(\frac 13, \frac 12)$,
are $ \{\mathscr{F}_t\}$-adapted processes
that are controlled by $\{B_k\}$,
and $X(t) G_k (t) dB_k(t)$ is taken in the sense of controlled rough paths
(see Definition \ref{X-roughpath-def} below).
The term $\mu$ is   of form
\begin{align}
\mu (t,x) = \frac 12 \sum_{k=1}^N  \phi_k(x)^2 g_k(t)^2,\ \ x\in \bbr^d,\ t\geq 0,
\end{align}
such that the conservation law of mass is satisfied.
In particular,
if the processes are $\{\mathscr{F}_t\}$-adapted,
then the rough integration coincides with the usual It\^o integration
(\cite[Chapter 5]{FH14}),
and $-\mu X dt + \sum_{k=1}^N X G_k(t)dB_k(t)$ is exactly the standard Stratonovich differential.
For convenience,
we  focus on the case $N<\9$,
but the infinite case $N=\9$ can also be treated under suitable summability conditions
of the spatial functions $\{\phi_k\}$.

Nonlinear Schr\"odinger equations have various applications in contunuum mechanics,
plasma physics and optics.
In crystals the noise corresponds to scattering of excitons by phonons,
due to thermal vibrations of the molecules,
and its effect on the coherence of the ground state solitary solutions
was investigated
in the two-dimensional $L^2$-critical case in \cite{BCIRG95}  (see also \cite{BCIR94}).
The influence of noise on the collapse
was also studied in \cite{RGBC95} for the $L^2$-critical case in dimensions $d=1,2$.
Another important application can be found in open quantum systems,
where the noise is of non-conservative type
and $\{\|X(t)\|^2_{L^2}\}$ is a continuous martingale
such that the mean $\mathbb{E}\|X(t)\|_{L^2}^2$ is conserved
and the ``physical'' probability law can be defined.
We refer to \cite[Section 2]{BG09} for more physical interpretations.
We also refer to \cite{CHJK17,BDM02,DL02,MRRY20,MRY20}
for the numerical experiments to investigate
the dynamics of stochastic solutions.

It is known that SNLS is $H^1$ globally well-posed in the $L^2$-subcritical case $1<p<1+ \frac 4d$,
and is locally well-posed  in the critical case $p=1 + \frac 4d$.
See, e.g.,  \cite{BD03,BM13,BRZ16} and references therein.

The large time behavior of solutions, however,
are more delicate.
Different phenomena have been exhibited in the defocusing and focusing cases.

As a matter of fact,
for the canonical  nonlinear Schr\"odinger equation
(NLS for short)
\begin{equation}   \label{equa-NLS}
\left\{ \begin{aligned}
     &  du = i\Delta udt + \lbb i|u|^{p-1}u dt,  \\
     &  u(T_0)= u_{0} \in H^1(\bbr^d),
\end{aligned} \right.
\end{equation}
in the defocusing $L^2$-critical case
(i.e., $\lbb =-1, p=1+\frac 4d$),
solutions
exist globally and even scatter at infinity,
i.e., solutions behave asymptotically as free linear solutions.
See the works by Dodson \cite{D12,D16.1,D16.2}.
The scattering phenomena are also exhibited in the stochastic case.
We refer to \cite{HRZ18} for the  $H^1$-subcritical and critical cases,
and \cite{FX18.1,FX18.2,FXZ21,FZ20,Z18} for the $L^2$-critical case.

However, in the focusing $L^2$-critical case (i.e., $\lbb =1$, $p=1+\frac 4d$)
different dynamics appear.
An important role is played by the mass of the {\it ground state},
which is the unique radial solution to the nonlinear elliptic equation
\begin{align} \label{equa-Q}
\Delta Q- Q+Q^{p}=0.
\end{align}
By \cite[Theorem 1]{BL83} (see also \cite[Theorem 8.1.1]{C}),
$Q$ is smooth and decays at infinity exponentially fast,
i.e., there exist $C, \delta>0$ such that for any multi-index $|\upsilon|\leq 3$,
\be\label{Q-decay}
|\partial_x^\upsilon Q(x)|\leq C e^{-\delta |x|}, \ \ x\in \bbr^d.
\ee

On one hand, in the subcritical mass regime $\|u_0\|_{L^2}^2 <\|Q\|_{L^2}^2$,
solutions exist globally and scatter at infinity,
see \cite{D15}.
On the other hand,
in the (super)critical mass regime $\|u_0\|_{L^2}^2 \geq \|Q\|_{L^2}^2$,
solutions may form singularities in finite time
or do not scatter at infinity.

One typical blow-up dynamics in the critical  mass regime is
the {\it pseudo-conformal blow-up solution}
\begin{align}  \label{S-blowup-intro}
    S_T(t,x)=(w(T-t))^{-\frac d2}Q \(\frac{x-x^*}{w(T-t)}\)
             e^{ - \frac i 4 \frac{|x-x^*|^2}{T-t} + \frac{i}{w^2(T-t)} + i\vartheta},
\end{align}
where $T\in \bbr$, $w>0$, $x^* \in \bbr^d$ and $\vartheta \in \bbr$.
We note that,
$\|S_T\|_{L^2}^2 = \|Q\|_{L^2}^2$ and
$S_T$ blows up at time $T$ with  speed $\|\na S_T(t)\| \sim (T-t)^{-1}$.
A remarkable result proved by Merle \cite{M93} is that,
the pseudo-conformal blow-up solution is the unique critical mass
blow-up solution to $L^2$-critical NLS, up to the symmetries of
the equation.

Another important dynamics is
the {\it solitary wave}
\begin{align}  \label{R-def}
    R(t,x):=Q_{w} \(x-v^*t-x^0\)e^{i(\half v^*\cdot x-\frac{1}{4}|v^*|^2t+w^{-2}t+\vartheta)}.
\end{align}
where the parameters $x^0\in \bbr^d$,
and $w \in \bbr^+$,  $v^* \in \bbr^d$, $\vartheta\in \bbr$,
correspond to the frequency,
propagation speed
and phase, respectively,
and
\be
Q_{w}(x)=w^{-\frac {2}{p-1}}Q \(\frac{x}{w} \),
\ee
satisfying  the nonlinear elliptic equation
\begin{align} \label{equa-Qw}
\Delta Q_w- w^{-2}Q_w+Q_w^{p}=0.
\end{align}
In contrast to the above scattering solutions
and pseudo-conformal blow-up solutions,
the solitary wave exists globally but does not scatter at infinity.
An important underlying relationship is that,
the solitary wave and the pseudo-conformal blow-up solution
can be transformed into each other through the
{\it pseudo-conformal transform}:
\begin{align} \label{pseu-conf-transf}
   S_T(t,x) = \mathcal{C}_T(R)(t,x):= \frac{1}{(T-t)^{\frac d2}}R \(\frac{1}{T-t}, \frac{x}{T-t}\) e^{-i\frac{|x|^2}{4(T-t)}}, \ \
   t\not =T,\ x^* =v^* + (T-t)x^0.
\end{align}

Furthermore,
according to the famous {\it soliton resolution conjecture},
global solutions to nonlinear dispersive equations are expected to
behave asymptotically as a sum of solitary waves
plus a dispersive part.
One particular global solution is the {\it multi-soliton} (or, {\it multi-solitary wave solution}),
which is defined on $[T_0, \9)$ for some $T_0 \in \bbr$
and satisfies
\begin{align} \label{u-R-H1-o1}
 \|u(t) - \sum\limits_{k=1}^K R_k(t)\|_{H^1} \to 0,\ \ as\ t\ \to \9,
\end{align}
where $K \in \mathbb{N}\setminus\{0\}$, $R_k$ is the solitary wave of form
\begin{align}  \label{Rk-def}
    R_k(t,x):=Q_{w_{k}^0} \(x-v_k t-x_{k}^0 \)e^{i(\half v_k\cdot x-\frac{1}{4}|v_k|^2t+(w_{k}^0)^{-2}t+\theta_{k}^0)},
\end{align}
with parameters $w_{k}^0 \in \bbr^+$, $v_k, x_{k}^0 \in \bbr^d$, $\theta_{k}^0 \in \bbr$,
and where $Q_{w_{k}^0}$ satisfies equation \eqref{equa-Qw}
with $w_k^0$ replacing $w$, $1\leq k\leq K$.
That is, the multi-soliton behaves exactly as a sum of  solitons
without loss of mass by dispersion.

Multi-solitons have attracted significant interest in the literature.
The construction of multi-solitons to NLS in the non-integrable case
was initiated by Merle \cite{M90} in the $L^2$-critical case.
The proof in \cite{M90} is based on the construction of
multi-bubble pseudo-conformal blow-up solutions
and on the pseudo-conformal invariance.
Afterwards,
multi-solitons in the $L^2$-subcritical and supercritical cases
were constructed, respectively,
by Martel and Merle \cite{MM06}
and by C\^ote, Martel and Merle \cite{CMM11}.
The method in \cite{MM06,CMM11} is quite different from that of \cite{M90}.
It relies on the modulation analysis
and the monotonicity of functionals adapted to multi-solitons.
This method has also been applied in the study of the stability problem of multi-solitons,
\cite{MMT06}.
Recently, for quite general nonlinearities,
the smoothness and conditional uniqueness of multi-solitons
were studied by C\^ote and Friederich \cite{CF20}
in both the $L^2$-subcritical and critical cases.
The uniqueness issue of multi-solitons to $L^2$-critical NLS,
particularly in the low  asymptotical regime,
was recently studied in  \cite{CSZ21}.

Multi-solitons are also exhibited in various models.
For the generalized Korteweg-de Vries (gKdV) equations,
we refer to the pioneering work by Martel  \cite{Ma05},
where the construction and uniqueness of multi-solitons were proved in the subcritical and critical cases.
The construction and classification in the supercritical case
were obtained by C\^ombet \cite{Co11}.
We also refer to \cite{L19} for the classification of dynamics near solitons
for the $L^2$-critical gKdV equation with a saturated perturbation.
For other dispersive equations,
see, e.g., \cite{CM14} for the Klein-Gordon equation,
\cite{K-M-R} for the Hartree equation
and \cite{MRT15} for the water-waves system.

In the stochastic case,
for the one dimensional cubic SNLS,
the small noise asymptotics of the tails of mass and timing jitter in soliton transmission
was studied by Debussche and Gautier \cite{DG08}.
Moreover, the influence of noise on the propagation of standing waves
was studied by de Bouard and Fukuizumi \cite{DF09}
for the Bose-Einstein condensation,
where the trapping potential varies randomly in time.
Quite interestingly, it was proved in \cite{DF09}
that the solution decomposes into the sum of a randomly modulated standing wave and a small remainder,
and the first order of the remainder converges to a Gaussian process,
as the amplitudes of noise tends to zero.
For the stochastic KdV equations
we refer to  \cite{BD07}  for the random modulation of solitons,
and \cite{BD09,DG10} for the exist problem from a neighborhood of solitons

The main interest of this paper is to understand the quantitative properties of
soliton dynamics for SNLS.

Recently,
several typical blow-up dynamics have been constructed for SNLS.
Critical mass blow-up solutions were constructed in \cite{SZ19},
which yields that the mass of the ground state is exactly the
threshold of the global existence and blow-up for SNLS.
The loglog blow-up solutions and the multi-bubble pseudo-conformal blow-up solutions
were  constructed in \cite{FSZ20} and \cite{SZ20}, respectively.
Very recently, in \cite{RSZ21}
we also constructed the multi-bubble Bourgain-Wang type solutions,
which behave asymptotically as a sum of pseudo-conformal blow-up solutions
and a smooth residue.
This, in particular, provides examples for the mass quantization conjecture (\cite{MR05}).
Another interesting outcome is the existence of non-pure multi-solitons to $L^2$-critical NLS,
which behave as a sum of multi-solitons plus a scattering  part,
predicted by the  soliton resolution conjecture.

It should be mentioned that,
one major difficulty in our construction of stochastic multi-solitons
is the absence of pseudo-conformal invariance.
Unlike in the deterministic case \cite{M90,RSZ21},
the existence of stochastic multi-solitons
cannot be obtained from that of stochastic multi-bubble blow-up solutions
in \cite{RSZ21,SZ20}.

In the present work,
we provide path-by-path constructions of stochastic multi-solitons to SNLS.
More precisely,
in both the $L^2$-subcritical and critical cases $1<p\leq 1+\frac 4d$,
for $\bbp$-a.e. $\omega \in \Omega$,
the multi-solitons to \eqref{equa-X-rough} are constructed and,
up to a random phase transformation,
behave asymptotically as a sum of $K$ solitary waves,
where $K$ is any given finite number.
Quite interestingly,
the decay rate of the corresponding asymptotical behavior
can be of either exponential or polynomial type,
which is  closely related to that of the
spatial functions $\{\phi_l\}$ and temporal functions $\{g_l\}$ in the noise.
To the best of our knowledge,
this provides the first explicit constructions of multi-solitons
to SNLS.

Our strategy of proof is mainly based on the rescaling approach
and  the modulation method.

The rescaling approach in \cite{HRZ18} relies on two types of Doss-Sussman transforms,
which enable us to study the large time behavior of solutions
by transforming the original equation to random Schr\"odinger equations,
for which the sharper {\it pathwise analysis} can be performed.
This method is actually quite robust for many other stochastic partial differential equations,
see, e.g., \cite{ABD21,BDR09,BF12} and references therein.
The solvability relationship between two equations via the transform is indeed nontrivial in infinite dimensional spaces.
An interesting outcome here is, that we extend
the solvability in the critical case for dimensions $d=1,2$  in \cite{SZ19}
to the entire (sub)critical regime for all dimensions.

Let us mention that,
the pathwise analysis in \cite{HRZ18} is based on the stability of scattering
which, however, is quite difficult for the multi-solitons
in the subcritical case (\cite{MMT06}) and even fails in the critical case.
Instead, we construct multi-solitions
in a direct way
by using the modulation method
and analysing the Lyapunov type functional constructed by Martel, Merle and Tsai \cite{MMT06}.

It is also worth noting that,
in order to treat the subcritical and critical cases in a uniform manner,
the soliton profiles in the geometrical decomposition here exhibit
a quite unified structure in both cases,
which is different from the works \cite{CF20,MM06}.
The unstable direction  ${\rm Re}\<\wt R_k,\ve\>$ (see Corollary \ref{Cor-Rkve-cri} below)
is not involved in the geometrical decomposition.
Instead, it will be controlled by the almost
conservation of the local mass.
This permits to fix the frequency $w_k\equiv w_k^0$ in the subcritical case
and, in particular, simplifies the derivation of the time-independent main part  of
the Lyapunov type functional.
In the critical case, though the frequency $w_k$ varies with time,
the main part keeps still time independent,
due to the scaling invariance and the key Pohozaev identity.
\vspace{1ex}

{\bf Notations.}
For any $x=(x_1,\cdots,x_d) \in \bbr^d$
and any multi-index $\nu=(\nu_1,\cdots, \nu_d)$,
let $|\nu|= \sum_{j=1}^d \nu_j$,
$\<x\>=(1+|x|^2)^{1/2}$,
$\partial_x^\nu=\partial_{x_1}^{\nu_1}\cdots \partial_{x_d}^{\nu_d}$
and
$\<\na\>=(I-\Delta)^{1/2}$.

For $1\le p\le\9$,
$L^p = L^p(\bbr^d)$ is
the space of $p$-integrable (complex-valued) functions
endowed with the norm $\|\cdot\|_{L^p}$,
and $W^{s,p}$ denotes the standard Sobolev space,
$s\in \bbr$.
In particular,
$L^2(\bbr^d)$ is the Hilbert space endowed with the inner product
$\<v,w\> =\int_{\bbr^d} v(x)\bar w(x)dx$,
and $H^s:= W^{s,2}$.
As usual,
$L^q(0,T;L^p)$ means the space of all integrable $L^p$-valued functions $f:(0,T)\to L^p$ with the norm
$\|\cdot\|_{L^q(0,T;L^p)}$,
and $C([0,T];L^p)$ denotes the space of all $L^p$-valued continuous functions on $[0,T]$ with the sup norm over $t$.
For any H\"older continuous function $f\in C^\a(I)$, $\a>0$ and $I\subseteq \bbr^+$,
we write $\delta f_{st} := f(t)-f(s)$, $s,t\in I$,
and $\|f\|_{\a, I} := \sup_{s,t\in I,s\not =t} \frac{|\delta f_{st}|}{|s-t|^\a}$.
Let $C_c^\9$ be the space of all compactly supported smooth functions on $\bbr^d$.
We also set $\dot{g}:= \frac{d}{dt}g$ for any $C^1$ functions.

The symbol $u =
\mathcal{O}(v)$ means that $|u/v|$ stays bounded, and $v_n=o(1)$
means $|v_n|$ tends to zero as $n\to \9$.
Throughout this paper,
we use $C,\delta$ for various constants that may
change from line to line.

\subsection{Formulation of main results}  \label{Subsec-Main}

To begin with,
we first recall some basic notions of controlled rough paths.
For more details of the theory of (controlled) rough path we refer to \cite{FH14,G04} and the references therein.

Given a path $X \in C^\a([0,T]; \bbr^N)$, $0<T<\9$,
we say that $Y \in C^\a([0,T]; \bbr^N)$ is controlled by $X$ if
there exists $Y' \in C^\a([0,T]; \bbr^{N\times N})$ such that
the remainder term $R^Y$ implicitly given by
\begin{align*}
   \delta Y_{j,st} = \sum\limits_{k=1}^N Y'_{jk}(s) \delta X_{k,st} + \delta R^Y_{j,st}
\end{align*}
satisfies $\|R_j^Y\|_{2\a, [s,t]} <\9$,
$1\leq j\leq N$.
This defines the controlled rough path $(Y,Y')\in \mathscr{D}_X^{2\a}([0,T]; \bbr^N)$,
and $Y'$ is the so called Gubinelli's derivative.

One typical example is that, the $N$-dimensional Brownian motions
$B=(B_j)_{j=1}^N$ can be enhanced to a rough path
${\bf B} = (B, \mathbb{B})$,
where $\mathbb{B}_{jk,st}:= \int_s^t \delta B_{j,sr} dB_k(r)$
with the integration taken in the sense of It\^o,
$\delta B_{j,st} = B_j(t) - B_j(s)$.
It is known that $\|B\|_{\a,[0,T]} <\9$, $\|\mathbb{B}\|_{2\a, [s,t]}<\9$, $\bbp$-a.s.,
where $\frac 1 3 <\a<\frac 12$
(see \cite[Section 3.2]{FH14}).

Given a path $Y$ controlled by the $N$-dimensional Brownian motion,
i.e., $Y\in \mathscr{D}_B^{2\a}([S,T]; \bbr^N)$,
$0<S<T<\9$,
we can define the rough integration of $Y$ against ${\bf B}=(B,\mathbb{B})$ as follows
(see \cite[Theorem 4.10]{FH14}),
for each $1\leq k\leq N$,
\begin{align} \label{def-rp}
   \int_S^T Y_k(r) dB_k(r)
   := \lim\limits_{|\mathscr{P}|\to 0} \sum\limits_{i=0}^{n-1}
       \(Y_k(t_i) \delta B_{k,t_it_{i+1}}
         + \sum\limits_{j=1}^N Y'_{kj}(t_i) \mathbb{B}_{jk,t_it_{i+1}}\),
\end{align}
where $\mathscr{P}:= \{t_0, t_1,\cdots, t_n\}$ is a partition of $[S,T]$
so that $t_0 =S$, $t_n=T$,
$|\mathscr{P}|:= \max_{0\leq i\leq n-1} |t_{i+1} - t_i|$.

Throughout this paper we assume that

{\bf (A0)}
  For every $1\leq l\leq N$,
  \begin{align} \label{asymp-flat}
       \lim\limits_{|x|\to \9} |x|^2 |\partial_x^\upsilon \phi_l(x)|=0, \ \ \upsilon \not =0.
  \end{align}

{\bf (A1)}
  For every $1\leq l\leq N$, $\{g_l\}$
  are $\{\mathscr{F}_t\}$-adapted continuous processes
  and controlled by the Brownian motions $\{B_l\}$,
  i.e., $\{g_l\} \subseteq \mathscr{D}_B^{2\a}(\bbr^+; \bbr^N)$
  with the Gubinelli  derivative $\{g'_{lj}\}_{j,l=1}^N$.
  In addition,
  $\phi_l$ and $g_l$ satisfy one of the following two cases:

  {\rm Case (I):}
  $g_l\in L^2(\bbr^+)$, $\bbp-a.s$., and
  there exists $c_l>0$ such that
  \begin{align} \label{phil-exp-decay}
     \sum\limits_{|\nu|\leq 4} |\partial^\nu \phi_l(x)| \leq C e^{-c_l |x|}.
  \end{align}

  {\rm Case (II):}
  $\bbp$-a.s., $g_l\in L^2(\bbr^+)$
  and there exists $c^*>0$ such that for $t$ large enough,
  \begin{align}   \label{gl-t2-decay}
     \int_t^\9 g_l^2 ds \log \(\int_t^\9 g_l^2 ds \)^{-1} \leq \frac {c^*}{t^2}.
  \end{align}
  \qquad \qquad \ \ \ \ In addition, let $\upsilon_*\in \mathbb{N}$. $\phi_l$ satisfies that
  \begin{align} \label{phil-poly-decay}
     \sum\limits_{|\upsilon|\leq 4} |\partial^\upsilon \phi_l(x)| \leq C |x|^{-\upsilon_*}.
  \end{align}

\begin{remark}
{\rm Case (I)} and {\rm Case (II)}
correspond, respectively, to the exponential and polynomial decay rates of noises.
Let us also mention that, the asymptotics \eqref{gl-t2-decay} is closely related to the
Levy H\"older continuity of Brownian motions.
See the proof of \eqref{B*-Levy-contin} below.
\end{remark}

For simplicity,
we mainly focus on the case $c_l=1$, $1\leq l\leq N$,
and denote by $\phi$ the decay functions in \eqref{phil-exp-decay} and \eqref{phil-poly-decay},
i.e.,
\begin{align} \label{phi-def}
   \phi(x) := \left\{
               \begin{array}{ll}
                 e^{-|x|}, & \hbox{in {\rm Case (I)};} \\
                 |x|^{-\upsilon_*}, & \hbox{in {\rm Case (II)}.}
               \end{array}
             \right.
\end{align}

The solution to \eqref{equa-X-rough} is taken in the sense of controlled rough path.
\begin{definition} \label{X-roughpath-def}
Let $1<p\leq 1+\frac 4d$, $d\geq 1$.
We say that $X$ is a solution to \eqref{equa-X-rough} on $[T_0,\tau^*)$,
where $T_0, \tau^*\in (0,\9]$ are random variables,
if $\bbp$-a.s. for any $\vf\in C_c^\9$,
$t \mapsto \<X(t), \vf\>$ is continuous on $[T_0,\tau^*)$
and for any $T_0 \leq s<t<\tau^*$,
\begin{align*}
   \<X(t)-X(s), \vf\>
   - \int_s^t \<i X, \Delta\vf\>  + \<i|X|^{p-1} X, \vf\>  - \< \mu X, \vf\> dr
   = \sum\limits_{k=1}^N \int_s^t \<i\phi_k g_k X, \vf\> dB_k(r).
\end{align*}
Here the integral $\int_s^t \<i\phi_k g_k X, \vf\> d B_k(r)$
is taken in the sense of controlled rough path
with respect to the rough paths $(B, \mathbb{B})$,
that is,
$\<i\phi_k g_k X, \vf\> \in C^\a([s,t])$,
\begin{align} \label{phikX-st}
   \delta (\<i\phi_k X, \vf\>)_{st}
   =  \sum\limits_{j=1}^N \<-\phi_j\phi_k g_j(s) g_k(s) X(s) + i \phi_k g'_{kj}(s) X(s), \vf\> \delta B_{j,st}
     + \delta R_{k,st},
\end{align}
and
$\|\<\phi_j\phi_k g_j g_kX, \vf\> \|_{\a, [s,t]}
+ \|\<\phi_k g'_{kj} X, \vf\> \|_{\a, [s,t]}  <\9, \ \
   \|R_k\|_{2\a, [s,t]} <\9,$
$\a\in (\frac 13, \frac 12)$.
\end{definition}

The $H^1$ local solvability of \eqref{equa-X-rough}
can be proved by using the fixed point arguments as in \cite[Theorems 1.2 and 2.1]{BRZ16}.
The key ingredients are the
Strichartz and local smoothing estimates for Schr\"odinger equations with lower order perturbations,
due to the asymptotical flatness condition \eqref{asymp-flat}.
See, e.g., \cite{MMT08,Z17}.
It also relies on Theorem \ref{Thm-Equiv-X-u} below,
which relates equations \eqref{equa-X-rough} and \eqref{equa-u-low}
through the Doss-Sussman type transform \eqref{resc-t}.

The main result of this paper is formulated in Theorem \ref{Thm-Soliton-SNLS} below,
concerning  the large time soliton dynamics of \eqref{equa-X-rough}
in both the $L^2$-subcritical and critical cases.

\begin{theorem} \label{Thm-Soliton-SNLS}
Consider \eqref{equa-X-rough} with $1 <p \leq  1+\frac 4d$, $d\geq 1$.
Let $w_k^0>0$, $\theta_k^0 \in \bbr$,
$x_k^0 \in \bbr^d$,
$v_k \in \bbr^d\setminus\{0\}$, $1\leq k\leq K$, such that
$v_j\not = v_k$ for any $j\not = k$.
Assume $(A0)$ and $(A1)$ with $\upsilon_*$ sufficiently large in {\rm Case (II)}.
Then, for $\bbp$-a.e. $\omega\in \Omega$,
there exists $T_0= T_0(\omega)$ sufficiently large
and $X_*(\omega)\in H^1$,
such that there exists an $H^1$ solution $X(\omega)$ to \eqref{equa-X-rough} on $[T_0,\9)$
satisfying
$X(\omega, T_0) = X_*(\omega)$ and
\begin{align}  \label{X-Rk-asym}
     \|e^{-W_*(t)}X(t) - \sum\limits_{k=1}^K R_k(t) \|_{H^1} \leq C \int_{t}^{\infty}s\phi^{\half}(\delta s)ds, \ \ t\geq T_0.
\end{align}
Here,
\begin{align} \label{W*-def}
  W_*(t,x) =- \sum\limits_{l=1}^N \int_t^\9 i \phi_l(x) g_l(s) d B_l(s),
\end{align}
$\{R_k\}$ are the solitary waves given by \eqref{Rk-def},
$\phi$ is the decay function in \eqref{phi-def}
and $C,\delta>0$.

Moreover, in the $L^2$-subcritical case $1<p<1+\frac 4d$,
there exists a solution $X$ to \eqref{equa-X-rough} on
the whole time regime $[0,\9)$,
satisfying the asymptotic behavior \eqref{X-Rk-asym}.

\end{theorem}

\begin{remark}
$(i)$. To the best of our knowledge,
Theorem \ref{Thm-Soliton-SNLS} provides the first quantitative construction of multi-solitons to \eqref{equa-X-rough} in the stochastic case.
It would be also interesting to note that,
the decay rate in \eqref{X-Rk-asym} can be of either exponential or polynomial type
in {\rm Cases (I)} and {\rm (II)}, respectively,
which reflects the noise effects on the soliton dynamics.

$(ii)$. In comparison with the scattering results in \cite{HRZ18,Z18},
where the solutions behave asymptotically like a free Schr\"odinger flow
in the defocusing (sub)critical cases,
the asymptotics \eqref{X-Rk-asym} in Theorem \ref{Thm-Soliton-SNLS} gives
a different asymptotic behavior
in the focusing case,
namely, the solutions do not scatter at infinity
and may even propagate as any finitely many decoupled solitary waves.
\end{remark}

The strategy of proof is mainly based on the rescaling approach in \cite{HRZ18}
which  relies on two types of Doss-Sussman transformations,
and on the modulation method in \cite{CF20, MM06, MMT06}.

One of the main advantages of Doss-Sussman type transform is,
that the sharper pathwise analysis can be performed to the resulting random solutions,
which  is quite robust in the study of stochastic partial differential equations.
We refer to, e.g.,
\cite{ABD21} for the stochastic Camassa-Holm equation
and
\cite{BDR09} for the stochastic porous media equation.
For the case of SNLS,
we refer to \cite{BF12} for SNLS with potentials multiplied by a temporal real-valued white noise,
\cite{BRZ16.1} for the stochastic logarithmic Schr\"odinger equation.
See also \cite{BRZ18, Z20} for optimal control problems,
\cite{Z18} for the defocusing critical case,
and \cite{FSZ20,RSZ21,SZ19,SZ20} for the construction of (multi-bubble) blow-up solutions.

Here we first apply the Doss-Sussman type transform
\begin{align} \label{resc-t}
   v:= e^{-W}X
\end{align}
with
\begin{align} \label{W-def}
   W(t,x) := \sum_{k=1}^N \int_0^t i \phi_k(x) g_k(s) d B_k(s).
\end{align}
to reduce \eqref{equa-X-rough} to an equation with random lower order perturbations
\begin{equation}    \label{equa-v-RNLS}
\left\{ \begin{aligned}
   & i\partial_t v + (\Delta + b\cdot \na + c) v + |v|^{p-1} v =0, \\
   & v(T_0) = e^{-W(T_0)} X_0,
\end{aligned} \right.
\end{equation}
where the coefficients of low order perturbations
\begin{align}
   b (t,x) =& 2 \na W (t,x)
           = 2 i \sum\limits_{l=1}^N \int_0^t \na \phi_l(x) g_l(s) d B_l(s), \label{b-def} \\
   c (t,x) =& \sum\limits_{j=1}^d (\partial_j W(t,x))^2 + \Delta W (t,x), \nonumber \\
           =& -\sum\limits_{j=1}^d \(\sum\limits_{l=1}^N \int_0^t \partial_j \phi_l(x) g_l(s) d B_l(s) \)^2
              + i \int_0^t \Delta \phi_l(x) g_l(s) d B_l(s). \label{c-def}
\end{align}

It should be mentioned that,
the solvability between two equations via the Doss-Sussman type transform
is indeed nontrivial in infinite dimensional spaces.

The $H^1$ local solvability of equation \eqref{equa-v-RNLS}
can be proved as in  \cite[Theorem 2.1 and Proposition 2.5]{BRZ16},
relying on the Strichartz and local smoothing estimates for
the Laplacian with lower order perturbations, due to Assumption $(A0)$.

Furthermore, the solvability of equation \eqref{equa-X-rough}
can be inherited from that of equation \eqref{equa-v-RNLS}
by Theorem \ref{Thm-Equiv-X-u} below,
which in particular extends the $L^2$-critical result for dimensions $d=1,2$ in \cite{SZ19}
to the whole $L^2$-(sub)critical regime for all dimensions.

\begin{theorem} \label{Thm-Equiv-X-u}
Let $1\leq p\leq 1+\frac 4d$, $d\geq 1$.
Let $v$ be the solution to \eqref{equa-v-RNLS} on $[T_0,\tau^*)$
with $v(T_0) = v_0\in H^1$,
where $T_0,\tau^* \in (0,\9]$ are random variables.
Then, $\bbp$-a.s.,
$X := e^{W} v$ is the solution to equation \eqref{equa-X-rough} on $[T_0,\tau^*)$
in the sense of Definition \ref{X-roughpath-def} above.
\end{theorem}

In order to study the large time behavior of solutions,
we use the ideas from \cite{HRZ18} to perform a second Doss-Sussman type transform.

To be precise, using the theorem on time change for continuous martingales
(cf. \cite{KS91}, Section 3.4) we regard
$ \int_0^t g_l(s) dB_l(s)$ as a
time-changed Brownian motion
$\wt B_l(s(t))$ with $s(t) =\int_0^t g_l^2(r) dr$, $\bbp$-a.s..
Then, by the $L^2$-integrability $g_l\in L^2(\bbr^+)$,
we infer that as time goes to infinity,
$s(t)$ converges to $\int_0^\9 g_l^2(r) dr$,
and thus $\int_0^t g_l(s) dB_l(s) \to \int_0^\9 g_l(s) dB_l(s)$
and $W(t) \to W(\9)$, $\bbp$-a.s.,
$1\leq l\leq N$.

Then,
as in \cite{HRZ18},
we apply a second transformation
\begin{align} \label{u-v-W*-res}
   u=e^{W(\9)}v = e^{-W_*(t)}X(t),
\end{align}
where $W_*$ is given by \eqref{W*-def},
and derive from \eqref{equa-v-RNLS}
a new equation
\begin{equation}   \label{equa-u-low}
\left\{ \begin{aligned}
   & iu_t + (\Delta + b_*\cdot \na + c_*) u + |u|^{p-1} u =0, \\
   & u(T_0) =  e^{-W_*(T_0)}X_0,
\end{aligned} \right.
\end{equation}
where the coefficients of lower order perturbations
\begin{align}
   b_* (t,x) =& 2 \na W_* (t,x)
           = - 2 i \sum\limits_{l=1}^N \int_t^\9 \na \phi_l(x) g_l(s) d B_l(s), \label{b*-def} \\
   c_* (t,x) =& \sum\limits_{j=1}^d (\partial_j W_*(t,x))^2 + \Delta W_* (t,x),  \nonumber  \\
           =& -\sum\limits_{j=1}^d \(\sum\limits_{l=1}^N \int_t^\9 \partial_j \phi_l(x) g_l(s) d B_l(s) \)^2
              - \sum\limits_{l=1}^N \int_t^\9 i \Delta \phi_l(x) g_l(s) d B_l(s). \label{c*-def}
\end{align}

The proof of Theorems \ref{Thm-Soliton-SNLS}
now can be reduced to that of the following result.

\begin{theorem} \label{Thm-Soliton-RNLS}
Consider equation \eqref{equa-u-low} with $1<p\leq 1+ \frac 4d$, $d\geq 1$.
Let $w_k^0>0$, $\theta_k^0 \in \bbr$, $x^0_k\in \bbr^d$, $v_k \in \bbr^d\setminus\{0\}$, $1\leq k\leq K$,
such that $v_j\not = v_k$ for any $j\not = k$.
Assume $(A0)$ and $(A1)$ with $\upsilon_*$ sufficiently large in {\rm Case (II)}.
Then, for $\bbp$-a.e. $\omega\in \Omega$,
there exists $T_0=T_0(\omega)$ large enough and $u_*(\omega) \in H^1$,
such that there exists a unique solution $u(\omega)\in C([T_0, \9); H^1)$
to \eqref{equa-u-low}
satisfying $u(\omega, T_0)= u_*(\omega)$ and
\begin{align}   \label{u-R-asymp}
     \|u(t) - \sum\limits_{k=1}^K R_k(t) \|_{H^1} \leq C \int_{t}^{\infty}s\phi^{\half}(\delta s)ds, \ \ t\geq T_0,
\end{align}
where $C,\delta>0$ and $\{R_k\}$ are the solitary waves given by \eqref{Rk-def}.
\end{theorem}

As mentioned above,
the absence of the pseudo-conformal symmetry
causes a major difficulty in the construction of multi-solitons
in the stochastic case.
Hence, unlike in the deterministic case,
the stochastic multi-solitons to \eqref{equa-u-low}
cannot be obtained from the multi-bubble  blow-up solutions constructed in \cite{RSZ21,SZ20}.

To overcome this problem,
we construct the multi-solitons
in a direct way by using the modulation method
in  \cite{CF20,MM06,MMT06}.

To be precise, we first obtain the geometrical decomposition of solutions
into the soliton profiles rescaled by different parameters and a remainder,
moduling suitable orthogonality conditions
corresponding to the coercivity of linearized operators around the ground state.
Unlike in \cite{CF20,MM06},
the structure of soliton profiles is of a quite unified form in the subcritical and critical cases,
which enables us to treat both cases in a uniform manner.

It is worth noting that,
the unstable direction (i.e., ${\rm Re}\<\wt R_k,\ve\>$ in \eqref{Rkve-sub} below)
is not involved in the
geometrical decomposition,
instead it will be controlled by the
almost conservation law of the local mass.
This enables us to fix the frequency parameter $w_k \equiv w_k^0$
in the geometrical decomposition in the subcritical case,
and  in particular to simplify the proof in the subcritical case.
The other geometrical parameters
will be  controlled by the modulation equations
under the orthogonality conditions.

Concerning the control of the remainder,
the crucial ingredient is the monotonicity of
the Lyapunov type functional adapted to  multi-solitons,
which was first constructed by Martel, Merle and Tsai \cite{MMT06}
in the study of stability problem of multi-solitons.
The analysis of the Lyapunov functional will be based on several controls of the local quantities.
We note that,
these functionals depend on Brownian paths in the stochastic case.
Moreover, the conservation law of energy is also destroyed by the presence of noise.
Again, the rescaling approach enables us to perform the sharp analysis
with Brownian paths fixed,
and thus to obtain the quantitative controls of the variation of functionals
in terms of Brownian paths.

Consequently,
together with the coercivity of linearized operators and bootstrap arguments,
the noise effects on the exponential or polynomial decay rate
of the remainder are derived,
which lead to the desirable stochastic multi-solitons to SNLS \eqref{equa-X-rough}
by using compactness arguments.
\vspace{1ex}

The remainder of this paper is structured as follows.
We first prove Theorem \ref{Thm-Equiv-X-u} in Section \ref{Sec-Rescal-RNLS}
which relates the solvability between two equations \eqref{equa-X-rough} and \eqref{equa-v-RNLS}.
Then, Section \ref{Sec-Geom-dec} contains the geometrical decomposition
and the estimate of modulation equations.
Section \ref{Sec-Local-Wein-cri} is mainly concerned with the control  of
several important functionals, including the local mass, local momentum, energy and the crucial
Lyapunov type functional.
Then, Section \ref{Sec-Proof-Main} is devoted to the proof of the main results,
the crucial ingredients there are the uniform estimate of remainder and modulation parameters,
based on bootstrap arguments,
and the compactness arguments.
At last, Section \ref{Sec-App}, i.e., the Appendix,
contains the coercivity of linearized operators,
the decoupling lemma for solitary waves with distinct velocities
and several technical proofs.

\section{Rescaled random equations} \label{Sec-Rescal-RNLS}

This section is mainly concerned with the proof of Theorem \ref{Thm-Equiv-X-u}
which permits to relate both equations \eqref{equa-X-rough} and \eqref{equa-v-RNLS}.

{\bf Proof of Theorem \ref{Thm-Equiv-X-u}.}
Let us fix any $T\in (T_0,\tau^*)$
and recall that $B_k \in C^\a([T_0,T])$
for any $\a\in (\frac 13, \frac 12)$, $1\leq k\leq N$,
$\bbp$-a.s..
For any $\vf\in C_c^\9$ and any $T_0\leq s<t\leq T$,
\begin{align} \label{Delta-X-st}
     \<\delta X _{st}, \vf\>
      = \<(\del e^W)_{st} u(s), \vf\>
     + \<e^{W(s)} \delta u_{st}, \vf\>
     + \<(\delta e^W)_{st} \delta u_{st}, \vf\>.
\end{align}
Below we treat each term on the R.H.S. above separately.

{\it $(i)$ Estimate of $\<(\del e^W)_{st} u(s), \vf\>$.}
Since $(g_k) \in \mathcal{D}_B^{2\a}([T_0,T]; \bbr^N)$, by \eqref{W-def},
\begin{align} \label{dWst}
   \delta W_{st}
   = \sum\limits_{k=1}^N \int_s^t i \phi_k g_k(r) d B_k(r)
   = \sum\limits_{k=1}^N  i\phi_k g_k(s) \delta B_{k,st}
     + \sum\limits_{j,k=1}^N  i \phi_k g'_{kj}(s) \mathbb{B}_{jk,st}
     +o(t-s).
\end{align}
Then, by Taylor's expansion,
\begin{align}
  (\del e^W)_{st}
  = e^{W(s)} \(\sum\limits_{k=1}^N i\phi_k g_k(s) \delta B_{k,st}
               - \frac 12 \sum\limits_{j,k=1}^N \phi_j\phi_k g_j(s)g_k(s)\delta B_{j,st}\delta B_{k,st}
                  + \sum\limits_{j,k=1}^N i \phi_k g'_{kj}(s) \mathbb{B}_{jk, st} \)
               +o(t-s).
\end{align}
Taking into account (see \cite[Section 3.3]{FH14}, \cite[p.9]{RZZ19})
\begin{align}
   \delta B_{j,st}\delta B_{k,st}
   = \mathbb{B}_{jk,st} + \mathbb{B}_{kj,st}
     + \delta_{jk}(t-s),
\end{align}
we thus obtain
\begin{align}
   (\del e^W)_{st}
   =& e^{W(s)} \bigg(-\mu (t-s)
               + \sum\limits_{k=1}^N i\phi_k g_k(s) \delta B_{k,st} \nonumber \\
    & \qquad  + \sum\limits_{j,k=1}^N \(-\phi_j\phi_k g_j(s)g_k(s)
                + i \phi_k g'_{kj}(s) \) \mathbb{B}_{jk,st} \bigg)
               +o(t-s),
\end{align}
which yields that
\begin{align} \label{deltaeW-st}
   \<(\del e^W)_{st}u(s), \vf\>
   =&   \<- \mu (e^{W(s)} u(s)), \vf\>  (t-s)
      + \sum\limits_{k=1}^N \<i \phi_k g_k(s) (e^{W(s)} u(s)), \vf\> \delta B_{k,st} \nonumber \\
    & + \sum\limits_{j,k=1}^N  \<\( -\phi_j\phi_k g_j(s) g_k(s) + i\phi_k g'_{kj}(s) \) (e^{W(s)} u(s)), \vf\> \mathbb{B}_{jk,st}
      + o(t-s).
\end{align}

{\it $(ii)$ Estimate of $\<e^{W(s)} \delta u_{st}, \vf\>$.}
Let $f(u):= |u|^{p-1} u$.
We claim that
\begin{align} \label{deltay-st}
  \<e^{W(s)}\delta u_{st}, \vf \>
  =& \< i \Delta (e^{W(s)} u(s)), \vf\> (t-s)
     + \<i f(e^{W(s)}u(s)), \vf\> (t-s)
     + o(t-s).
\end{align}

In order to prove \eqref{deltay-st},
using equation \eqref{equa-v-RNLS} we have
\begin{align} \label{K1-K2}
  \<e^{W(s)}\delta u_{st}, \vf \>
  =& \<e^{W(s)} \int_s^t i e^{-W(r)}\Delta (e^{W(r)} u(r)) dr, \vf\>
     + \<e^{W(s)} \int_s^t i f(u(r))dr, \vf\>  \nonumber \\
  =:& K_1 + K_2.
\end{align}

Note that,
\begin{align} \label{K1-K11-K12}
   K_1 =& \<i\Delta (e^{W(s)}u(s)), \vf\> (t-s)
          +\int_s^t \<u(r)-u(s) , (-i) \Delta  (e^{-W(s)}\vf)\> dr   \nonumber \\
        & + \int_s^t \< L(r)u(r) - L(s)u(s), (-i)e^{-W(s)}\vf\> dr  \nonumber \\
       =:& \<i\Delta (e^{W(s)}u(s)), \vf\> (t-s)
          + K_{11} + K_{12},
\end{align}
where $L(r)u(r) := (b(r)\cdot \na + c(r)) u(r)$,
and $L(s)u(s)$ is defined similarly.

By the integration by parts formula,
\begin{align} \label{K11-esti}
   K_{11}
   =& \int_s^t \<e^{-ir\Delta} u(r) - e^{-is\Delta} u(s) , (-i)e^{-is\Delta} \Delta (e^{-W(s)}\vf ) \> dr \nonumber \\
    &  + \int_s^t \<u(r),(-i) (1-e^{i(r-s)\Delta})  \Delta (e^{-W(s)}\vf ) \> dr \nonumber \\
   \leq& \int_s^t \|e^{-ir\Delta} u(r) - e^{-is\Delta} u(s)\|_{L^2} \| \Delta (e^{-W(s)}\vf) \|_{L^2} dr \nonumber \\
    &  + \int_s^t \|u(r)\|_{L^2} \|(1-e^{i(r-s)\Delta})  \Delta (e^{-W(s)}\vf)\|_{L^2} dr.
\end{align}

We claim that, there exists $\zeta>0$ such that,
\begin{align} \label{ut-us-t}
   \|e^{-ir\Delta} u(r) - e^{-is\Delta} u(s)\|_{L^2}
   \leq C (r-s)^{\zeta}.
\end{align}
To this end, by equation \eqref{equa-v-RNLS},
\begin{align} \label{erD-fu-bc}
   \|e^{-ir\Delta} u(r) - e^{-is\Delta} u(s)\|_{L^2}
   \leq \bigg\|\int_s^r e^{-is'\Delta} \( f(u(s')) + b(s')\cdot \na u(s') + c(s')u(s') \) d s' \bigg\|_{L^2}.
\end{align}
Applying Strichartz estimate with the Strichartz pair $(p+1, q)$,
$q=\frac{4(p+1)}{d(p-1)}$, we get
\begin{align}
   \bigg\|\int_s^r e^{-is'\Delta} f(u(s')) ds' \bigg\|_{L^2}
   \leq C \|f(u)\|_{L^{q'}(s,r;L^{\frac{p+1}{p}})}
   \leq C (r-s)^{1-\frac{d(p-1)}{4}} \|u\|_{L^q(s,r; L^{p+1})}^{p},
\end{align}
which, via Sobolev's embedding, yields that
\begin{align}  \label{esD-fu-esti}
    \bigg\|\int_s^r e^{-is'\Delta} f(u(s')) ds' \bigg\|_{L^2}
    \leq C (r-s)^{1-\frac{d(p-1)}{4} + \frac{p}{q}} \|u\|_{C([s,r]; H^1)}^{p}.
\end{align}
Moreover, we have
\begin{align} \label{esD-bc-esti}
    \bigg\|\int_s^r e^{-is'\Delta}  \(b(s')\cdot \na u(s') + c(s')u(s')\) ds' \bigg\|_{L^2}
   \leq C (r-s) \|u\|_{C([s,r]; H^1)}.
\end{align}
Hence, plugging \eqref{esD-fu-esti} and \eqref{esD-bc-esti} into \eqref{erD-fu-bc}
we obtain \eqref{ut-us-t}, as claimed.

Thus, using \eqref{ut-us-t} and the estimate that for any multi-index $\upsilon$,
\begin{align} \label{eirsD-t}
   \|(1-e^{i(r-s)\Delta})  \partial^\upsilon (e^{-W(s)}\vf)\|_{L^2}
   \leq C (r-s) \|e^{-W(s)}\vf\|_{H^{2+|\upsilon|}}
   \leq C(T,\upsilon) (r-s),
\end{align}
we derive from \eqref{K11-esti} that
\begin{align} \label{K11-st}
    K_{11} \leq&  C(T) \int_s^t (r-s)^\zeta + (r-s) dr = o(t-s).
\end{align}

Similarly,  we compute
\begin{align*}
  K_{12}
  =& \int_s^t \<e^{-ir\Delta} u(r) - e^{-is\Delta} u(s), (-i) e^{-is\Delta} L^*(s) (e^{-W(s)}\vf) \> dr  \nonumber \\
   & + \int_s^t \<u(r), (-i) (1-e^{i(r-s)\Delta})  L^*(s)  (e^{-W(s)}\vf) \>dr \nonumber \\
   & + \int_s^t \<(L(r)-L(s))u(r), (-i) e^{-W(s)} \vf\> dr,
\end{align*}
where $L^*(s)$ is the adjoint operator of $L(s)$.
Since
$$\|(L(r)-L(s))u(r)\|_{L^2} \leq C(T) \max_{1\leq l\leq N}|B_l(r)-B_l(s)| \leq C(T)(r-s)^{\a},$$
using \eqref{ut-us-t} and \eqref{eirsD-t}
we get
\begin{align} \label{K12-st}
   K_{12}
   \leq C(T) \int_s^t   (r-s)^{\zeta} + (r-s) + (r-s)^{\a} dr
   =  o(t-s).
\end{align}
Thus, plugging \eqref{K11-st} and \eqref{K12-st} into \eqref{K1-K11-K12}
we conclude that
\begin{align} \label{K1-st}
   K_1  =  \<i\Delta (e^{W(s)}u(s)), \vf\> (t-s)
          + o(t-s).
\end{align}

Regarding the second term $K_2$ in \eqref{K1-K2},
we see that
\begin{align} \label{K2.0}
   K_2 = \<i f(e^{W(s)}u(s)), \vf\> (t-s)
         + \int_s^t \<f(u(r))- f(u(s)), (-i) e^{-W(s)} \vf\> dr.
\end{align}
Since
\begin{align}
  |f(u(r))- f(u(s))| \leq C (|u(r)|^{p-1} + |u(s)|^{p-1}) |u(r) - u(s)|,
\end{align}
Sobolev's embedding $H^1\hookrightarrow L^{p+1}$ yields that
\begin{align*}
   \big | \<f(u(r))- f(u(s)), (-i) e^{-W(s)} \vf \>  \big|
   \leq& \| e^{-W} \vf\|_{C([s,t];L^{p+1})}
       \|f(u(r))- f(u(s))\|_{L^{\frac{p+1}{p}}} \nonumber \\
   \leq&  \| e^{-W} \vf\|_{C([s,t];L^{p+1})}
         \( \|u(r)\|_{L^{p+1}} + \|u(s)\|_{L^{p+1}} \)
         \|u(r) - u(s)\|_{L^{p+1}} \nonumber \\
   \leq&  C \| e^{-W} \vf\|_{C([s,t];L^{p+1})}  \|u\|_{C([s,t];H^1)}
          \|u(r) - u(s)\|_{H^1},
\end{align*}
which yields that
\begin{align}
       & \bigg|\int_s^t \<f(u(r))- f(u(s)), (-i) e^{-W(s)} \vf\> dr \bigg|  \nonumber \\
   \leq& C \|e^{-W} \vf\|_{C([T_0,T];L^{p+1})}
         \|u\|^{p-1}_{C([T_0,T]; H^1)} \sup\limits_{s\leq r\leq t} \|u(r)-u(s)\|_{H^1} (t-s)
   =  o(t-s),
\end{align}
where in the last step we  used the fact that
$\sup_{s\leq r\leq t} \|u(r)-u(s)\|_{H^1} =o(1)$
as $t \to s$, due to the continuity of $u$ in $H^1$.
Thus,
we obtain
\begin{align} \label{K2-st}
   K_2 = \<i f(e^{W(s)}u(s)), \vf\> (t-s)
          + o(t-s).
\end{align}

Therefore,
plugging \eqref{K1-st} and \eqref{K2-st} into \eqref{K1-K2}
we obtain \eqref{deltay-st}, as claimed.

{\it $(iii)$ Estimate of $\<(\delta e^W)_{st} \delta u_{st},\vf\>$.}
By the integration by parts formula and H\"older's inequality,
\begin{align}
     \<(\delta e^W)_{st} \delta u_{st}, \vf\>
   =& \int_s^t \<u(r), (-i)e^{-W(r)}\Delta\(e^{W(r)} \ol{(\delta e^W)_{st}} \vf\)\>
      +  \<f(u(r)), (-i) \ol{(\delta e^W)_{st}} \vf\> dr  \nonumber \\
   \leq& \int_s^t \|u(r)\|_{L^2} \|\Delta\(e^{W(r)} \ol{(\delta e^W)_{st}} \vf\) \|_{L^2}
     + \|u(r)\|^{p}_{L^{\rho p}} \|(\delta e^W)_{st}\|_{L^\9} \|\vf \|_{L^{\rho'}} dr,
\end{align}
where $\rho \in (1,\9)$ is taken such that $2\leq \rho p \leq 2+\frac{4}{d-2}$ if $d\geq 3$,
$2\leq \rho p <\9$ if $d=1,2$,
$\frac{1}{\rho}+\frac{1}{\rho'}=1$.
Since for any multi-index $\upsilon$,
\begin{align*}
   \|\partial_x^\upsilon (\delta e^W)_{st} \|_{L^\9} \leq C(T,\a) (t-s)^{\a},
\end{align*}
using Sobolev's embedding $H^1 \hookrightarrow L^{\rho p}$,
we obtain
\begin{align}  \label{deltaeW-deltay-st}
       \<(\delta e^W)_{st} \delta u_{st}, \vf\>
  \leq& C \int_s^t \|u(r)\|_{L^2} + \|u(r)\|_{H^1}^{p} dr (t-s)^\a \nonumber \\
  \leq& C (1+ \|u\|_{C([s,T]; H^1)}^{p} ) (t-s)^\a
  = o(t-s).
\end{align}

Now, plugging \eqref{deltaeW-st}, \eqref{deltay-st} and \eqref{deltaeW-deltay-st} into \eqref{Delta-X-st}
and using $X=e^W u$
we obtain
\begin{align}  \label{deltaX-vf-Ca}
   \<\delta X_{st}, \vf\>
   =& \<i \Delta X(s) + i f(X(s)) - \mu X(s), \vf\>(t-s)
      + \sum\limits_{k=1}^N \<i\phi_k g_k(s) X(s), \vf\> \delta B_{k,st}  \nonumber\\
    &  + \sum\limits_{j,k=1}^N  \<-\phi_j\phi_k g_j(s) g_k(s) X(s) + i\phi_k g'_{kj}(s) X(s), \vf\> \mathbb{B}_{jk,st}
      + o(t-s).
\end{align}
In particular,
this yields that for any $\vf \in C_c^\9$,
\begin{align} \label{X-vf-Ca}
   \<X, \vf\> \in C^\a([T_0,T], \bbr).
\end{align}

Let $Y:=(Y_k)$ with $Y_k := \<i \phi_k g_k X, \vf\> $.
We claim that
\begin{align} \label{Yk-vf-Ca}
   Y  \in \mathcal{D}^{2\a}_{B}([T_0,T]; \bbr^N),
\end{align}
with the Gubinelli derivative
\begin{align} \label{Y'kj-vf-Gubi}
    Y'_{kj} =  \<-\phi_j \phi_k g_j g_k X + i\phi_k g'_{kj} X, \vf\>  \in C^\a ([T_0,T]; \bbr).
\end{align}

To this end,
using \eqref{X-vf-Ca} and the fact that $g_k\in  C^\a ([T_0,T]; \bbr)$,
we have
$$Y_k = g_k \<X, -i \phi_k \vf\>  \in C^\a ([T_0,T]; \bbr).$$
Moreover, note that
\begin{align} \label{deltaYk-expan}
  \delta Y_{k,st}
  = \delta g_{k,st} \<X(s), -i\phi_k\vf\>
    + g_k(s) \< \delta X_{st}, -i\phi_k\vf\>.
\end{align}
Since $(g_k) \in \mathcal{D}^{2\a}_{B}([T_0,T]; \bbr^N)$, we have
\begin{align} \label{gkt-gks}
   \delta g_{k,st} = \sum\limits_{j=1}^N g'_{kj}(s) \delta B_{j,st} +  \calo((t-s)^{2\a}).
\end{align}
It also follows from \eqref{deltaX-vf-Ca} that
\begin{align}  \label{deltaX-phikvf}
   \<\delta X_{st}, -i \phi_k \vf\>
   = - \sum\limits_{j=1}^N \<\phi_j g_j(s) X(s), \phi_k \vf\> \delta B_{j,st}
     + \calo((t-s)^{2\a}).
\end{align}
Plugging \eqref{gkt-gks} and \eqref{deltaX-phikvf} into \eqref{deltaYk-expan}
we obtain
\begin{align} \label{deltaYst}
   \delta Y_{k,st}
   = \sum\limits_{j=1}^N \<-\phi_j\phi_k g_j(s) g_k(s) X(s) + i \phi_k g'_{kj}(s) X(s), \vf\> \delta B_{j,st}
      + \calo((t-s)^{2\a}),
\end{align}
which yields \eqref{Yk-vf-Ca} and \eqref{Y'kj-vf-Gubi}, as claimed.

Thus,
we conclude from \eqref{X-vf-Ca}, \eqref{Yk-vf-Ca} and \eqref{Y'kj-vf-Gubi}
that
$\<X, \vf\> \in \mathcal{D}_{B}^{2\a}([T_0,T]; \bbr)$
with the Gubinelli derivative
$\<i\phi_k g_k X,\vf\>$,
$\a \in (\frac 13, \frac 12)$,
$X:=e^W u$ satisfies equation \eqref{equa-X-rough} in the sense of Definition \ref{X-roughpath-def}
and \eqref{phikX-st} follows from \eqref{deltaYst}.
Therefore, the proof is complete.
\hfill $\square$

\section{Geometrical decomposition} \label{Sec-Geom-dec}

This section mainly treats the geometrical decomposition of solutions to
equation
\begin{equation}  \label{equa-un-tn}
\left\{ \begin{aligned}
 &i\partial_tu+\Delta u+|u|^{p-1}u+b_* \cdot \nabla u+c_* u=0,   \\
 &u(T)=R(T),
\end{aligned}  \right.
\end{equation}
where $R=\sum_{k=1}^{K}R_k$,
$\{R_k\}$ are given by \eqref{Rk-def} and $T>0$ is sufficiently large.
We mainly focus on the critical case, i.e., $p=1+\frac 4d$,
as the subcritical case is easier and can be proved similarly.

For convenience,
we set
$\calp_k:= (\al_k,\t_k,w_k) \in  \mathbb{X}:= \bbr^{d} \times \bbr \times \bbr$,
$1\leq k\leq K$, and
$\calp :=(\calp_1, \cdots, \calp_K)
\in \mathbb{X}^K$.
For simplicity of exposition, we will omit the dependence on $\omega (\in \Omega)$.

\subsection{Critical case}  \label{Subsec-Geom-cri}

\begin{proposition} (Geometrical decomposition)  \label{Prop-dec-un-cri}
Assume that $u$ solves \eqref{equa-un-tn} with $p=1+\frac4d$, $d\geq 1$.
For any $T$  sufficiently large,
there exist  $0\leq T^*<T$ and unique modulation parameters
$\mathcal{P}\in C^1([T^*, T]; \mathbb{X}^K)$,
such that $u$ admits the geometrical decomposition
\be\label{geo-dec-cri}
    u(t,x)
    =\sum_{k=1}^{K} \wt R_{k}(t,x) + \ve(t,x)\ \(=: \wt R(t,x) + \ve(t,x)\),
\ee
with the modulation parameters
$\calp_k:=(\a_k, \theta_k, w_k) \in \mathbb{X}$
and
\be \label{Rk-Qk-cri}
    \wt R_{k}(t,x) := Q_{w_k(t)} \(x-v_k t-\alpha_{k}(t) \)e^{i\(\half v_k\cdot x-\frac{1}{4}|v_k|^2t+(w_{k}^0)^{-2}t+\t_{k}(t)\)},
\ee
satisfying
\begin{align}
     \ve(T) =0,\ \ \calp_k(T)=(x_k^0, \t^0_k,w_k^0).
\end{align}
Moreover, the following orthogonality conditions hold on $[T^*,T] $:
for every $1\leq k\leq K$,
\be\ba\label{ortho-cond-R-ve-cri}
 & {\rm Re}\int \nabla  \wt R_{k}(t) \ol{\ve}(t)dx=0,\ \
   {\rm Im} \int  \wt R_{k}(t) \ol{\ve}(t)dx=0,\\
 & {\rm Re}\int \(\Lambda_k  \wt R_k(t)-\frac{i}{2}v_k\cdot y_k(t)  \wt R_k(t)\)\bar{\ve}(t)dx=0,
\ea\ee
where
\begin{align} \label{Lambdak-def}
\Lambda_k:= \frac{2}{p-1} I_d + y_k \cdot \na,\ \ with\  y_k(t)=x-v_kt-\alpha_k(t).
\end{align}
\end{proposition}

\begin{remark}
$(i)$. The orthogonality conditions in \eqref{ortho-cond-R-ve-cri}
correspond to the coercivity of linearized operators around the ground state
in Lemma \ref{Lem-coerc-L}.
The only one remaining unstable direction ${\rm Re}\<\wt R_k,\ve\>$
will be controlled by the almost conservation of the local mass
in Corollary \ref{Cor-Rkve-cri} below.

$(ii)$.
We note that, the frequency $w_k^0$ in the phase of $\wt R_k$ is fixed,
but the frequency parameter $w_k(t)$ in $Q_{w_k(t)}$ varies with time.
In Proposition \ref{Prop-dec-un-sub} below
we are also able to fix the frequency $w_k^0$ in $Q_{w_k^0}$ in the subcritical case.
This is possible because
the linearized operators in the subcritical case have one less unstable direction
than those in the critical case.
\end{remark}

The proof of Proposition \ref{Prop-dec-un-cri} is based on
the implicit function theorem.
See, e.g., \cite{MMT06}.
For the reader's convenience,
we present the proof in the Appendix in the fashion close to that of \cite{CSZ21}.

In the sequel,
we set
$B_{*,l}(t):=  \int_t^\9   g_l(s) d B_l(s)$,
$B_*(t)= \sup_{t\leq s<\9} \sum_{l=1}^N |B_{*,l}(s)|$.
Since $g_l\in L^2(\bbr^+)$, we have
\begin{align} \label{B*-o1}
    \lim\limits_{t\to \9} B_*(t) = 0, \ \ \bbp-a.s..
\end{align}
In particular, for $\bbp$-a.e. $\omega\in \Omega$,
we may take a large (random) time $T_*=T_*(\omega)>0$ such that
\begin{align}  \label{B*-T*-1}
   \sup\limits_{t\geq T_*} B_*(t) \leq 1.
\end{align}
We also consider $T^*\geq T_*(\omega)$ sufficiently large such that for any $t\in [T^*,T]$,
\begin{align}  \label{ve-T*-1}
  \sup_{T^*\leq t\leq T}\|\ve(t)\|_{H^1}<1,
\end{align}
and
\begin{align} \label{B*-wk-ak-deter-bdd}
   B_*(t) + |w_k(t) - w^0_k| + |\a_k(t) - x_k^0|
   \leq \frac {1}{10} \min\{1, w^0_k,  x_k^0 \},
\end{align}
where $1\leq k\leq K$.
Hence, $B_*$, $w_k$, $w_k^{-1}$ and $|\a_k|$ are bounded by a deterministic constant
on $[T^*,T]$.

Next, the dynamic of geometric parameters are controlled by the  modulation equation below.

\begin{proposition} (Control of modulation equations) \label{Prop-Mod-cri}
Define the modulation equations by
\begin{align}
   Mod_k(t):= |\dot{w}_k(t)|+|\dot{\alpha}_{k}(t)|+|\dot{\t}_{k}(t)-(w^{-2}_k(t)-(w_{k}^0)^{-2})|,
\end{align}
where $1\leq k\leq K$,
and set $Mod := \sum_{k=1}^K Mod_k$.
Then, there exist deterministic constants $C,\delta_1, \delta_2>0$
such that
for $T$ large enough and $T^*$ close to $T$
\begin{align} \label{Mod-esti-cri}
Mod (t) \leq C(  \|\ve(t)\|_{H^1} +  B_*(t)\phi(\delta_1 t)  + e^{-\delta_2 t}),\ \ \forall t\in [T^*,T].
\end{align}
\end{proposition}

{\bf Proof.}
For simplicity we set the phase function
\begin{align} \label{phase-fct}
\Phi_k(t,x):= \half v_k\cdot x-\frac{1}{4}|v_k|^2t+(w_{k}^0)^{-2}t+\t_{k}(t),
\end{align}
where $1\leq k\leq K$.
Using the explicit formula \eqref{geo-dec-cri} we compute
\begin{align} \label{dt-R-cri}
i\partial_t \wt R_k(t,x)=&\(\frac{|v_k|^2}{4} - (w_k^0)^{-2} -\dot{\t}_k(t)\) \wt R_k(t,x)
      -i(\dot{\alpha}_k(t)+v_k)\cdot\nabla Q_{w_k(t)} (x-v_kt-\a_k)
  e^{i\Phi_k(t,x)} \nonumber\\
  &-i\frac{\dot{w}_k(t)}{w_k(t)} \Lambda_k Q_{w_k(t)} (x-v_kt-\a_k)
  e^{i \Phi_k(t,x)},
\end{align}
and
\begin{align}
& \nabla  \wt R_k(t,x) - \frac{i}{2}v_k \wt R_k(t,x) =\nabla Q_{w_k(t)} (x-v_kt-\a_k) e^{i\Phi_k(t,x)}, \label{nablaR}\\
& \Delta  \wt R_k(t,x) = \(\Delta Q_{w_k(t)}+iv_k\cdot \nabla Q_{w_k(t)} -\frac{|v_k|^2}{4}Q_{w_k(t)}\)(x-v_kt-\a_k)
               e^{i\Phi_k(t,x)}.\label{lambdaR}
\end{align}
Then, it follows from \eqref{equa-Qw}, \eqref{dt-R-cri} and \eqref{lambdaR} that
\begin{align} \label{equa-Rk-cri}
  & i\partial_t \wt R_k(t,x)
    +\Delta  \wt R_k(t,x)
    +| \wt R_k(t,x)|^{p-1}  \wt R_k(t,x) \nonumber \\
= & \(-i\frac{\dot{w_k}(t)}{w_k(t)} \Lambda_k Q_{w_k(t)}
   -i\dot{\alpha}_k(t) \nabla Q_{w_k(t)}
  - \(\dot{\t}_k(t) - (w_k^{-2}(t)- (w_{k}^0)^{-2})\) Q_{w_k(t)}\)(x-v_kt-\a_k)
    e^{i \Phi_k(t,x)}.
\end{align}

Moreover, set
\begin{align}
  & H_1:= -\sum\limits_{j\not = k} \( i\frac{\dot{w_j}}{w_j} \Lambda_j Q_{w_j}
       + i\dot{\alpha}_j \nabla Q_{w_j} +  \(\dot{\t}_j - (w_{j}^{-2}- (w_{j}^0)^{-2})\) Q_{w_j}\)(x-v_kt-\a_k) e^{i\Phi_j}, \label{H1-def} \\
  & H_2:=| \wt R|^{p-1}  \wt R-\sum_{k=1}^{K}| \wt R_k|^{p-1}  \wt R_k,  \label{H2-def} \\
  & H_3:=| \wt R+\ve|^{p-1}( \wt R+\ve)-| \wt R|^{p-1} \wt  R.   \label{H3-def}
\end{align}

It then follows from equation \eqref{equa-un-tn}, \eqref{geo-dec-cri}
and \eqref{dt-R-cri} that
\begin{align} \label{equa-ve-cri}
     & i\partial_t\ve(t,x)+\Delta\ve(t,x) \nonumber \\
       & -   \( i\frac{\dot{w_k}(t)}{w_k(t)} \Lambda_k Q_{w_k(t)}
    + i\dot{\alpha}_k(t) \nabla Q_{w_k(t)} +  \(\dot{\t}_k(t) - (w_{k}^{-2}(t)- (w_{k}^0)^{-2})\) Q_{w_k(t)}\)
        (x-v_kt-\a_k) e^{i \Phi_k(t,x)}     \nonumber \\
    =& -H_1(t,x)-H_2 (t,x)- H_3(t,x)
      - b_*(t,x)\cdot(\nabla  \wt R(t,x)+\nabla \ve(t,x)) - c_*( \wt R(t,x)+\ve(t,x)).
\end{align}

We are now in position to derive the estimates of modulation equations.

$(i)$  {Estimate of $\dot{\a}_k$}.
Taking the inner product of \eqref{equa-ve-cri} with
$\nabla  \wt R_k-\frac{i}{2}v_k \wt R_k$,
then taking the imaginary part and using \eqref{nablaR}
we get
\begin{align} \label{ptve-nablaRk-inner}
&{\rm Re}\langle\partial_t\ve, \nabla  \wt R_k-\frac{i}{2}v_k \wt R_k\rangle
   +{\rm Im}\langle\Delta \ve,\nabla  \wt R_k-\frac{i}{2}v_k \wt R_k\rangle
    -   \frac{\dot{w_k}}{w_k}
   {\rm Re}\langle \Lambda  Q_{w_k} ,\nabla Q_{w_k}  \rangle
    - \dot{\alpha}_k \|\nabla Q_{\omega_k}\|_{L^2}^2 \nonumber \\
  &- \(\dot{\t}_k - (w_k^{-2} - (w_{k}^0)^{-2})\)
   {\rm Im} \langle  Q_{w_k}   ,\nabla Q_{w_k} \rangle  \nonumber \\
=&  -{\rm Im}\langle H_1 +  H_2+ H_3,\nabla  \wt R_k-\frac{i}{2}v_k \wt R_k\rangle
  - {\rm Im}\langle b_* \cdot (\nabla  \wt R + \na \ve)+c_* ( \wt R+\ve), \nabla  \wt R_k-\frac{i}{2}v_k \wt R_k\rangle.
\end{align}

For the L.H.S. of \eqref{ptve-nablaRk-inner},
by \eqref{ortho-cond-R-ve-cri}, \eqref{nablaR} and \eqref{equa-Rk-cri},
\begin{align} \label{ptve-nablaRj-sub}
{\rm Re}\langle\partial_t\ve, \nabla  \wt R_k-\frac{i}{2}v_k \wt R_k\rangle
=  {\rm Re}\langle\ve, \partial_t
    \(\nabla Q_{w_k} (x-v_kt-\a_k) e^{i\Phi_k}\) \rangle
=  \calo( Mod_k +1)
\|\ve\|_{L^2},
\end{align}
and
\begin{align} \label{Dve-naRk}
   {\rm Im} \<\Delta \ve, \na  \wt R_k - \frac i2 v_k  \wt R_k \>
   = {\rm Im} \< \ve, \Delta(\na Q_{w_k}(x-v_kt-\a_k) e^{i\Phi_k})\>
   = \calo(\|\ve\|_{L^2}).
\end{align}

Moreover, by  the radial symmetry of $Q_{w_k}$,
\begin{align} \label{LambdaQ-naQ-0}
   \< \Lambda Q_{w_k}, \nabla Q_{w_k} \> =0,\ \
   \<Q_{w_k}, \na Q_{w_k}\> =0.
\end{align}

Regarding the R.H.S. of \eqref{ptve-nablaRk-inner},
we claim that there exist deterministic constants $C, \delta_1, \delta_2 >0$
such that
\begin{align} \label{ptve-naRk-inner-RHS}
   | {\rm R.H.S.\ of}\ \eqref{ptve-nablaRk-inner}|
   \leq C(\|\ve(t)\|_{H^1} + B_*(t)\phi(\delta_1 t) + (Mod(t)+1)e^{-\delta_2 t}).
\end{align}

In order to prove \eqref{ptve-naRk-inner-RHS},
we use Lemma \ref{Lem-decoup} to derive that
\begin{align}  \label{H1-Rj-cri}
  |\langle H_1(t), \na  \wt R_k(t) - \frac{i}{2} v_k  \wt R_k(t) \rangle|
   + |\langle H_2(t), \na  \wt R_k(t) - \frac{i}{2} v_k \wt  R_k(t) \rangle|
  \leq C (Mod(t)+1) e^{-\delta t}.
\end{align}

Moreover, since
$p\leq 1+\frac 4d$, we may take $\rho\geq 1$ close to $1$ such that
$\rho p \leq \frac{2d}{d-2}$.
Taking into account
\begin{align}
    |H_3| \leq C (| \wt R|^{p-1} + |\ve|^{p-1})|\ve|,
\end{align}
and Gagliardo-Nirenberg's inequality
we get
\begin{align}   \label{H2-Rj-cri}
   |\<H_3,  \wt R_k\>|
   \leq& \int (| \wt R|^{p-1} + |\ve|^{p-1})|\ve| |\na \wt  R_k - \frac{i}{2} v_k  \wt R_k| dx \nonumber \\
   \leq&  \sum\limits_{k=1}^K \| \wt R^{p-1}(\na  \wt R_k - \frac{i}{2} v_k  \wt R_k)\|_{L^{2}} \|\ve\|_{L^2}
          + \|\na  \wt R_k - \frac{i}{2} v_k  \wt R_k\|_{L^{\rho'}}\|\ve\|_{L^{\rho p}}^{p} \nonumber \\
   \leq& C\(\|\ve\|_{L^2} + \|\ve\|_{H^1}^{p}\)
   \leq C \|\ve\|_{H^1}.
\end{align}

Concerning the lower order perturbations,
applying Lemma \ref{Lem-decoup} again we have
\begin{align} \label{bnaR-Rj-cri}
 {\rm Re}\langle b_*\cdot\nabla \wt  R+c_* \wt R, \na  \wt R_k - \frac{i}{2} v_k  \wt R_k \rangle
&={\rm Re}\langle b_*\cdot\nabla  \wt R_k+c_* \wt R_k, \na  \wt R_k - \frac{i}{2} v_k  \wt R_k \rangle + \calo(e^{-\delta t}),
\end{align}
where the implicit constant is independent of $\omega$,
due to \eqref{B*-wk-ak-deter-bdd}.
Note that,
by \eqref{b-def}, \eqref{nablaR} and the change of variables,
\begin{align}  \label{bnaR-R-esti}
    & |{\rm Re} \<b_*\cdot \na \wt  R_k, \wt  R_k\>|
   = 2 \bigg|\sum\limits_{l=1}^N B_{*,l} {\rm Im} \<\na \phi_l \cdot \na  \wt R_k, \na  \wt R_k - \frac{i}{2} v_k  \wt R_k\> \bigg| \nonumber \\
   \leq& C  B_*  \sum\limits_{l=1}^N \( |v_k| \int |\na \phi_l(y+v_kt+\a_k)| |Q_{w_k}\na Q_{w_k}(y)| dy
          + \int |\na \phi_l(y+v_kt+\a_k)||\na Q_{w_k}(y)|^2 dy \).
\end{align}
Since by \eqref{B*-wk-ak-deter-bdd},
$|\a_k| \leq 2 |x_k^0|$,
and for $|y|\leq \frac{|v_k| t}{2w_k}$ and $t$ large enough such that $t\geq \frac{8|x_k^0|}{|w_k|}$,
$|y + v_k t +\a_k| \geq \frac 12 |v_k| t - |\a_k| \geq \frac 14 |v_k|t$.
Then, by the exponential decay of $Q$,
the lower bound $\inf_{t}w_{k} >0$, $\min_{1\leq k\leq K} |v_k|>0$
and the decay conditions in Assumption $(A1)$,
the first integration on the R.H.S. above can be bounded by
\begin{align}  \label{naphinaQ-esti}
  &  CB^*\int_{|y|\leq\frac{|v_k| t}{2}}|\nabla\phi_l(y+v_kt+\alpha_k)||Q_{w_k}\na Q_{w_k}(y)| dy
+C\int_{|y|\geq\frac{|v_k|t}{2}} | Q_{w_k}\na Q_{w_k}(y)|dy \nonumber  \\
\leq& CB^* \( \phi(\frac 14 |v_k|t) \int |\na Q_{w_k} Q_{w_k}(y)| dy +e^{-\delta t} \)   \nonumber \\
 \leq& CB^*\( \phi(\frac 14 |v_k|t)+ e^{-\delta t}\).
\end{align}
Similarly, we have that for some $\delta>0$,
\begin{align} \label{naphiRk-esti}
     \int |\na \phi_l(y+v_kt+\a_k)||\na Q_{w_k}(y)|^2 dy
     \leq C \( \phi(\frac 14 |v_k|t) + e^{-\delta t}\).
\end{align}
Hence, plugging \eqref{naphinaQ-esti} and \eqref{naphiRk-esti} into \eqref{bnaR-R-esti}
we obtain that
\begin{align} \label{b*naRk-bdd}
    |{\rm Re} \<b_* (t) \cdot \na  \wt R_k(t), \na  \wt R_k(t) - \frac{i}{2} v_k \wt  R_k(t)\>|
    \leq CB^*(t)  \( \phi(\delta_1 t)  + e^{-\delta_2 t}\),
\end{align}
where $C,\delta_1, \delta_2> 0$ are universal deterministic constants.

Moreover, by \eqref{c*-def} and analogous arguments,
\begin{align} \label{c*naRk-bdd}
    & |{\rm Re} \<c_*  \wt R_k, \na \wt  R_k - \frac{i}{2} v_k  \wt R_k\>|   \nonumber \\
   =& \bigg|{\rm Re}\<\sum\limits_{j=1}^d
      \(\sum\limits_{l=1}^N B_{*,l} \partial_j \phi_l \)^2 \wt  R_k
      - i \sum\limits_{l=1}^N \Delta \phi_l B_{*,l} \wt R_k, \na  \wt R_k - \frac{i}{2} v_k  \wt R_k\> \bigg|   \nonumber \\
   \leq& C \sum\limits_{j=1}^d \sum\limits_{l=1}^N (B_{*,l}+B_{*,l}^2)
        \int (|\partial_j \phi_l(y+v_kt+\a_k)|^2 + |\Delta\phi_l(y+v_kt+\a_k)|)
        |Q_{w_k}\na Q_{w_k}(y)| dy \nonumber \\
   \leq& C \sum\limits_{j=1}^d \sum\limits_{l=1}^N  B_{*,l}
           \(\int\limits_{|y|\leq \frac{|v_k|t}{2}}
           (|\partial_j \phi_l(y+v_kt+\a_k)|^2 + |\Delta\phi_l(y+v_kt+\a_k)|)  |Q_{w_k}\na Q_{w_k}(y)|dy + e^{-\delta t}\) \nonumber \\
   \leq& C B_*  \( \phi(\delta_1 t)  + e^{-\delta_2 t} \).
\end{align}
Hence, plugging \eqref{b*naRk-bdd} and \eqref{c*naRk-bdd}  into \eqref{bnaR-Rj-cri}
we obtain
\begin{align} \label{bnaR-Rj-exp-cri}
|{\rm Re}\langle b_*(t)\cdot\nabla  \wt R(t)+c_*(t)  \wt R(t),
        \na  \wt R_k (t)- \frac{i}{2} v_k \wt  R_k(t) \rangle|
\leq C B_*(t) \(  \phi(\delta_1 t)  + e^{-\delta_2 t} \).
\end{align}
Using H\"{o}lder's inequality and $\|\wt R_k\|_{H^1}\leq C$
we also have
\begin{align} \label{bnave-Rj-cri}
|{\rm Re}\langle b_*(t)\cdot\nabla \ve(t)+c_*(t) \ve(t),
      \na \wt  R_k(t) - \frac{i}{2} v_k  \wt R_k(t) \rangle|
      \leq C B_*(t) \|\ve(t)\|_{H^1}.
\end{align}
Here, the constants in \eqref{bnaR-Rj-exp-cri} and \eqref{bnave-Rj-cri} are independent of $\omega$.

Thus, combining \eqref{H1-Rj-cri}, \eqref{H2-Rj-cri}, \eqref{bnaR-Rj-exp-cri} and \eqref{bnave-Rj-cri}
we prove \eqref{ptve-naRk-inner-RHS}, as claimed.

Therefore, we conclude from \eqref{ptve-nablaRj-sub}-\eqref{ptve-naRk-inner-RHS}
and the lower bound, via \eqref{B*-wk-ak-deter-bdd},
\begin{align*}
   \|\na Q_{w_k}\|_{L^2}
   = w_k^{-1} \|\na Q\|_{L^2}
   \geq \frac 12 (w_k^0)^{-1} \|\na Q\|_{L^2}
\end{align*}
that
\begin{align}  \label{dalpha-t-bdd}
 \frac 12 (w_k^0)^{-1} \|\na Q\|_{L^2} |\dot{\a}_k(t)|
 \leq  C \( (\|\ve(t)\|_{L^2} + e^{-\delta_2 t}) Mod(t) + \|\ve(t)\|_{H^1} + B_* (t) \phi(\delta_1 t)  + e^{-\delta_2 t}\),
\end{align}
where $C,\delta_1, \delta_2>0$ are universal deterministic constants.

$(ii)$ Estimate of $\dot{\theta}_k$.
Taking the inner product of \eqref{equa-ve-cri} with
$ \wt R_k$ and taking the real part we get
\begin{align} \label{ptve-Rk-inner}
&-{\rm Im}\langle\partial_t\ve,  \wt R_{k}\rangle+{\rm Re}\langle\Delta \ve,  \wt R_{k}\rangle
    -  \(\dot{\t}_k - (w_k^{-2}- (w_{k}^0)^{-2} )\)
      \| Q_{w_k}\|_{L^2}^2   \nonumber \\
=&  -{\rm Re} \langle H_1 + H_2 + H_3, \wt R_k\rangle
  - {\rm Re}\langle b_*\cdot\nabla  \wt R+ c_* \wt  R,  \wt R_k\rangle
  - {\rm Re}\langle b_*\cdot\nabla \ve+ c_* \ve,  \wt R_k\rangle .
\end{align}

Similarly to \eqref{ptve-naRk-inner-RHS},  we have
\begin{align} \label{ptve-Rk-inner-RHS}
   |{\rm R.H.S.\ of\ \eqref{ptve-Rk-inner}}|
   \leq C\(\|\ve(t)\|_{H^1} + B_*(t) \phi(\delta_1 t) + (Mod(t)+1)e^{-\delta_2 t}\),
\end{align}
where $C, \delta_1, \delta_2$ are universal deterministic constants.

For the L.H.S. of \eqref{ptve-Rk-inner},
we note that, by \eqref{ortho-cond-R-ve-cri} and \eqref{equa-Rk-cri},
\begin{align} \label{ptve-Rj-cri}
-{\rm Im}\langle\partial_t\ve,  \wt R_{k}\rangle={\rm Im}\langle\ve, \partial_t \wt R_{k}\rangle
=\calo( Mod_k+1)
\|\ve\|_{L^2},
\end{align}
and
\begin{align} \label{Deve-Rk-cri}
   {\rm Re} \< \Delta \ve,  \wt R_k\>
   = {\rm Re} \<\ve, \Delta  \wt R_k\>
   =\calo(\|\ve\|_{L^2}).
\end{align}

Therefore, we conclude from   \eqref{ptve-Rk-inner-RHS}-\eqref{Deve-Rk-cri}
and the identity
\begin{align*}
    \|\wt R_k\|_{L^2}^2 = \|Q_{w_k}\|_{L^2}^2 = \|Q\|_{L^2}^2
\end{align*}
that for some deterministic constants $C, \delta_1, \delta_2> 0 $,
\begin{align}  \label{dthetak-t-bdd}
     & \| Q\|_{L^2}^2 | \dot{\t}_k(t)-(w_{k}^{-2}(t)-w_{k,0}^{-2})|  \nonumber \\
 \leq&  C \( (\|\ve(t)\|_{L^2} + e^{-\delta_2 t}) Mod(t) + \|\ve(t)\|_{H^1} + B_* (t)\phi(\delta_1 t) + e^{-\delta_2 t}\).
\end{align}

$(iii)$  {Estimate of $\dot{w}_k$}.
Taking the inner product of \eqref{equa-ve-cri} with
$\Lambda_k \wt R_k-\frac{i}{2}v_k\cdot y_k \wt  R_k$,
$y_k$ as in \eqref{Lambdak-def},
then taking the imaginary part
and using the identity
\begin{align} \label{LambdaRk-LambdaQ}
   \Lambda_k \wt R_k(t,x) -\frac{i}{2}v_k\cdot (x-v_kt - \a_k)  \wt R_k(t,x)
   = \Lambda_k Q_{w_k(t)} (x-v_kt-\a_k)  e^{i\Phi_k(t,x) },
\end{align}
we derive that
\begin{align} \label{ptve-lambdaRk-inner}
&{\rm Re}\langle\partial_t\ve,\Lambda_k \wt R_k-\frac{i}{2}v_k\cdot y_k  \wt R_k\rangle
  +{\rm Im}\langle\Delta \ve,\Lambda_k \wt R_k-\frac{i}{2}v_k\cdot y_k  \wt R_k\rangle
  -   \frac{\dot{w_k}}{w_k} \| \Lambda Q_{w_k}\|_{L^2}^2    \nonumber \\
  & -   \dot{\alpha}_k {\rm Re}
  \langle \nabla Q_{w_k}   , \Lambda Q_{w_k}   \rangle
    -  \(\dot{\t}_k +((w_{k}^0)^{-2} - w_k^{-2})\)
    {\rm Im}\< Q_{w_k}   , \Lambda Q_{w_k} \> \nonumber \\
=&  -{\rm Im}\langle H_1 + H_2 + H_3,\Lambda_k  \wt R_k-\frac{i}{2}v_k\cdot y_k  \wt R_k\rangle \nonumber\\
  &- {\rm Im}\langle b_*\cdot (\nabla \wt  R + \na \ve)+c_*(  \wt R+ \ve), \Lambda_k \wt R_k-\frac{i}{2}v_k\cdot y_k  \wt R_k\rangle.
\end{align}

Again the {\rm R.H.S.} of \eqref{ptve-lambdaRk-inner}
contributes the orders as in \eqref{ptve-naRk-inner-RHS} and \eqref{ptve-Rk-inner-RHS}.

For  the {\rm L.H.S.} of  \eqref{ptve-lambdaRk-inner},
by \eqref{ortho-cond-R-ve-cri},
\eqref{LambdaRk-LambdaQ}
and the exponential decay \eqref{Q-decay},
\begin{align} \label{ptve-lambdaRj-sub}
 {\rm Re}\langle\partial_t\ve, \Lambda_k \wt R_k-\frac{i}{2}v_k\cdot y_k  \wt R_k\rangle
  ={\rm Re}\langle\ve, \partial_t (\Lambda_k  \wt R_k - \frac i2 v_k\cdot y_k  \wt R_k) \rangle
 =\calo( Mod_k +1)
\|\ve\|_{L^2},
\end{align}
and
\begin{align*}
   {\rm Im} \<\Delta \ve, \Lambda_k \wt R_k - \frac i2 v_k\cdot y_k \wt R_k\>
   = {\rm Im} \< \ve, \Delta ( \Lambda_k \wt R_k - \frac i2 v_k\cdot y_k \wt R_k)\>
   = \calo(\|\ve\|_{L^2}).
\end{align*}

Moreover,
we have
\begin{align}  \label{naQ-LambdaQ}
   {\rm Re} \<\na Q_{w_k}, \Lambda Q_{w_k} \> =0, \ \
   {\rm Im} \< Q_{w_k}, \Lambda Q_{w_k}\> =0.
\end{align}

Thus, takeing $w_k$ close to $w_k^0$
such that $\|\Lambda Q_{w_k}\|_{L^2} \geq \frac 12 \|\Lambda Q_{w_k^0}\|_{L^2}$
we obtain that
\begin{align}  \label{dwk-t-bdd}
  \frac 12 \|\Lambda Q_{w^0_k}\|_{L^2}^2| \dot{w}_k(t)|
 \leq C\( (\|\ve(t)\|_{L^2} + e^{-\delta_2 t}) Mod(t) + \|\ve(t)\|_{H^1} + B_*(t)  \phi(\delta_1 t)  + e^{-\delta_2 t}\).
\end{align}

Therefore,
combining \eqref{dalpha-t-bdd}, \eqref{dthetak-t-bdd} and \eqref{dwk-t-bdd} together
we conclude that
\begin{align}
   Mod(t) \leq C \( (\|\ve(t)\|_{L^2} + e^{-\delta_2 t}) Mod(t) + \|\ve(t)\|_{H^1} + B_*(t)\phi(\delta_1 t) + e^{-\delta_2 t}\).
\end{align}
where $C,\delta_1, \delta_2>0$ are deterministic constants.
Hence, for $t$ close to $T$ and large enough
such that
$C(\|\ve(t)\|_{L^2} + e^{-\delta_2 t}) \leq 1/2$
we obtain
\begin{align}
   Mod (t) \leq C \(  \|\ve(t) \|_{H^1} + B_*(t) \phi(\delta_1 t) + e^{-\delta_2 t} \).
\end{align}
The proof of Proposition \ref{Prop-Mod-cri} is complete. \hfill $\square$

\subsection{Subcritical case}  \label{Subsec-Geom-subcri}

In the subcritical case,
we only need to control three unstale directions,
corresponding to the coercivity of the linearized operator.
Two of them will be controlled by the following geometrical decomposition
and the remaining one ${\rm Re}\<\wt R_k,\ve\>$ can be controlled by the almost conservation of the local mass
in Section \ref{Sec-Local-Wein-cri} below.

\begin{proposition} (Geometrical decomposition)  \label{Prop-dec-un-sub}
Assume that $u$ solves \eqref{equa-un-tn} with $1<p<1+\frac4d$.
For any $T$  sufficiently large,
there exist  $0\leq T^*<T$ and unique modulation parameters
$\mathcal{P}_k:= (\a_k, \theta_k)
\in C^1([T^*, T]; \bbr^{d}\times \bbr)$, $1\leq k\leq K$,
such that $u$ admits the geometrical decomposition
\begin{align} \label{geo-dec-sub}
    u(t,x)
    =\sum_{k=1}^{K} \wt R_{k}(t,x) + \ve(t,x)\ \(=: \wt R(t,x) + \ve(t,x)\),
\end{align}
where for every $1\leq k\leq K$,
\be \label{Rk-Qk-sub}
    \wt R_{k}(t,x) := Q_{w_{k}^0} \(x-v_k t-\alpha_{k}(t) \)e^{i(\half v_k\cdot x-\frac{1}{4}|v_k|^2t + (w_{k}^0)^{-2}t+\t_{k}(t))},
\ee
the modulation parameters satisfy
\be
\varepsilon(T)=0, \quad \calp_k(T)=(x_k^0,\t^0_k),
\ee
and the following orthogonality conditions hold on $[T^*,T] $:
\be\ba\label{ortho-Rk-ve-sub}
   {\rm Re}\int \nabla \wt  R_{k}(t) \ol{\ve}(t)dx=0,
   \ \ {\rm Im} \int  \wt R_{k}(t) \ol{\ve}(t)dx=0.
\ea\ee
\end{proposition}

\begin{remark}
$(i)$. We note that, in  \eqref{Rk-Qk-sub}
$Q_{w_{k}^0}$ is indexed by a fixed parameter $w_k^0$,
which is different from the previous soliton profile \eqref{Rk-Qk-cri}
in the critical case
and from \cite{MM06} in the subcritical case,
where the parameter $w_k$ depends on time.

$(ii)$. The proof of Proposition \ref{Prop-dec-un-sub}
is quite similar to that of Proposition \ref{Prop-dec-un-cri}.
Actually, the corresponding Jacobian matrices $(\frac{\partial F^k}{\partial \wt \calp_j})$
can be obtained from those in the proof of Lemma \ref{Lem-Geo-Decom}
by removing $f_3^k$ and $\wt w_k$.
Hence, by \eqref{Jacob-posit},
the Jacobian matrix $\frac{\partial F}{\partial \wt \calp}$ is still
uniformly non-degenerate,
the arguments there are  applicable in the subcritical case.
\end{remark}

As in the previous critical case,
for $\bbp$-a.e. $\omega\in \Omega$,
we take a random time $T_*(\omega)>0$  large enough such that
\eqref{B*-T*-1}-\eqref{B*-wk-ak-deter-bdd}
hold on $[T^*,T]$,
and so $B_*$, $\|\ve\|_{H^1}$ and $|\a_k|$
bounded by a deterministic constant
on $[T^*,T]$.

Using the decomposition \eqref{geo-dec-sub} and orthogonality condition \eqref{ortho-Rk-ve-sub},
we can use similar arguments as in the proof of Proposition \ref{Prop-Mod-cri}
to  derive the control of modulation equations.

\begin{proposition} (Control of modulation equations)   \label{Prop-Mod-sub}
There exist deterministic constants $C, \delta_1, \delta_2>0$
such that
for $T$ large enough, $T^*$ close to $T$
and for any $t\in [T^*,T]$,
\begin{align} \label{Mod-esti-sub}
\sum_{k=1}^{K}(|\dot{\alpha}_{k}(t)|+|\dot{\t}_{k}(t)|)
 \leq C(  \|\ve(t)\|_{H^1} +  B_*(t) \phi(\delta_1 t)  + e^{-\delta_2 t}).
\end{align}
\end{proposition}

{\bf Proof.}
The arguments follow the lines as in the proof of Proposition \ref{Prop-Mod-cri}.
Using the explicit formula \eqref{Rk-Qk-sub} we compute
\begin{align} \label{dt-R}
i\partial_t \wt R_k(t,x)
 =\(\frac{|v_k|^2}{4}- (w_k^0)^{-2}-\dot{\t}_k(t) \) \wt R_k(t,x)
   -i(\dot{\alpha}_k(t)+v_k)\cdot\nabla Q_{w_k^0} (x-v_kt-\a_k) e^{i\Phi_k(t,x)},
\end{align}
where $\Phi_k$ is as in \eqref{phase-fct}.

Then, by \eqref{equa-Qw}, $ \wt R_k$ satisfies the equation
\begin{align}  \label{equa-Rk-sub}
  & i\partial_t \wt R_k(t,x)+\Delta  \wt R_k(t,x) +| \wt R_k(t,x)|^{p-1}  \wt R_k(t,x) \nonumber \\
 =& -i\dot{\alpha}_k(t) \nabla Q_{w_k^0} (x-v_kt-\a_k) e^{i\Phi_k(t,x)}
  -\dot{\t}_k (t) \wt R_k(t,x)
\end{align}

It then follows from equations \eqref{equa-u-low} and \eqref{equa-Rk-sub} that
\begin{align}   \label{equa-ve-sub}
  & i\partial_t\ve+\Delta\ve
 -  i\dot{\alpha}_k \nabla Q_{w_k^0}  (x-v_kt-\a_k) e^{i\Phi_k}   -  \dot{\t}_k \wt R_k
\nonumber \\
=& -H_1-H_2-H_3  - b_*\cdot(\nabla  \wt R+\nabla \ve) - c_*( \wt R+\ve),
\end{align}
where
\begin{align}
  & H_1:= -\sum\limits_{j\not = k} \(
        i\dot{\alpha}_j \nabla Q_{w^0_j} +  \dot{\t}_j  Q_{w^0_j}\)(x-v_kt-\a_k) e^{i\Phi_j}, \\
  & H_2:=| \wt R|^{p-1}  \wt R-\sum_{k=1}^{K}| \wt R_k|^{p-1}  \wt R_k,   \\
  & H_3:=| \wt R+\ve|^{p-1}( \wt R+\ve)-| \wt R|^{p-1} \wt  R.
\end{align}

Now, taking the inner product of \eqref{equa-ve-sub} with
$ \wt R_k$ and then taking the real part
we can control the dynamic of $\dot{\theta_k}$
\begin{align} \label{thetaj-esti-sub}
\|Q_{w_k^0}\|_{L^2}^2|\dot{\t}_k(t)|
\leq C(\|\ve(t)\|_{H^1} + B_* (t)\phi(\delta_1 t) + (Mod(t)+1)e^{-\delta_2 t}).
\end{align}

Moreover,
taking the inner product of \eqref{equa-ve-sub} with
$\nabla  \wt R_k-\frac{i}{2}v_k \wt R_k$ and then taking the imaginary part
we get the estimate of $\dot{\a_k}$
\be
\|\nabla Q_{w_k^0}\|_{L^2}^2|\dot{\alpha}_k(t)|
\leq   C(\|\ve(t)\|_{H^1} + B_*(t) \phi(\delta_1 t) + (Mod(t)+1)e^{-\delta_2 t}).
\ee
Here $C,\delta_1,\delta_2$  are deterministic positive constants.
Therefore, summing over $k$ and taking $t$ close to $T$ we obtain \eqref{Mod-esti-sub}
and finish the proof.
\hfill $\square$

\section{Local quantities and Lyapunov type functional} \label{Sec-Local-Wein-cri}

In this section we  control several important functionals,
including the local mass, local momentum, energy and
the Lyapunov  type functional,
for the subcritical and critical cases where $1<p\leq 1+\frac 4d$
simultaneously.

Note that,
these functionals depend on Brownian paths
and the energy is no longer conserved in the stochastic case.
Below we perform the path-by-path analysis in order to obtain the sharp estimates.
As in  Section \ref{Sec-Geom-dec},
for $\bbp$-a.e. $\omega\in \Omega$,
we take a random time $T_*(\omega)>0$  large enough such that
\eqref{B*-T*-1}-\eqref{B*-wk-ak-deter-bdd}
hold on $[T^*(\omega),T]$,
and so $B_*(\omega)$, $\|\ve(\omega)\|_{H^1}$,
$|\a_k(\omega)|$, $|w_k(\omega)|$ and $|w_k^{-1}(\omega)|$ are
bounded by a deterministic constant
on $[T^*(\omega),T]$.
For simplicity,  the dependence on $\omega$  is omitted.

\subsection{Local mass and local momentum}

Let us start with the analysis of the local mass. Because equation \eqref{equa-u-low} is invariant under the orthogonal transform,
we may take an orthonormal basis $\{\textbf{e}_j\}_{j=1}^d$ of $\bbr^d$ as in \cite{MM06},
such that $(v_j-v_k)\cdot \textbf{e}_1 \not = 0$ for any $j\not = k$.
Let $v_{k,1}:=v_k\cdot \textbf{e}_1, 1\leq k\leq K$. Without loss of generality,
we may assume that $v_{1,1}<v_{2,1}<\cdots<v_{K,1}$.
Following \cite{CF20} (see also  \cite{MM06}),
we set $A_0:=\frac{1}{4}\min_{2\leq k\leq K}\{v_{k,1}-v_{k-1,1}\}$
and $\sigma_k :=\frac{1}{2}(v_{k-1,1}+v_{k,1})$, $2\leq k\leq K$.
Let $\psi(x)$ be a smooth nondecreasing function on $\R$ such that $0\leq \psi\leq 1$,
$\psi(x)=0$ for $x\leq -A_0$,  $\psi(x)=1$ for $x>A_0$,
and there exists $C>0$ such that
\be
(\psi^\prime(x))^2\leq C\psi(x),\quad (\psi^{\prime\prime}(x))^2\leq C\psi^\prime(x),\ \ x\in \bbr^d.
\ee
The  localization functions are defined by
\be\ba
\varphi_1(t,x)=1-\psi\(\frac{x_1-\sigma_2t}{t}\),\quad \varphi_K(t,x)=\psi\(\frac{x_1-\sigma_Kt}{t}\),\\
\varphi_k(t,x)=\psi\(\frac{x_1-\sigma_kt}{t}\)-\psi\(\frac{x_1-\sigma_{k+1}t}{t}\), \quad 2\leq k\leq K-1.
\ea\ee
We have the  partition of unity $\sum_{k=1}^K \varphi_k(t,x)=1$.
Moreover,  for every $1\leq k\leq K$,
\be\label{de-phi}
|\varphi_k^{\prime}(t,x)|+|\varphi_k^{\prime\prime\prime}(t,x)|+|\partial_t\varphi_k(t,x)|\leq \frac{C}{t}.
\ee

For $1\leq k\leq K$, define the local mass and local momentum by
\be
I_{k}(t):=\int|u(t,x)|^2\varphi_k(t,x)dx, \quad M_{k}(t):={\rm Im}\int \nabla u(t,x)\bar{u}(t,x)\varphi_k(t,x)dx.
\ee

Though the local mass and local momentum are no longer conserved,
the explicit estimates in Proposition \ref{Prop-Mass-mom-cri} below
show that both local quantities are almost conserved.

\begin{proposition} (Control of local mass and local momentum) \label{Prop-Mass-mom-cri}
We have that for any $t\in [T^*,T]$,
\begin{align} \label{dtIk-cri}
\bigg|\frac{d}{dt}I_k(t)\bigg|\leq \frac{C}{t}(\|\ve(t)\|_{H^1}^2+e^{-\delta t}),
\end{align}
and
\begin{align}  \label{dtMk-cri}
\bigg|\frac{d}{dt}M_k(t)\bigg|
   \leq \frac{C}{t}(\|\ve(t)\|_{H^1}^2+e^{-\delta_2 t})
   +C B_*(t) (\|\ve(t)\|_{H^1}^2+ \phi(\delta_1 t)+e^{-\delta_2 t} ),
\end{align}
where $C,\delta_1, \delta_2>0$
are deterministic positive constants.
\end{proposition}

{\bf Proof.}
Using the integration-by-parts formula we compute
\be
\frac{d}{dt}I_k={\rm Im}\int (2\ol{u}\partial_{x_1}u+b_*|u|^2)\cdot \na \varphi_k dx
+ \int|u|^2\partial_t\varphi_kdx.
\ee

Note that,
the supports of $\vf'_k$ and $\partial_t\vf_k$
are contained in the regime
\begin{align*}
& \Omega_1=[(-A_0+\sigma_{2})t,(A_0+\sigma_{2})t]\times \R^{d-1}, \ \  \Omega_K=[(-A_0+\sigma_{K})t,(A_0+\sigma_{K})t]\times \R^{d-1},  \nonumber \\
& \Omega_k=[(-A_0+\sigma_{k})t,(A_0+\sigma_{k})t]\times \R^{d-1} \cup [(-A_0+\sigma_{k+1})t,(A_0+\sigma_{k+1})t]\times \R^{d-1}, 2\leq k\leq K-1.
\end{align*}
Taking into account \eqref{de-phi} we obtain
\begin{align}
\bigg|\frac{d}{dt}I_k(t)\bigg|\leq \frac{C}{t}\int_{\Omega_k}|u(t)|^2+|\nabla u(t)|^2dx.
\end{align}
Note that, for $x\in \Omega_k$ and $t$ large enough
so that $t\geq 4 A_0^{-1} \max_{1\leq k\leq K}\{1,|x_k^0|\}$,
$$|x-v_l t-\a_l|\geq |x_1 - v_{l,1}t| -|\a_l|
\geq A_0t - |\a_l| \geq \frac 12 A_0t,\ \ 1\leq l\leq K.  $$
Using the exponential decay of the ground state
we thus obtain
\be\ba
\bigg|\frac{d}{dt}I_k (t) \bigg|
  &\leq \frac{C}{t}(\|\ve(t)\|^2_{H^1(\Omega_k)}+\| \wt R(t)\|^2_{H^1(\Omega_k)})
\leq \frac{C}{t}(\|\ve(t)\|_{H^1}^2+e^{-\delta t}),
\ea\ee
which yields \eqref{dtIk-cri}.

Concerning the local momentum, straightforward computations show that
\begin{align} \label{dt-Mk}
\frac{d}{dt}{\rm Im}\int \partial_{x_1}u\bar{u}\varphi_kdx
 =&2\int |\partial_{x_1}u|^2\varphi^\prime_{k}dx
-\half \int |u|^2\varphi^{\prime\prime\prime}_{k}dx
-\frac{p-1}{p+1}\int |u|^{p+1}\varphi^{\prime}_{k}dx
+  {\rm Im}  \int\partial_{x_1}u\bar{u}\partial_t\varphi_kdx \nonumber \\
& -2 {\rm Re} \<\partial_1 u \vf_k, b_*\cdot \na u + c_* u\>
  - {\rm Re} \<u \partial_1 \vf_k, b_* \cdot \na u + c_* u\>,
\end{align}
and for $2\leq j\leq d$,
\begin{align} \label{dt-Mk-2}
\frac{d}{dt}{\rm Im}\int \partial_{x_j}u\bar{u}\varphi_kdx
  =&
2{\rm Re}\int\partial_{x_1} u\partial_{x_j}\bar{u}\varphi^\prime_kdx
+\int\partial_{x_j}u\bar{u}\partial_t\varphi_kdx \nonumber\\
  &-2 {\rm Re} \<\partial_j u \vf_k, b_*\cdot \na u + c_* u\>.
\end{align}

The first line in \eqref{dt-Mk} and \eqref{dt-Mk-2}
can be  bounded similarly as above by, up to a universal constant,
\begin{align} \label{localMk-esti.1}
   \frac 1 t \(\|\ve(t)\|_{H^1}^2 + e^{-\delta t}\).
\end{align}

Regarding the remaining inner products involving lower order perturbations,
as in the proof of \eqref{naphinaQ-esti},
the key fact is that,
since $Q_{w_k}$ is well localized,
$x$ is essentially localized around $|v_k|t$, i.e.,
$|x| \thicksim|v_k|t$.
Hence, taking into account the decay conditions in Assumption $(A1)$
we get
\begin{align} \label{localMk-esti.2}
  & |{\rm Re} \<\partial_1  \wt R(t) \vf_k(t), b_* (t)\cdot \na \ve(t)\>
    + {\rm Re} \< \partial_1 \ve (t)\vf_k(t), b_*(t) \cdot \na \wt  R(t)\>
    + {\rm Re} \<\partial_1 \wt  R(t) \vf_k(t), b_*(t) \cdot \na \wt  R(t)\>|  \nonumber \\
  \leq& C  B_*(t)   (1+ \|\na \ve(t)\|_{L^2}) (\phi(\delta_1 t) + e^{-\delta_2t}) \nonumber \\
  \leq& C B_*(t) (\phi(\delta_1 t) + e^{-\delta_2t})  .
\end{align}
Taking into account
\begin{align}
  |{\rm Re} \<\partial_1 \ve \vf_k, b_* \na \ve\>|
  \leq C B_* \|\na \ve\|_{L^2}^2,
\end{align}
we thus obtain
\begin{align} \label{p1ub*-esti}
   |{\rm Re} \< \partial_1 u (t)\vf_k(t), b_*(t)\cdot \na u(t)\>|
   \leq C  B_*(t) (\phi(\delta_1 t) + \|\na \ve(t)\|_{L^2}^2 +e^{-\delta_2t} ).
\end{align}

Using analogous arguments we obtain
\begin{align} \label{b*c*-Mk-esti}
    & |\< \partial_1 u(t) \vf_k(t), b_*(t)\cdot \na u (t)+ c_*(t) u(t)\>|
      + |\< u(t) \partial_1 \vf_k(t), b_*(t)\cdot \na u (t)+ c_*(t) u(t)\>|   \nonumber \\
    &  + | \<\partial_j u (t)\vf_k(t), b_*(t)\cdot \na u(t) + c_*(t) u(t)\>|  \nonumber \\
    \leq& C  B_*(t) (  \phi(\delta_1 t) + \|\ve(t) \|_{H^1}^2 +e^{-\delta_2t})
\end{align}

Therefore, combining \eqref{localMk-esti.1} and \eqref{b*c*-Mk-esti}  we obtain \eqref{dtMk-cri}.
The proof is complete.
\hfill $\square$

One important outcome of the almost conservation of local mass
is the following control of
the unstable direction ${\rm Re}\<\wt R_k, \ve\>$ in both the critical and subcritical settings.

\begin{corollary} (Control of unstable direction)   \label{Cor-Rkve-cri}
We have that for any $t\in [T^*,T]$,
\begin{align}  \label{Rkve-sub}
\bigg|{\rm Re}\int  \wt R_k(t) \bar{\ve}(t) dx\bigg|
\leq  C \( \int_{t}^{\infty}\frac{1}{s} \|\ve(s)\|_{H^1}^2 ds+ \|\ve(t) \|^2_{L^2}+ e^{-\delta t}\),
\end{align}
where $C,\delta>0$ are  deterministic constants.
\end{corollary}

{\bf Proof.}
Using the decomposition \eqref{geo-dec-cri} and \eqref{geo-dec-sub},
respectively,
in the critical and subcritical case,
we expand
\begin{align} \label{Ik-def-sub}
I_k =\int| \wt R|^2\varphi_k dx+2{\rm Re}\int  \wt R\bar{\ve}\varphi_kdx+\int|\ve|^2\varphi_k dx.
\end{align}

Note that,  by the decoupling Lemma \ref{Lem-decoup},
\begin{align} \label{Rvfk-esti.1}
   \int | \wt R|^2 \vf_k dx
     = \int | \wt R_k|^2 \vf_k dx
     + \sum\limits_{j\not = k} \int | \wt R_j|^2 \vf_k dx
     + \calo(e^{-\delta t}).
\end{align}
Since $\a_k$ is uniformly bounded, $|\a_k|\leq 2|x_k^0|$,
and on the support of $\vf_k$,
$|x-v_jt|\geq A_0 t$, $j\not =k$,
we infer that
for $t$ large enough
\begin{align}
   |x-v_j t - \a_j| \geq A_0t - |\a_j| \geq \frac 12 A_0 t,\ \ j\not =k,
\end{align}
which yields that
\begin{align} \label{Rvfk-esti.2}
   \int| \wt R_j|^2 \vf_k dx
   \leq C \int\limits_{|x-v_j t - \a_j| \geq \frac 12 A_0t }
          Q_{w_j}^2(x-v_jt-\a_j) dx
   \leq C e^{-\delta t}, \ \ j\not = k.
\end{align}
Moreover, since on the support of $1-\vf_k$,
$|x-v_kt|\geq A_0t$,
and so for $t$ very large
it holds that
$|x-v_kt - \a_k| \geq \frac 12 A_0 t$,
we   have
\begin{align} \label{Rvfk-esti.3}
   \int | \wt R_k|^2 \vf_k dx
   = \int | \wt R_k|^2 dx + \calo(e^{-\delta t}).
\end{align}
Thus, we derive from \eqref{Rvfk-esti.1}, \eqref{Rvfk-esti.2} and \eqref{Rvfk-esti.3} that
\begin{align} \label{intR2vfk}
\int| \wt R|^2\vf_k dx=\| \wt R_k\|_{L^2}^2+\calo(e^{-\delta t}).
\end{align}
Similarly,
\begin{align} \label{intRve}
{\rm Re}\int  \wt R\bar{\ve}\varphi_kdx={\rm Re}\int  \wt R_k\bar{\ve}dx+\calo(e^{-\delta t}\|\ve\|_{L^2}).
\end{align}
Thus, we conclude from  \eqref{Ik-def-sub}, \eqref{intR2vfk} and \eqref{intRve} that
\begin{align} \label{Ik-esti-sub}
I_k(t)=\| \wt R_k\|_{L^2}^2+2{\rm Re}\int  \wt R_k\bar{\ve}dx+\int|\ve|^2\varphi_kdx+\calo\(e^{-\delta t} \).
\end{align}

In particular,
letting $t=T$ and using $\ve(T)=0$ we get
\begin{align} \label{IkT-RkT}
I_k(T)=\| \wt R_k(T)\|_{L^2}^2+\calo(e^{-\delta T}).
\end{align}
Note that, in both the  critical and subcritical cases,
\begin{align}  \label{Rkt-RkT}
\| \wt R_k(t)\|_{L^2}=\| \wt R_k(T)\|_{L^2}.
\end{align}
In fact, via the scaling invariance,
one has $\| \wt R_k(t)\|_{L^2}=\| \wt R_k(T)\|_{L^2}=\|Q\|_{L^2}$ in the critical case.
While in the subcritical case,
since $w_k\equiv w_k^0$,
one has  $\| \wt R_k(t)\|_{L^2}=\| \wt R_k(T)\|_{L^2}=(w_k^{0})^{\frac{d}{2}-\frac{2}{p-1}}\|Q\|_{L^2}$.

Therefore, plugging \eqref{IkT-RkT} and \eqref{Rkt-RkT} into \eqref{Ik-esti-sub}
we then obtain
\be
{\rm Re}\int  \wt R_k(t) \bar{\ve}(t) dx
 =\half(I_k(t)-I_k(T)) - \half\int|\ve(t) |^2\varphi_kdx+\calo\(e^{-\delta t}\),
\ee
which, via Proposition \ref{Prop-Mass-mom-cri}, yields that
\begin{align}
\bigg|{\rm Re}\int \wt  R_k(t) \bar{\ve}(t)  dx\bigg|
&\leq \half\int_{t}^{T} \bigg|\frac{dI_k}{ds}\bigg| ds+\half\int|\ve(t) |^2dx+ C\(e^{-\delta t}\)   \nonumber \\
&\leq C\int_{t}^{T}\frac{1}{s}(\|\ve\|_{H^1}^2+e^{-\delta s})ds
       +C \(\|\ve(t) \|^2_{L^2}+ e^{-\delta t}\),
\end{align}
thereby proving \eqref{Rkve-sub} by letting $T$ tend to infinity.
\hfill $\square$

\subsection{Energy}

Proposition \ref{Prop-energy-cri} below
is concerned with the control of energy
defined by
\begin{align} \label{energy-def}
   E(u): = \frac 12 \|\na u\|_{L^2}^2 - \frac{1}{p+1}\|u\|_{L^{p+1}}^{p+1},
\end{align}
where $u$ is the solution to equation \eqref{equa-un-tn}.

Again the energy is no longer conserved
due to the  presence of lower order perturbations (or noise).
The variation control of the energy is estimated in the following proposition.

\begin{proposition} (Control of energy)  \label{Prop-energy-cri}
There exist deterministic constants $C, \delta_1, \delta_2>0$
such that
\begin{align} \label{dt-energy-sub}
\bigg|\frac{d}{dt}E(u(t))\bigg|\leq C B_*(t) (\phi(\delta_1 t)+\|\ve(t) \|_{H^1}^2 + e^{-\delta_2t}), \ \ \forall t\in [T^*,T].
\end{align}
\end{proposition}

{\bf Proof.}
Using (\ref{equa-un-tn}) and the integration-by-parts formula we compute
\begin{align} \label{dtE}
\frac{d}{dt}E (u)
=&-{\rm Im}\int(\ol{b_*}\cdot\nabla \ol{u}+ \ol{c_*} \ol{u})(\Delta u+|u|^{p-1} u) dx \nonumber \\
=&2 \sum\limits_{l=1}^N B_{*,l}
    {\rm Re}\int \nabla^2 \phi_l(\nabla u,\nabla \ol{u})dx
    - \frac 12 \sum\limits_{l=1}^N B_{*,l} \int \Delta^2 \phi_l |u|^2dx  \nonumber \\
&  -  \frac{p-1}{p+1} \sum\limits_{l=1}^N B_{*,l}
     \int \Delta \phi_l |u|^{p+1} dx
   - {\rm Im } \int \na \sum\limits_{j=1}^d
     \(   \sum\limits_{l=1}^N \partial_j \phi_l B_{*,l} \)^2 \cdot \na u \ol{u} dx.
\end{align}
Then, using \eqref{geo-dec-cri} and estimating similarly as in
the proof of \eqref{naphinaQ-esti}, we obtain
\begin{align} \label{e}
 \bigg|\frac{d}{dt}E(u(t))\bigg|
 \leq&C B_*(t) ( \phi(\delta_1 t) +\|\ve(t) \|_{H^1}^2+\|\ve(t) \|_{H^1}^{p+1} + e^{-\delta_2t}),
\end{align}
where $C,\delta_1, \delta_2>0$, and thus  \eqref{dt-energy-sub} follows.
\hfill $\square$

\subsection{Lyapunov type functional}

The key ingredient to control the size of remainder is the following Lyapunov type functional
\begin{align}  \label{Wei-def}
G(t):=2E(u(t))+\sum_{k=1}^{K}\left\{\( (w_{k}^0)^{-2}+\frac{|v_k|^2}{4} \)I_k(t) -v_k\cdot M_k(t) \right\}.
\end{align}
Recall that, in the subcritical case, we have  $w_k(t)\equiv w_k^0$
in the geometrical decomposition \eqref{geo-dec-sub}.

The main estimate for $G(t)$ is formulated in Proposition \ref{Prop-W-expan-cri} below.

\begin{proposition} (Expansion of Lyapunov type functional) \label{Prop-W-expan-cri}
Let $1<p\leq 1+\frac 4d$, $d\geq 1$.
Then, for any $t\in [T^*, T]$ we have
\begin{align} \label{fun-g-cri}
G(t)
 =& \sum_{k=1}^{K}(2E(Q_{w_{k}^0})+ (w_{k}^0)^{-2}\|Q_{w_{k}^0}\|_{L^2}^2)
 +H(\varepsilon(t) )
  + \calo(|w_k(t) -w_{k}^0|\ \|\ve(t) \|^2_{H^1}+ e^{-\delta t})  \nonumber \\
  &  +o(\|\ve(t) \|_{H^1}^2)
  +\calo\(\sum_{k=1}^{K}\bigg|(w_{k}^0-w_k(t) ){\rm Re}\int  \wt R_k(t) \bar{\ve}(t) dx\bigg| \).
\end{align}
where $H(\ve)$ contains the quadratic terms of $\ve$, i.e.,
\be\ba
H(\ve)&=\int|\nabla \ve|^2dx
         -\sum_{k=1}^{K} \int \frac{p+1}{2} | \wt R_k|^{p-1} |\ve|^2
         +(p-1)| \wt R_k|^{p-3}[{\rm Re}( \wt R_k\bar{\ve})^2]dx \\
      &\quad+\sum_{k=1}^{K}\left\{\(w_{k}^{-2}+\frac{|v_k|^2}{4}\)
        \int |\ve|^2\varphi_kdx-v_k\cdot {\rm Im}\int\nabla \ve \bar{\ve}\varphi_kdx\right\},
\ea\ee
and the implicit constant and $\delta$ are independent of $\omega$.
\end{proposition}

\begin{remark}
$(i)$.
It should be mentioned that,
the main part in \eqref{fun-g-cri} is independent of time.
This fact is obvious in the subcritical case
because the parameter $\omega_k \equiv w_k^0$ is independent of time.
While, in the critical case
it relies on the scaling invariance and Pohozaev identity (see \eqref{identity} below).

$(ii)$.
Another important property is the coercivity of
the quadratic term $H(\ve)$, i.e.,
for some $C>0$,
\begin{align}  \label{H-coercive}
H(\ve)\geq C\|\ve\|_{H^1}^2-\frac{1}{C}\sum_{k=1}^{K}\left({\rm Re}\int  \wt R_k\bar{\ve}dx\right)^2.
\end{align}
The coercivity in particular enables us to control the remainder $\ve$
in the geometrical decomposition.
\end{remark}

{\bf Proof of Proposition \ref{Prop-W-expan-cri}.}
First, using \eqref{geo-dec-cri} (or \eqref{geo-dec-sub}) and Lemma \ref{Lem-decoup}
we expand the kinetic energy
\begin{align}  \label{nau-expa-cri}
\int |\nabla u|^2dx&=\int |\nabla  \wt R|^2dx+\int |\nabla \ve|^2dx-2{\rm Re}\int\Delta  \wt R\bar{\ve}dx  \nonumber  \\
=&\sum_{k=1}^{K}\(\int |\nabla  \wt R_k|^2dx+\int |\nabla \ve|^2dx-2{\rm Re}\int\Delta  \wt R_k\bar{\ve}dx\)+\calo(e^{-\delta t})     \\
=&\sum_{k=1}^{K}\(\int |\nabla Q_{w_{k}}|^2dx
   +\frac{|v_k|^2}{4}\int|Q_{w_{k}}|^2dx
   +\int |\nabla \ve|^2dx-2{\rm Re}\int\Delta  \wt R_k\bar{\ve}dx\)+\calo(e^{-\delta t}). \nonumber
\end{align}
Moreover, for the potential energy we expand
\begin{align}   \label{up-expan-cri.0}
    \int |u|^{p+1}dx
    =&   \int | \wt R|^{p+1}dx
        +(p+1) {\rm Re}\int | \wt R|^{p-1}  \wt R\bar{\ve}dx  \nonumber \\
     &  + \frac{p+1}{2} \int \frac{p+1}{2} | \wt R|^{p-1} |\ve|^2
        +(p-1)| \wt R|^{p-3}[{\rm Re}( \wt R\bar{\ve})^2]dx +\calo(Er),
\end{align}
where the error term
\begin{align}
  Er:= \int\sum\limits_{z^*, \wt{z}^* \in \{z,\ol{z}\}}
        \int_0^1 r \int_0^1 \(\partial_{z^* \wt z^*} g( \wt R+sr\ve) -  \partial_{z^* \wt z^*} g( \wt R)\) \ve^2 dr dsdx,
\end{align}
and $g:= |u|^{p+1}$.
We note that,
since $1<p \leq 1+ \frac 4d$,
we may take
$\rho \in (1,\9)$ such that
$\frac 1 \rho = (\frac 12 - \frac 1d)(p-1)$ if $d\geq 3$,
and $\rho = \frac{p-1}{8}$ if $d=1,2$.
Then,  $2\leq \rho (p-1) \leq 2+\frac{4}{d-2}$
and $2\leq 2\rho' \leq 2+\frac{4}{d-2}$ if $d\geq 3$,
$2\leq 2\rho'<\9$ and $2\leq \rho(p-1)<\9$ if $d=1,2$.
By Sobolev's embedding $H^1 \hookrightarrow L^{2\rho'}$,
\begin{align}
  |Er| \leq& C  \|\ve\|_{2\rho'}^2 \sum\limits_{z^*, \wt{z}^* \in \{z,\ol{z}\}}
            \bigg\| \int_0^1 r \int_0^1 \(\partial_{z^* \wt z^*} g( \wt R+sr\ve) -  \partial_{z^* \wt z^*} g( \wt R)\) dr ds \bigg\|_{L^\rho}
             \nonumber \\
      \leq& C  \|\ve\|_{H^1}^2   \sum\limits_{z^*, \wt{z}^* \in \{z,\ol{z}\}}
             \bigg\| \int_0^1 r \int_0^1 \(\partial_{z^* \wt z^*} g( \wt R+sr\ve) -  \partial_{z^* \wt z^*} g( \wt R)\) dr ds \bigg\|_{L^\rho}.
\end{align}
Moreover, since $\|\ve(t)\|_{H^1}\to 0$ as $t\to T$,
we infer that for any sequence $\{t_n\}$, $t_n\to T$,
there exists a subsequence (still denoted by $\{n\}$)
such that
$\ve(t_n) \to 0$, $dx$-a.e..
By Sobolev's embedding $H^1 \hookrightarrow L^{\rho (p-1)}$,
\begin{align}
   |\partial_{z^* \wt z^*} g( \wt R(t_n)+sr\ve(t_n)) -  \partial_{z^* \wt z^*} g( \wt R(t_n))|
   \leq C (| \wt R(t_n)|^{p-1} + |\ve(t_n)|^{p-1}) \in L^\rho.
\end{align}
Hence, by the dominated convergence theorem,
\begin{align}
   \bigg\| \int_0^1 r \int_0^1 \(\partial_{z^* \wt z^*} g( \wt R(t_n)+sr\ve(t_n)) -  \partial_{z^* \wt z^*} g( \wt R(t_n))\) dr ds \bigg\|_{L^\rho}
   \to 0, \ \ as\ t_n\to T.
\end{align}
Since $\{t_n\}$ is any arbitrary sequence converging to $T$,
we obtain that the above convergence is valid for any $t\to T$,
and thus
\begin{align} \label{Er-o1}
    |Er| = o(\|\ve\|_{H^1}^2).
\end{align}
Moreover, we claim that
\begin{align} \label{4}
&\frac{2}{p+1} \int | \wt R|^{p+1}dx
       + 2 {\rm Re}\int | \wt R|^{p-1}  \wt R\bar{\ve}dx    +\int \frac{p+1}{2} | \wt R|^{p-1} |\ve|^2
  +(p-1)| \wt R|^{p-3}[{\rm Re}( \wt R\bar{\ve})^2]dx  \notag \\
&=\frac{2}{p+1} \sum_{k=1}^{K}\int | \wt R_k|^{p+1}dx
       + 2 \sum_{k=1}^{K}{\rm Re}\int | \wt R_k|^{p-1}  \wt R_k\bar{\ve}dx  \notag  \\
& +\sum_{k=1}^{K} \int \frac{p+1}{2} | \wt R_k|^{p-1} |\ve|^2
  +(p-1)| \wt R_k|^{p-3}[{\rm Re}( \wt R_k\bar{\ve})^2]dx +\calo(e^{-\delta t}).
\end{align}
The proof of \eqref{4} is postponed in the Appendix. Thus, plugging \eqref{Er-o1} and \eqref{4} into \eqref{up-expan-cri.0}
we then obtain
\begin{align}    \label{up-expan-cri}
    \frac{2}{p+1} \int |u|^{p+1}dx
    =&  \frac{2}{p+1} \sum_{k=1}^{K}\int | \wt R_k|^{p+1}dx
       + 2 \sum_{k=1}^{K}{\rm Re}\int | \wt R_k|^{p-1}  \wt R_k\bar{\ve}dx    \nonumber    \\
& +\sum_{k=1}^{K} \int \frac{p+1}{2} | \wt R_k|^{p-1} |\ve|^2
  +(p-1)| \wt R_k|^{p-3}[{\rm Re}( \wt R_k\bar{\ve})^2]dx +o(\|\ve\|^2_{H^1}) +\calo(e^{-\delta t}).
\end{align}

Thus, combining \eqref{nau-expa-cri} and \eqref{up-expan-cri} together
we obtain
\begin{align}  \label{E-expan-cri}
2E(u)&=\sum_{k=1}^{K}\left(2E(Q_{w_{k}})+\frac{|v_k|^2}{4}\|Q_{w_{k}}\|_{L^2}^2\right)
-\sum_{k=1}^{K}2{\rm Re}\int(\Delta  \wt R_k + | \wt R_k|^{p-1} \wt R_k)\bar{\ve}dx   \nonumber     \\
&\quad+\int |\nabla \ve|^2 dx
   -\sum_{k=1}^{K} \int \frac{p+1}{2} | \wt R_k|^{p-1}|\ve|^2
    +(p-1)| \wt R_k|^{p-3}[{\rm Re}( \wt R_k\bar{\ve})^2] dx+o(\|\ve\|^2_{H^1})+\calo(e^{-\delta t}).
\end{align}

We also see from \eqref{Ik-esti-sub} that
\begin{align}  \label{wkIk-expan-cri}
\( (w_{k}^0)^{-2}+\frac{|v_k|^2}{4}\)I_k
=&\( (w_{k}^0)^{-2}+\frac{|v_k|^2}{4}\)\|Q_{w_{k}}\|_{L^2}^2
  +\(2 (w_{k}^0)^{-2}+\frac{|v_k|^2}{2}\){\rm Re}\int  \wt R_k\bar{\ve}dx \nonumber \\
 &+\( w_{k}^{-2}+\frac{|v_k|^2}{4}\)\int|\ve|^2\varphi_kdx
 +((w_{k}^0)^{-2}-w_{k}^{-2})\int|\ve|^2\varphi_kdx  +\calo(e^{-\delta t}).
\end{align}

Regarding the local momentum, we expand
\begin{align} \label{Mk-expan-cri}
M_k&={\rm Im}\int\nabla  \wt R_k\ol{ \wt R}_k dx+2{\rm Im}\int \nabla  \wt R_k\bar {\ve}dx
+{\rm Im}\int \nabla \ve\bar{\ve}\varphi_kdx+\calo(e^{-\delta t}) \nonumber  \\
&=\frac{v_k}{2}\int|Q_{w_{k}}|^2dx+2{\rm Im}\int \nabla  \wt R_k\bar {\ve}dx
+{\rm Im}\int \nabla \ve\bar{\ve}\varphi_kdx+\calo(e^{-\delta t}),
\end{align}
which yields that
\be
v_k\cdot M_k=\frac{|v_k|^2}{2}\|Q_{w_{k}}\|_{L^2}^2+2v_k\cdot{\rm Im}\int \nabla \wt  R_k\bar {\ve}dx
+v_k\cdot{\rm Im}\int \nabla \ve\bar{\ve}\varphi_kdx+\calo(e^{-\delta t}).
\ee

Therefore, collecting \eqref{E-expan-cri}, \eqref{wkIk-expan-cri} and \eqref{Mk-expan-cri} altogether
we conclude that
\begin{align} \label{G-expan-cri.0}
G(t)=&\sum_{k=1}^{K}(2E(Q_{w_{k}(t)})+ (w_{k}^0)^{-2}\|Q_{w_{k}(t)}\|_{L^2}^2)  \nonumber  \\
   & -2\sum_{k=1}^{K}{\rm Re}\int (\Delta  \wt R_k(t) - (w_{k}^0)^{-2} \wt  R_k(t) +| \wt R_k(t)|^{p-1} \wt R_k(t))\bar{\ve}(t)dx
+\sum_{k=1}^{K}\frac{|v_k|^2}{2}{\rm Re}\int  \wt R_k(t)\bar{\ve}(t)dx  \\
   & - 2\sum_{k=1}^{K}v_k\cdot{\rm Im}\int \nabla  \wt R_k(t)\bar {\ve}(t)dx
+H(\varepsilon(t))
+\sum_{k=1}^{K} ((w_{k}^0)^{-2}-w_{k}^{-2}(t))\int|\ve(t)|^2\varphi_k(t)dx   \nonumber \\
   & +o(\|\ve(t)\|_{H^1}^2)+\calo(e^{-\delta t}).  \nonumber
\end{align}

Now, let us estimate the R.H.S. of \eqref{G-expan-cri.0}.
For the first term, we claim that
\be\label{identity}
2E(Q_{w_{k}(t)})+ (w_k^0)^{-2}\|Q_{w_{k}(t)}\|_{L^2}^2=
2E(Q_{w_{k}^0})+ (w_k^0)^{-2}\|Q_{w_{k}^0}\|_{L^2}^2.
\ee

Note that,
the R.H.S. above only depends on $w_k^0$ which is independent of time.
Hence, the identity \eqref{identity} shows that
the modulation parameter $w_k(t)$ (depending on time in the critical case) indeed does not affect the
main part of the Lyapunov  type functional.

This identity is obvious in the subcritical case as $w_k(t)\equiv w_k^0$.
Concerning the critical case,
the scaling invariance in the  $L^2$-critical case
yields that
\begin{align}  \label{EQk-expan}
2E(Q_{w_{k}})+ (w_k^0)^{-2}\|Q_{w_{k}}\|_{L^2}^2-
2E(Q_{w_{k}^0})-(w_k^0)^{-2}\|Q_{w_{k}^0}\|_{L^2}^2=2(w_k^{-2}-(w_k^0)^{-2})E(Q).
\end{align}
Then, by the key Pohozaev identity
\begin{align} \label{naQL2-QLp}
     (d-2) \|\na Q\|_{L^2}^2+d\|Q\|_{L^2}^2 = \frac{2d}{p+1} \|Q\|_{L^{p+1}}^{p+1},
\end{align}
we obtain
\begin{align} \label{EQ-0}
     E(Q)=0,
\end{align}
which along with \eqref{EQk-expan} yields \eqref{identity}, as claimed.

For the linear terms of $\ve$
on the R.H.S. of \eqref{G-expan-cri.0},
by \eqref{equa-Qw}, \eqref{nablaR} and \eqref{lambdaR},

\begin{align}
\Delta  \wt R_k- (w_{k}^0)^{-2} \wt R_k+| \wt R_k|^{p} \wt R_k
=\(w_k^{-2}- (w_{k}^0)^{-2} \)  \wt R_k
  + i v_k\cdot \na \wt  R_k + \frac 14 |v_k|^2  \wt R_k,
\end{align}
which yields the identity
\begin{align}    \label{G-expan-linear-esti}
&-2 {\rm Re}\int (\Delta  \wt R_k- (w_k^0)^{-2} \wt R_k+| \wt R_k|^{p} \wt R_k)\bar{\ve}dx
+\frac{|v_k|^2}{2}{\rm Re}\int  \wt R_k\bar{\ve}dx
  - 2v_k\cdot{\rm Im}\int \nabla  \wt R_k\bar {\ve}dx  \nonumber  \\
=& 2 ( (w_{k}^0)^{-2} -w_k^{-2}){\rm Re}\int  \wt R_k\bar{\ve}dx.
\end{align}

Therefore,
plugging \eqref{identity} and \eqref{G-expan-linear-esti}  into \eqref{G-expan-cri.0}
we obtain \eqref{fun-g-cri} and finish the proof.
\hfill $\square$

As a consequence, we have the crucial coercivity type control of the remainder.
\begin{proposition} (Coercivity type control of remainder) \label{Prop-veH1-cri}
Let $1<p\leq 1+\frac4d$, $d\geq 1$.
Then, there exist
deterministic constants $C, \delta_1, \delta_2>0$ such that for $t\in[T^*,T]$,
\begin{align} \label{est-gro}
\|\ve(t)\|_{H^1}^2
   \leq& C\( \int_{t}^{\infty}\frac{1}{s} \|\ve(s)\|_{H^1}^2 ds
             + \(\int_{t}^{\infty}\frac{1}{s} \|\ve(s)\|_{H^1}^2 ds\)^2 \) \nonumber \\
   & +C\( \int_{t}^{\infty} B_*(s) (\|\ve\|_{H^1}^2+ \phi(\delta_1 s))ds + e^{-\delta_2 t}\).
\end{align}
\end{proposition}

{\bf Proof.}
Since $\ve(T)=0$, by \eqref{fun-g-cri},
\begin{align}  \label{Gtn-sub}
   G(T)=\sum_{k=1}^{K}(2E(Q_{w_{k}^0})+ (w_{k}^0)^{-2}\|Q_{w_{k}^0}\|_{L^2}^2)
+\calo(e^{-\delta T}),
\end{align}
which along with Proposition \ref{Prop-W-expan-cri} yields that
\begin{align}  \label{Hve-cri}
H(\varepsilon(t))
=&G(t)-G(T)
+\calo\(\sum_{k=1}^{K}\bigg|(w_k(t)-w_{k}^0){\rm Re}\int  \wt R_k(t)\bar{\ve}(t)dx\bigg|\) \nonumber \\
 &+ \calo(|w_k(t)-w_{k}^0|\ \|\ve(t)\|^2_{H^1}+ e^{-\delta t})+o(\|\ve(t)\|_{H^1}^2) .
\end{align}
Taking into account the coercivity \eqref{H-coercive}
we then come to, for $t$ close to $T$,
\begin{align} \label{ve-G-esti}
  \|\ve(t)\|_{H^1}^2
  \leq& C \( |G(t)-G(T)|
        + \sum\limits_{k=1}^K\({\rm Re} \int  \wt R_k(t)\ol{\ve}(t) dx \)^2
+ \sum_{k=1}^{K}\bigg|(w_k(t)-w_{k}^0){\rm Re}\int \wt  R_k(t)\bar{\ve}(t)dx\bigg| \) \nonumber \\
  &+ C \(|w_k(t)-w_{k}^0|\ \|\ve(t)\|^2_{H^1}+ e^{-\delta t} \)+o(\|\ve(t)\|_{H^1}^2) .
\end{align}

Note that,
by definition \eqref{Wei-def} and Propositions \ref{Prop-Mass-mom-cri} and \ref{Prop-energy-cri},
\begin{align} \label{1}
|G(t)-G(T)|&\leq C|E(t)-E(T)|+C\sum_{k=1}^{K}(|I_k(t)-I_k(T)|+|M_k(t)-M_k(T)|) \nonumber \\
&\leq  C\int_{t}^{\infty}\frac{1}{s}(\|\ve(s)\|_{H^1}^2+e^{-\delta_2 s})ds
+ C\int_{t}^{\infty} B_*(s) (\|\ve(s)\|_{H^1}^2+ \phi(\delta_1 s) + e^{-\delta_2 s})ds,
\end{align}
where $C>0$.
Moreover, by Corollary \ref{Cor-Rkve-cri},
\be\label{2}
 \bigg| {\rm Re}\int \wt  R_k(t)\bar{\ve}(t) dx \bigg|
\leq C \( \int_{t}^{\infty}\frac{1}{s}\|\ve(s)\|_{H^1}^2ds + \|\ve(t)\|^2_{H^1}  + e^{-\delta t} \) .
\ee

Therefore,
combing \eqref{Hve-cri}, \eqref{ve-G-esti}, \eqref{1}, \eqref{2}
and letting  $T^*$ close to $T$ such that $|\omega_k-\omega_{k}^0|$ is small enough
we obtain \eqref{est-gro} and thus finish the proof.
\hfill $\square$

\section{Proof of main results} \label{Sec-Proof-Main}

This section is devoted to the proof of main results.
As in Section \ref{Sec-Local-Wein-cri}, we shall perform
the path-by-path analysis for $\mathbb{P}$-a.e. $\omega\in \Omega$.
The crucial ingredients of  the proof are the uniform estimates
of the remainder and geometrical parameters.

\subsection{Uniform estimates}    \label{Subsec-Uni-esti-cri}
Take any increasing sequence $\{T_n\}$ such that $\lim_{n\to \infty}T_n=+\infty$ and
consider the approximating solutions $u_n$
satisfying the equation on $[T_0,T_n]$
(for the definition of $T_0$ see Theorem \ref{Thm-un-uni-cri} below)
\be    \label{equa-u-t}
\left\{ \begin{aligned}
 &i\partial_t u_n+\Delta u_n+|u_n|^{\frac 4d}u_n+ (b_* \cdot \nabla +c_*) u_n =0,   \\
 &u_n(T_n)=\sum_{k=1}^{K}R_k(T_n)\ (=: R(T_n)).
\end{aligned}\right.
\ee

The uniform estimates of the remainder
and geometrical parameters are contained in Theorem \ref{Thm-un-uni-cri} below.

\begin{theorem} (Uniform estimates) \label{Thm-un-uni-cri}
Let $\delta_1, \delta_2>0$ be as in Propositions \ref{Prop-Mod-cri}, \ref{Prop-Mod-sub} and \ref{Prop-veH1-cri}.
Let $\wt \delta \in (0,\delta_1\wedge \delta_2)$ in {\rm Case (I)},
and $\wt \delta=1$ in {\rm Case (II)}.
Then, there exists $T_0>0$
such that for $n$ large enough,
$u_n$ admits the
geometrical decomposition \eqref{geo-dec-cri}
and \eqref{geo-dec-sub} on $[T_0, T_n]$
in the critical and subcritical cases, respectively,
and $u_n$ obeys the following estimate:
\be\label{ven-boot}
\|\ve_n(t)\|_{H^1}^2\leq \phi(\wt\delta t),\ \ t\in [T_0, T_n],
\ee
where $\phi$ is the decay function given by \eqref{phi-def}.

Moreover,
let $\calp_{n,k} = (\a_{n,k}, \theta_{n,k}, w_{n,k}) \in \bbx$, $1\leq k\leq K$,
be the corresponding modulation parameters.
Then, there exists $C,\wt \delta>0$ such that for $n$ large enough,
\be\label{a-theta-boot}
\sum_{k=1}^{K}(|w_{n,k}(t)-w_{k}^0|+|\al_{n,k}(t)-x_k^0|+|\t_{n,k}(t)-\theta^0_k|)
 \leq C \int_{t}^{\infty}s\phi^\frac 12(\wt \delta s)ds,
\quad \forall t\in[T_0,T_n].
\ee
\end{theorem}

The proof of Theorem \ref{Thm-un-uni-cri} relies crucially on
the following bootstrap estimate.

\begin{proposition} (Bootstrap estimate)  \label{Prop-dec-un-t0-boot}
Let $\wt \delta$ be as in Theorem \ref{Thm-un-uni-cri}.
Then, for $n$ large enough,
the following holds:

Suppose that there exists $t^*(<T_n)$ such that
$u_n$ admits the  decomposition \eqref{geo-dec-cri} (resp. \eqref{geo-dec-sub})
in the critical case (resp. the subcritical case)
and obeys estimate \eqref{ven-boot} on $[t^*,T_n]$.
Then,
there exists $t_* (<t^*)$ such that the decomposition
and the following improved estimate hold on $[t_*, T_n]$:
\begin{align} \label{ven-boot-2}
\|\ve_n(t)\|^2_{H^1}\leq \frac {1}{2}\phi(\wt\delta t).
\end{align}
\end{proposition}

{\bf Proof.}
By the continuity of solutions in $H^1$
and of the Jacobian matrices in the proof of geometrical decomposition,
we may take $t_*(<t^*)$  close to $t^*$,
such that the geometrical decompositions \eqref{geo-dec-cri},
\eqref{geo-dec-sub}
and the following estimate  holds on $[t_*,T_n]$,
\be\label{ven-boot-3}
\|\ve_n(t)\|_{H^1}^2\leq 2 \phi(\wt\delta t).
\ee

Using Proposition  \ref{Prop-Mod-cri} and \eqref{ven-boot-3}
we derive that for any $t\in[t_*,T_n]$,
\begin{align} \label{dot-wnkank}
\sum_{k=1}^{K}(|\dot{w}_{n,k}(t)|+|\dot{\alpha}_{n,k}(t)|)
\leq  C(\|\ve(t)\|_{H^1} + B_*(t)\phi (\wt \delta t) +e^{-\delta_2 t})
 \leq  C \phi^\frac 12(\wt \delta t),
\end{align}
which along with \eqref{Mod-esti-cri} and \eqref{Mod-esti-sub} yields that
\begin{align} \label{dot-thetank}
\sum_{k=1}^{K}|\dot{\t}_{n,k}(t)|
\leq& C(|w_{n,k}(t) -w_{k}^0|+ \|\ve(t)\|_{H^1}+e^{-\wt\delta t}) \nonumber \\
\leq& C \phi^\frac 12(\wt\delta t)+\int_{t}^{\9} \phi^\frac 12(\wt \delta s)ds
\leq Ct \phi^{\half}(\wt \delta t),
\end{align}
where the last step is due to the fact that
$\int_t^\9 e^{-\frac 12 \wt \delta s} ds \leq 2 \wt \delta^{-1}  e^{-\frac 12 \wt \delta s}$
and
$\int_t^\9 s^{-\frac{\upsilon_*}{2}} ds = \frac{2}{\upsilon_*-2}t^{-\frac{\upsilon_*}{2} + 1}
\leq 2 t \phi^\frac 12 (\wt \delta t)$ if $\upsilon_* \geq 3$.

Then, integrating \eqref{dot-wnkank} and \eqref{dot-thetank} over $[t,T_n]$
we get that
for a deterministic constant $C>0$,
\be\label{est-w-a-t}
\sum_{k=1}^{K}(|w_{n,k}(t) -w_{k}^0| + |\al_{n,k}(t)-x_{k}^0|+|\t_{n,k}(t)-\theta_{k}^0|)
\leq C \int_t^\9 s \phi^{\half}(\wt \delta s) ds,
\ee
which in particular converges to $0$ as $t\to \9$.
Hence,
the estimates in the previous subsection are all valid
for $t$ large enough.

Below we consider {\rm Case (I)} and {\rm Case (II)} separately.

In {\rm Case (I)},
by \eqref{phil-exp-decay}, \eqref{est-gro} and \eqref{ven-boot-3},
\begin{align}
\|\ve_n(t)\|_{H^1}^2
    \leq& C \( \int_t^\9 \frac{1}{s}e^{-\wt\delta s}d s
      + \(\int_t^\9 \frac{1}{s}e^{-\wt\delta s}d s \)^2
      + \int_t^\9  B_* e^{-\wt\delta s} ds
     + e^{-\delta t}  \)   \nonumber \\
    \leq& C \(\frac{1}{\wt\delta t} + \frac{1}{(\wt\delta t)^2}
                + \frac{1}{\wt\delta} B_*(t)
                + e^{-(\delta -\wt\delta)t}\) e^{-\wt\delta t}.
\end{align}
Taking $t$ large enough such that
\begin{align*}
   C \(\frac{1}{\wt\delta t} + \frac{1}{(\wt\delta t)^2}
                + \frac{1}{\wt\delta} B_* (t)
                + e^{-(\delta -\wt\delta)t}\)
   \leq \frac 12
\end{align*}
we get
\be
    \|\ve_n(t)\|_{H^1}^2 \leq \frac 12 e^{-\wt\delta t}.
\ee
This verifies estimate \eqref{ven-boot-2}  in {\rm Case (I)}.

Concerning {\rm Case (II)},
using \eqref{phil-poly-decay},  \eqref{est-gro} and \eqref{ven-boot-3}
we infer that
\begin{align} \label{ven-H1-poly-esti}
\|\ve_n(t)\|_{H^1}^2
    \leq& C \( \int_t^\9  s^{-\upsilon_*-1} d s
      + \(\int_t^\9 s^{-\upsilon_*-1} d s \)^2
      + \int_t^\9  B_*(s) s^{-\upsilon_*} ds
     + e^{-\delta t}  \)   \nonumber \\
    \leq& C \(\frac{1}{\upsilon_*} + \frac{t^{-\upsilon_*}}{\upsilon_*^2}
                + \frac{t B_*(t)}{\upsilon_*-1}
                + \frac{t^{\upsilon_*} e^{-\delta t}}{\delta}  \) t^{-\upsilon_*}.
\end{align}

Using the theorem on time change for continuous martingales
and the Levy H\"older continuity of Brownian motions
we derive from \eqref{gl-t2-decay}  that
$\bbp$-a.s.
for $t$ large enough,
\begin{align} \label{B*-Levy-contin}
   |B_{*,l}(t)| \leq 2 \(\int_t^\9 g_l^2 ds \log\(\int_t^\9 g_l^2 ds\)^{-1} \)^\frac 12
   \leq \frac {2\sqrt{c^*}}{t},\ \ 1\leq l\leq N,
\end{align}
which yields that $\bbp$-a.s. for $t$ large enough
$tB_*(t) \leq 2\sqrt{c^*}$.

Thus,
we may take $\upsilon_*$ large enough such that for any $t$ large enough,
\begin{align}
    C \(  \frac{1}{\upsilon_*} + \frac{t^{-\upsilon_*}}{\upsilon_*^2}
                + \frac{t B_*(t)}{\upsilon_*-1}
                + \frac{t^{\upsilon_*} e^{-\delta t}}{\delta}
         \)
    \leq \frac 12,
\end{align}
which in particular yields that
\begin{align}
   \|\ve_n(t)\|_{H^1}^2 \leq \frac 12 t^{-\upsilon_*}.
\end{align}
Therefore,
estimate \eqref{ven-boot-2} in {\rm Case (II)} is verified.
The proof is complete.
\hfill $\square$

{\bf Proof of Theorem \ref{Thm-un-uni-cri}.}
Estimate \eqref{ven-boot} can be proved by
using the  bootstrap estimate in Proposition \ref{Prop-dec-un-t0-boot} and
standard continuity arguments, see, e.g., \cite{MM06}, \cite{SZ20}.
Estimate \eqref{a-theta-boot} then follows from \eqref{est-w-a-t}.
\hfill $\square$

\subsection{Proof of main results} \label{Subsec-proof-main}

We are now in position to prove
Theorems \ref{Thm-Soliton-RNLS} and \ref{Thm-Soliton-SNLS}.

\paragraph{\bf Proof of Theorem \ref{Thm-Soliton-RNLS}}
By Theorem \ref{Thm-un-uni-cri},
$\{u_n(T_0)\}$ is uniformly bounded in $H^1$.
This yields that  up to a subsequence (still denoted by $\{n\}$),
for some $u_0\in H^1$,
\begin{align}
    u_n(T_0) \rightharpoonup u_0\ \ weakly\ in\ H^1,\ as\ n \to \9.
\end{align}

We claim that the convergence is strong in $L^2$, i.e.,
\begin{align} \label{un-u0-L2-cri}
    u_n(T_0) \to u_0\ \  in\ L^2,\ as\ n \to \9.
\end{align}

For this purpose,
it suffices to prove that
$\{u_n(T_0)\}$ is uniformly integrable,
i.e., for any $\epsilon>0$, there exists $A_\epsilon>0$ such that for all $n$ large,
\begin{align} \label{un-Ave-cri}
\int_{|x|\geq A_\epsilon}|u_n(T_0)|^2dx\leq \epsilon.
\end{align}

In order to prove \eqref{un-Ave-cri}, we first fix $T_1>T_0$ such that
\begin{align} \label{ve-T1}
   \|\ve(T_1)\|_{H^1}^2 \leq \phi(\wt \delta T_1) \leq \frac 16 \epsilon.
\end{align}
By \eqref{a-theta-boot},
we may take $A_0= A_0(v_k, T_1, x_k^0, 1\leq k\leq K)$ large enough such that for $|x|\geq A_0$ and for $1\leq k\leq K$,
\begin{align}
   |x-v_kT_1 - \a_{n,k}(T_1)|
  \geq |x| - |v_k| T_1
       - \sup\limits_{n\geq 1,t\geq T_0}  |\a_{n,k}(t)|
  \geq A_0 - \frac 1 2 A_0
  \geq \frac 12 A_0.
\end{align}
Hence, by the exponential decay of ground state,
for $A_0$ possibly larger,
\begin{align} \label{int-Rn-A0}
  \sup\limits_{n\geq1} \int_{|x|\geq A_0}| \wt R_n(T_1)|^2dx
  \leq C \int\limits_{|x|\geq \frac 12 A_0} e^{-\delta |x|} dx
  \leq C e^{-\frac{\delta}{4}A_0}
  \leq \frac{\epsilon}{6}.
\end{align}
Then, it follows from  \eqref{ve-T1} and \eqref{int-Rn-A0}  that
for all $n$ large enough,
\begin{align} \label{uT1-A-cri}
  \int_{|x|\geq A_0}|u_n(T_1)|^2dx
  \leq 2\int_{|x|\geq A_0}| \wt R_n(T_1)|^2dx+2\|\ve_n(T_1)\|^2_{H^1} \leq \frac{2\epsilon}{3}.
\end{align}
Moreover,
let $\chi$ be a smooth cut off function on $\R$ such that $0\leq\chi(x)\leq1$,
$\chi(x)=0$ for $|x|\leq\frac{1}{2}$; $\chi(x)=1$ for $|x|\geq1$ and $|\chi^\prime|\leq2$.
Let $\chi_{A_\epsilon}(x) :=\chi\(\frac{|x|}{A_\epsilon}\)$,
where $A_\epsilon=\max\{\frac{3\wt C (T_1-T_0)}{\epsilon},2A_0\}$
and $\wt C$ is the constant in \eqref{dtun-A-cri} below.
By the integration-by-parts formula,
\begin{align} \label{dtun-A-cri}
  \bigg|\frac{d}{dt}\int \chi_{A_\epsilon}|u_n(t)|^2dx\bigg|
=  \bigg|{\rm Im} (2\ol{u_n}\na u_n + b_* |u_n|^2) \cdot \na \chi_{A_\epsilon} \bigg|
   \leq \frac{\wt C}{A_\epsilon}
   \leq \frac{\epsilon}{3(T_1-T_0)},
\end{align}
Thus, we derive from \eqref{uT1-A-cri} and \eqref{dtun-A-cri} that,
for $n$ large enough,
\begin{align}
\int_{|x|\geq A_\epsilon}|u_{n}(T_0)|^2dx
&\leq \int_{\R^d}|u_{n}(T_1)|^2\chi_{A_\epsilon}dx+
\int_{T_0}^{T_1}\left|\frac{d}{dt}\int_{\R^d}|u_{n}(t)|^2\chi_{A_\epsilon}dx\right|dt  \nonumber \\
&\leq \int_{|x|\geq A_0}|u_{n}(T_1)|^2dx+\frac{\epsilon}{3}\leq \epsilon,
\end{align}
which yields \eqref{un-Ave-cri},
and thus proves \eqref{un-u0-L2-cri}, as claimed.

Now, for $n$ large enough,
since $u_n$ solves the equation \eqref{equa-u-t}
on $[T_0,T_n]$ with $\lim_{n\rightarrow\infty}T_n=+\infty$
and obey the uniform estimates in $C([T_0, T]; H^1$)
for any $T_0<T<\9$,
using the asymptotic \eqref{un-u0-L2-cri} and
comparison arguments
(see, e.g., \cite{HRZ18,TVZ07,Z18})
we infer that,
there exists a unique $L^2$-solution $u$ to \eqref{equa-u-low} on $[T_0, \9)$
such that
\begin{align} \label{un-u-0-L2}
\lim_{n\rightarrow \infty}\|u_n(t)-u(t)\|_{L^2}=0,\ \ \forall t\in [T_0, \9).
\end{align}
Moreover,
since $u_0\in H^1$,
the preservation of $H^1$-regularity also yields
$u(t) \in H^1$ for $t\in [T_0, \9)$.

Furthermore,
by \eqref{a-theta-boot}
and straightforward computations,
if $R:=\sum_{k=1}^K R_k$ with $R_k$ given by \eqref{Rk-def},
\begin{align}
\| \wt R_n(t)-R(t)\|_{H^1}
\leq& C \sum_{k=1}^{K}(|w_{n,k}(t)-\omega_{k}^0|+|\alpha_{n,k}(t)-x_{k}^0|
+|\t_{n,k}(t)-\t_k^0|)  \nonumber \\
 \leq& C \int_t^\9 s \phi^{\half}(\wt \delta s) ds.
\end{align}
Taking into account estimate \eqref{ven-boot}
we then obtain
\begin{align}
\|u_n(t)-R(t)\|_{H^1}
\leq& \|\ve_n(t)\|_{H^1}+\| \wt R_n(t)-R(t)\|_{H^1} \nonumber \\
\leq& C \( \phi^\frac 12 (\wt \delta t) + \int_t^\9 s \phi^{\half}(\wt \delta s) ds\) \nonumber \\
\leq& C \int_t^\9 s \phi^{\half}(\wt \delta s) ds,
\end{align}
where the last step is due to the explicit expression \eqref{phi-def}
of the decay function $\phi$.

In particular,
this yields that
$u_n(t)-R(t)$ is uniformly bounded in $H^1$ for every $t\in [T_0, \infty)$,
which along with \eqref{un-u-0-L2} implies that,
up to a subsequence (still denoted by $\{n\}$ which may depend on $t$),
\be
u_n(t)-R(t)\rightharpoonup u(t)-R(t),\ \ weakly\ in\ H^1,\ as\ n\rightarrow\infty.
\ee
This yields that
\be
\|u(t)-R(t)\|_{H^1}\leq \liminf_{n}\|u_n(t)-R(t)\|_{H^1}
\leq   C \int_t^\9 s \phi^{\half}(\wt \delta s) ds.
\ee

Therefore, the proof of Theorem \ref{Thm-Soliton-RNLS}  is complete.
\hfill $\square$

\paragraph{\bf Proof of Theorem \ref{Thm-Soliton-SNLS}}

Theorem \ref{Thm-Soliton-SNLS} now follows from Theorem \ref{Thm-Soliton-RNLS}
and Theorem \ref{Thm-Equiv-X-u} via the Doss-Sussman type transforms.

More precisely, by Theorem \ref{Thm-Soliton-RNLS},
there exists a unique solution $u$ to \eqref{equa-u-low} on $[T_0,\9)$
with $T_0>0$ sufficiently large, and the asymptotic behavior \eqref{u-R-asymp} holds.
This yields that
\begin{align}
   v:=e^{-W(\9)} u
\end{align}
is a unique solution to equation \eqref{equa-v-RNLS} on $[T_0,\9)$.
Thus, applying Theorem \ref{Thm-Equiv-X-u}
we obtain that
\begin{align} \label{X-v-u}
   X:=e^{W} v = e^{W_*} u
\end{align}
solves equation \eqref{equa-X-rough} on $[T_0, \9)$ in the sense of Definition \ref{X-roughpath-def}.
The asymptotic behavior \eqref{X-Rk-asym} thus follows from \eqref{u-R-asymp}.

Furthermore, in the $L^2$-subcritical case,
using the fixed point arguments as in \cite{BRZ14,BRZ16},
based on the Strichartz and local smoothing estimates,
we may extend the solution $u$ to a larger time interval $[\sigma^*, \9)$,
where $\sigma^* \in [0,T_0)$ is a non-negative random variable.
Because of the  subcriticality of the nonlinearity,
$\sigma^*$ depends on the $H^1$-norm of the solution,
and thus
$\sigma^*=0$ if the following uniform $H^1$-bound  holds
\begin{align} \label{u-H1-bdd-subcri}
   \sup\limits_{t\in(\sigma^*, T_0]} \|u(t)\|_{H^1} <\9.
\end{align}

In order to prove \eqref{u-H1-bdd-subcri},
we derive from the evolution formula  \eqref{dtE} of energy that,
for any $t\in (\sigma^*,T_0]$,
\begin{align}
   E(u(t)) \leq E(u(T_0)) + C \int_t^{T_0} \|u(s)\|_{H^1}^2 + \|u(s)\|_{L^{p+1}}^{p+1} ds.
\end{align}
which, via the interpolation estimate (see \cite[Lemma 3.5]{BRZ16}), for some $\rho>2$,
\begin{align} \label{uLp-uH1-sub}
   \|u\|_{L^{p+1}}^{p+1} \leq C_\ve \|u\|_{L^2}^\rho + \ve \|u\|^2_{H^1},
\end{align}
and the conservation law of mass,
yields that
\begin{align}
   E(u(t)) \leq C \(1+ \int_t^{T_0} \|u(s)\|_{H^1}^2 ds \),
\end{align}
where $C>0$ is independent of $t$.
Taking into account
the definition of energy \eqref{energy-def} and
using \eqref{uLp-uH1-sub} and the conservation law of mass again
we thus arrive at
\begin{align}
   \frac 12 \|u(t)\|_{H^1}^2
   =& E(u(t)) + \frac{1}{p+1} \|u(t)\|_{L^{p+1}}^{p+1} \nonumber \\
   \leq& C \(1+ \int_t^{T_0} \|u(s)\|_{H^1}^2 ds \) + \frac{\ve}{p+1} \|u\|^2_{H^1},
\end{align}
which yields \eqref{u-H1-bdd-subcri} by taking $\ve < \frac 14 (p+1)$
and applying Gronwall's inequality.

Therefore, it follows that $\sigma^*=0$,
$u$ and so $X$   can be extended to the whole time regime $[0,\9)$
in the subcritical case.
The proof of Theorem \ref{Thm-Soliton-SNLS} is complete.
\hfill $\square$

\section{Appendix} \label{Sec-App}

Let $L=(L_+,L_-)$ be the linearized operator around the ground state state
defined by
\begin{align} \label{L+-L-}
     L_{+}:= -\Delta + I -(1+p)Q^{p}, \ \
    L_{-}:= -\Delta +I -Q^{p}.
\end{align}
For any complex valued $H^1$ function,
set $f := f_1 + i f_2$
in terms of the real and imaginary parts and
\be
(Lf,f) :=\int f_1L_+f_1dx+\int f_2L_-f_2dx.
\ee

The crucial coercivity property of linearized operators
in the subcritical and critical case are summarized  below.
\begin{lemma}  (\cite{Wenn}, see also \cite[Lemma 2.2]{MMT06}) \label{Lem-coerc-L}
Let $1< p< 1+\frac{4}{d}$.
Then, there exists $C>0$ such that
\begin{align}
(Lf,f)\geq C\|f\|_{H^1}^2
-\frac{1}{C}\left((\int Qf_1dx)^2+(\int Qf_2dx)^2+(\int \nabla Qf_1dx)^2\right),
\end{align}
where $f=f_1+if_2$.
\end{lemma}

\begin{lemma} (\cite[Proposition 3.17]{CF20})  \label{Lem-coerc-cL}
Let $p=1 + \frac{4}{d}$.
Then, there exists $C>0$ such that
\be
(Lf,f)\geq C\|f\|_{H^1}^2-\frac{1}{C}\left((\int Qf_1dx)^2+(\int Qf_2dx)^2+(\int \nabla Qf_1dx)^2+(\int x\cdot\nabla Qf_1dx)^2\right),
\ee
where $f=f_1+if_2$.
\end{lemma}

\begin{lemma}  (Decoupling lemma)  \label{Lem-decoup}
For every $1\leq k\leq K$,
let
\begin{align}  \label{Gik-def}
  G_{i,k}(t,x) = w_k^{-\frac {2}{p-1}} g_{i}(\frac{x-v_k t - \a_k}{w_k}),\ \ i=1,2,
\end{align}
where $1\leq p\leq 1+ \frac 4d$, $g_i  \in C_b^2$ decays exponentially fast at infinity,
i.e., for some $C_1,\delta_1>0$,
\begin{align} \label{gi-exp-decay}
   |g_i(y)| \leq C_1 e^{-\delta_1|y|},\ \ y\in \bbr^d,\ i=1,2,
\end{align}
the parameters $w_k>0$, $v_k, \a_k\in \bbr^d$,
satisfying that
\begin{align} \label{wkvkak-O1}
  w_k^{-1}+ w_k+|v_k|+|\a_k| \leq C_2.
\end{align}
Then, if  $v_j \not = v_k$, $j\not =k$,
we have that
for any $p_1, p_2 >0$,
\begin{align}  \label{RjRk-decoup}
   \int | G_{1,j}(t)|^{p_1}|G_{2,k}(t)|^{p_2}dx  \leq Ce^{-\delta |v_j-v_k| t},
\end{align}
where $C,\delta>0$ depend on $\delta_1, C_i,p_i$, $i=1,2$.
\end{lemma}

{\bf Proof.}
We use \eqref{Gik-def} and the change of variables to compute
\begin{align} \label{RjRk-I0}
   \int | G_{1,j}(t)|^{p_1}|G_{2,k}(t)|^{p_2}dx
   =& w_j^{d-\frac{2p_1}{p-1}} w_k^{-\frac{2p_2}{p-1}}
     \int |g_1|^{p_1}(y) |g_2|^{p_2}\(\frac{w_j y + (v_j-v_k) t +(\a_j- \a_k)}{w_k}\) dy   \nonumber \\
   =& w_j^{d-\frac{2p_1}{p-1}} w_k^{-\frac{2p_2}{p-1}}
     \(\int_{\Omega} + \int_{\Omega^c}\) |g_1|^{p_1}(y)  |g_2|^{p_2}
     \(\frac{w_j y + (v_j-v_k) t +(\a_j- \a_k)}{w_k}\) dy \nonumber \\
   =&: I_1 + I_2.
\end{align}
where $\Omega := \{y \in \bbr^d: |y|\leq \frac{1}{2w_j} |v_j - v_k|t\}$
and $\Omega^c= \bbr^d\setminus \Omega$.
On one hand, by \eqref{wkvkak-O1},  for $t$ large enough,
\begin{align*}
   \big| {w_j y + (v_j-v_k) t +(\a_j- \a_k)} \big|
   \geq \frac{1}{2} |v_j -v_k|t -  |\a_j- \a_k|
   \geq \frac 14 |v_j - v_k| t,  \ \ y\in \Omega,
\end{align*}
which along with the exponential decay \eqref{gi-exp-decay}
and $w_j, w_k \geq C_2^{-1}>0$ yields that
\begin{align}   \label{RjRk-I1}
     I_1 \leq  C e^{-\frac{\delta_1 p_2}{4w_k} |v_j-v_k|t} \int g_1^{p_1}(y) dy
   \leq  C e^{-\delta' |v_j - v_k|t},
\end{align}
where $C,\delta'>0$ depend on $\delta_1,  C_i, p_i$, $i=1,2$.

On the other hand,
using \eqref{gi-exp-decay} again we infer that
\begin{align*}
    |g_1(y)| \leq C_1 e^{-\frac{\delta_1}{2w_j}|v_j-v_k|t} , \ \ y\in \Omega^c,
\end{align*}
and thus
\begin{align} \label{RjRk-I2}
  I_2   \leq C e^{-\frac{\delta_1 p_1}{2w_j} |v_j-v_k|t}
   \int_{\Omega}  g_2^{p_2}\(\frac{w_j y + (v_j-v_k) t +(\a_j- \a_k)}{w_k}\) dy
  \leq C  e^{-\delta'' |v_j - v_k|t},
\end{align}
where where $C,\delta_2>0$ depend on $\delta_1,  C_i, p_i$, $i=1,2$.

Therefore, plugging \eqref{RjRk-I1} and \eqref{RjRk-I2} into \eqref{RjRk-I0}
we obtain \eqref{RjRk-decoup} and finish the proof.
\hfill $\square$

Below we present the proof of the geometrical decomposition in Proposition \ref{Prop-dec-un-cri}
in a fashion close to that of \cite{CSZ21}.
Given any $L>0$, $w_k^0\in \R^+$, $x_k^0, v_k\in\R^d$, $\t_k^0\in\R$, $1\leq k\leq K$,
set
\be\label{RL}
R_L(x) : =\sum_{k=1}^{K}R_{k,L}(x)
=\sum_{k=1}^{K}(w_k^0)^{-\frac{2}{p-1}}Q\left(\frac{x-v_kL-x_k^0}{w_{k}^0}\right)
e^{i\(\half v_k\cdot x-\frac{1}{4}|v_k|^2L+(w_{k}^0)^{-2}L+\t_{k}^0\)}.
\ee
Note that, if $L=t$, then $R_{k,L}=R_k$ with $R_k$ given by \eqref{Rk-def}.

\begin{lemma}   \label{Lem-Geo-Decom}
There exists a universal small constant $\delta_*>0$ such that the following holds.
For any $0<r, L^{-1} <\delta_*$  and for any $u\in H^{-1}(\R^d)$ satisfying $\|u-R_L\|_{H^{-1}}\leq r$,
there exist unique $C^1$ functions
$\calp(u)=(\wt\al, \wt\t, \wt w): H^{-1}\to \bbx^K$
such that $u$ admits the decomposition
\be \label{dec-u-Ur-app}
 u  =\sum_{k=1}^{K}(\wt w_k w_k^0)^{-\frac{2}{p-1}}Q\left(\frac{x-v_kL-x_k^0-\wt \al_k}{\wt w_k w_{k}^0}\right)
e^{i\(\half v_k\cdot x-\frac{1}{4}|v_k|^2L+(w_{k}^0)^{-2}L+\t_{k}^0+\wt\t_k\)}+ \ve_L
 \ \ (=: \sum_{k=1}^{K}\wt R_{k,L}+\ve_L),
\ee
and the following orthogonality conditions hold: for $1\leq k\leq K$,
\be\ba \label{ortho-cond-R-ve-app}
& {\rm Re}\ _{H^{1}}\langle  \nabla \wt R_{k,L}, \ve_L\rangle_{H^{-1}}=0,\ \
   {\rm Im}\ _{H^{1}}\langle \wt R_{k,L}, \ve_L \rangle_{H^{-1}}=0,\\
& {\rm Re}\ _{H^{1}}\langle  \frac{d}{2}\wt R_{k,L}+y_k\cdot  \nabla \wt R_{k,L} - \frac{i}{2}v_k\cdot y_k \wt R_{k,L}, \ve_L \rangle_{H^{-1}}=0,
\ea\ee
where $y_k :=x-v_kL-x_k^0-\wt \al_k$.
Moreover, there exists a universal constant $C>0$ such that,
\be\label{3}
\|\varepsilon_L\|_{H^{-1}}+\sum_{k=1}^{K}(|\wt \al_k|+|\wt\t_k|+|\wt w_k-1|)\leq C\|u-R_L\|_{H^{-1}}.
\ee
\end{lemma}

\begin{proof}
The proof proceeds in four steps.

{\it Step 1.}
Set
$\wt \calp_{0,k}: =(0,0,1) \in \mathbb{X}$
and $\wt \calp_0 = (\wt \calp_{0,1}, \cdots, \wt \calp_{0,K})\in \mathbb{X}^K$.
Similarly,
let $\wt \calp_k:= (\wt \a_k, \wt \theta_k,\wt w_k) \in \mathbb{X}$,
$\wt \calp:= (\wt \calp_1, \cdots, \wt \calp_K) \in \mathbb{X}^K$.
Let
\be
\al_k : =\wt\al_k+x_k^0,\quad \t_k:=\wt \t_k+\t_k^0,\quad w_k:=\wt w_kw_k^0.
\ee
For any $u_0\in H^1$, let $B_\delta(u_0, \wt \calp_0)$ denote
the closed ball centered at $(u_0, \wt\calp_0)$ of radius $\delta$,
i.e.,
\begin{align}
  B_\delta(u_0, \wt \calp_0) :=\{(u, \wt \calp) \in H^{-1}\times \mathbb{X}^K:\ \|u-u_0\|_{H^{-1}}\leq \delta,\ \ |\wt \calp- \wt \calp_0|\leq \delta\},
\end{align}
where $\delta$ is a small constant to be chosen later,
and
\begin{align}
   |\wt \calp-\wt \calp_0|:=
   \sum\limits_{k=1}^K |\wt \calp_k- \wt \calp_{0,k}|
   =\sum_{k=1}^{K}(|\wt \a_k|+|\wt \t_k|+|\wt w_k-1|).
\end{align}

For $1\leq k\leq K$, let
\begin{align*}
 & f_{1,j}^k(u, \wt\calp) :={\rm Re}\ _{H^{1}}\langle \partial_j \wt R_{k,L}, \ve_L \rangle_{H^{-1}}, \ \ 1\leq j\leq d, \\
 & f_{2}^k(u,\wt\calp) :={\rm Im}\ _{H^{1}} \langle \wt R_{k,L},  \ve_L,  \rangle_{H^{-1}},  \\
 & f_{3}^k(u, \wt\calp )
   :={\rm Re}\ _{H^{1}}\langle \frac{2}{p-1}\wt R_{k,L}+y_k\cdot  \nabla \wt R_{k,L} - \frac{i}{2}v_k\cdot y_k \wt R_{k,L}, \ve_L  \rangle_{H^{-1}},
\end{align*}
where $y_k$ is as in \eqref{ortho-cond-R-ve-app}.
Let $F^k:= (f^k_{1,1}, \cdots, f^k_{1,d}, f^k_2, f^k_{3})$ and
$\frac{\partial F^k}{\partial  \wt\calp_j}$ denote the Jacobian matrix
\begin{align}
\frac{\partial F^k}{\partial  \wt\calp_j}:=\left(
     \begin{array}{ccc}
        \frac{\partial f^k_{1,1}}{\partial \wt\a_{j,1}} &\cdots \ \frac{\partial f^k_{1,1}}{\partial \wt\a_{j,d}},\ \frac{\partial f^k_{1,1}}{\partial \wt\t_j}, &
         \frac{\partial f^k_{1,1}}{\partial \wt w_j} \\
         \vdots &  &  \vdots \\
        \frac{\partial f^k_{3}}{\partial \wt\a_{j,1}} & \cdots \ \frac{\partial f^k_{3}}{\partial \wt\a_{j,d}},\ \frac{\partial f^k_{3}}{\partial \wt\t_j},&
         \frac{\partial f^k_{3}}{\partial \wt w_j} \\
     \end{array}
   \right), \ \ 1\leq j, k\leq K,
\end{align}
where $\wt \a_j:=(\wt \a_{j,l}, 1\leq l\leq d)\in \bbr^d$.
Similarly,
let $F:=(F^1, \cdots, F^K)$
and $\frac{\partial F}{\partial \wt  \calp} := (\frac{\partial F^k}{\partial \wt \calp_j})_{1\leq j,k\leq K}$.

Note that,
by the definition \eqref{RL} of $R_L$,
$ F^k(R_L, \wt\calp_0)=0$,
$1\leq k\leq K$.
Moreover, for any $(u, \wt \calp)\in B_\delta(R_L, \wt \calp_0)$,
we have that,
if $\wt R_L:= \sum_{k=1}^K \wt R_{k,L}$,
\begin{align} \label{R-wtR-SLUL}
\|\varepsilon_L\|_{H^{-1}} \leq \|u-R_L\|_{H^{-1}} + \|R_L - \wt R_L\|_{H^{-1}}.
\end{align}
By the explicit expressions of $R_L$ and $\wt R_L$
in \eqref{RL} and \eqref{dec-u-Ur-app}, respectively,
\begin{align} \label{SL-UL-diff}
 \|R_L - \wt R_L\|_{H^{-1}}
 \leq \|R_L - \wt R_L\|_{L^2}
\leq C\sum_{k=1}^{K}(|\wt\al_k|+|\wt\t_k|+|\wt w_k-1|)\leq C |\wt \calp - \wt \calp_0|,
\end{align}
where $C>0$. Thus,
we get that for a universal constant $\wt C>0$,
\begin{align} \label{R-wtR-PP0}
\|\varepsilon_L\|_{H^{-1}}\leq \wt C (\|u-R_L\|_{H^{-1}} + |\wt \calp - \wt \calp_0|)
           \leq 2 \wt C \delta,\ \ \forall (u, \wt \calp)\in B_\delta(R_L, \wt \calp_0).
\end{align}

{\it Step 2.}
We claim that, there exist small constants $\delta_*, c_1, c_2>0$
such that for any $0<\delta, L^{-1} \leq \delta_*$,
\begin{align} \label{Jacob-posit}
   0<c_1\leq \bigg|\det \frac{\partial F}{\partial \wt \calp}(u, \wt \calp) \bigg| \leq c_2<\9, \ \ \forall (u, \wt \calp)\in B_{\delta}(R_L, \wt \calp_0).
\end{align}

To this end, we compute that for $1\leq j,k\leq d$,
\be \ba  \label{F-P-Jacob}
  & \pa_{\wt\al_{k,j}} f^k_{1,j}  = - w_k^{-2} \|\partial_j Q_{w_k}\|_{L^2}^2+\calo(\|\varepsilon_L\|_{H^{-1}}), \  \
    \pa_{\wt\t_k} f_{1,j}^k   = - \frac{ v_{k,j}}{2}\|Q_{w_k}\|_{L^2}^2+\calo(\|\varepsilon_L\|_{H^{-1}}) , \\
  & \pa_{\wt\t_k} f^k_{2}  = \|Q_{w_k}\|_{L^2}^2+\calo(\|\varepsilon_L\|_{H^{-1}}), \ \
    \pa_{\wt w_k} f_3^k  =  w_k ^{-1} \|\Lambda Q_{w_k}\|_{L^2}^2+\calo(\|\varepsilon_L\|_{H^{-1}}).
\ea \ee
Moreover,
by the exponential decay of $Q$,
we infer that,
there exists $\delta>0$ such that
the other terms in the Jacobian matrices are of the order
$\calo(\|\ve\|_{H^{-1}} + e^{-\delta L})$.
This yields that
\begin{align}
   \bigg|\det \(\frac{\partial F}{\partial \wt \calp} \)\bigg|
   = \prod\limits_{k=1}^K \( (w_k^0)^{-2d} \wt w_k^{-2d-1} \|Q_{w_k}\|_{L^2}^2
      \|\Lambda Q_{w_k}\|_{L^2}^2 \prod\limits_{j=1}^d \|\partial_j Q_{w_k}\|_{L^2}^2  \)
      + \calo\(\|\ve\|_{H^{-1}} + e^{-\delta L}\).
\end{align}
Taking into account $|\wt \calp - \wt \calp_0| \leq \delta$
we obtain \eqref{Jacob-posit}, as claimed.

{\it Step 3.}
In this step, we claim that
there exists a universal constant $C_*(\geq1)$ such that, for any
$0<\delta, L^{-1} \leq \delta_*$
and  any
$(u_1, \wt \calp(u_1)), (u_2, \wt \calp(u_2))\in B_\delta(R_L, \wt \calp_0 )$,
if $F(u_1, \wt \calp(u_1))= F(u_2, \wt \calp(u_2))=0$,
then
\begin{align} \label{Pv1-Pv2}
|\wt \calp(u_1)-\wt \calp(u_2)|\leq C_* \|u_1-u_2\|_{H^{-1}}.
\end{align}

To this end,  we infer that
\begin{align} \label{Fv12-Fv21}
F(u_1, \wt \calp(u_1))-F(u_1, \wt \calp(u_2))= F(u_2, \wt \calp(u_2) ) - F(u_1, \wt \calp(u_2)).
\end{align}
By the differential mean value theorem,
\begin{align} \label{Fv12-pFpP}
   \(\frac{\partial F}{\partial \wt \calp}(u_1, \wt \calp_{r})\) (\wt \calp(u_1)-\wt \calp(u_2))^t
   = (F(u_2, \wt \calp(u_2))-F(u_1,\wt  \calp(u_2)))^t,
\end{align}
where $\wt \calp_r=r\wt \calp(u_1)+(1-r)\wt \calp(u_2)$
for some $0<r<1$,
and the superscript $t$ means the transpose of matrices.
Since the Jacobian matrix
$\frac{\partial F}{\partial \wt \calp}(u_1, \wt \calp_r)$ is invertible by  \eqref{Jacob-posit},
this leads to
\begin{align} \label{Fv12-Fv21-A}
   (\wt \calp(u_1)-\wt \calp(u_2))^t= \(\frac{\partial F}{\partial \wt \calp}(u_1,\wt \calp_r) \)^{-1} (F(u_2, \wt \calp(u_2)) - F(v_1, \wt \calp(u_2)))^t.
\end{align}
Note that, by \eqref{F-P-Jacob},
there exists a universal constant $C>0$ such that
\begin{align} \label{Fp-Jacob-invert}
   \bigg\|\bigg(\frac{\partial F}{\partial \wt \calp}(u_1, \wt \calp_r)\bigg)^{-1} \bigg\|\leq C,
\end{align}
where $\|\cdot\|$ denotes the Hilbert-Schmidt norm of matrices.
Moreover,
by  \eqref{Q-decay},
\begin{align} \label{f-v2v1-diff}
&|F(u_2, \wt \calp(u_2)) - F(u_1, \wt \calp(u_2))|\leq C \|u_2-u_1\|_{H^{-1}}.
\end{align}

Thus, we infer from \eqref{Fv12-Fv21-A}, \eqref{Fp-Jacob-invert}
and \eqref{f-v2v1-diff} that \eqref{Pv1-Pv2} holds, as claimed.

{\it Step 4.}
Let $\delta_*, C_*$ be the universal constants
as in Step 1 and Step 2, respectively,
and set
\begin{align}
B:=\{ v\in B_{\frac{\delta_*}{C_*}}(R_L): \exists \wt \calp \in B_{\delta_*}(\wt \calp_0),
     \ such\ that\ F(v,\wt \calp) =0\}.
\end{align}
Since $B_{\frac{\delta_*}{C_*}}(R_L)$ is connected
and $R_L \in B$,
in order to prove that
\begin{align} \label{B-BSL}
   B = B_{\frac{\delta_*}{C_*}}(R_L).
\end{align}
we only need to show that
$B$ is both open and closed in $B_{\frac{\delta_*}{C_*}}(R_L)$.

To this end,
For any $u\in B$, by definition
there exists $\wt \calp(u) \in B_{\delta_*}(\wt \calp_0)$ such that $F(u,\wt \calp(u))=0$.
Taking into account the non-degeneracy of the Jacobian matrix at $(u, \wt \calp(u))$
due to \eqref{Jacob-posit},
we can apply the implicit function theorem
to get a small open neighborhood $\calu (u)$ of $u$ in $B_{\frac{\delta_*}{C_*}}(R_L)$
such that $\calu (u) \subseteq B$.
This yield that $B$ is open in $B_{\frac{\delta_*}{C_*}}(R_L)$.

Moreover,
for any sequence $\{u_n\} \subseteq B$
such that $u_n \to u_*$ in $H^{-1}$ for some $u_*\in B_{\frac{\delta_*}{C_*}}(R_L)$,
by definition
there exist modulation parameters $\wt \calp(u_n)\in  B_{\delta_*}(\wt \calp_0)$
such that $F(u_n, \wt \calp(u_n))=0$,  $n \geq 1$.
In particular,
$\{\wt \calp(v_n)\} \subseteq \mathbb{X}^K$ is uniformly bounded
and so, along a subsequence (still denoted by $\{n\}$),
$\wt \calp(v_{n}) \to \wt \calp_*\ (\in B_{\delta_*}(\wt \calp_0))$
for some $\wt \calp_* \in \mathbb{X}^K$.

Then, let $\wt R_{k,L, \wt \calp(u_n)}$ and $\wt R_{k,L,\wt \calp_*}$
be the $k$-th soliton profiles corresponding to $\wt \calp(u_n)$
and $\wt \calp_*$, respectively.
By the above convergence of $u_n$ and $\wt \calp(u_n)$
we infer that $u_{n} - \sum_{k=1}^K \wt R_{k,L, \wt \calp(u_{n})}  \to u_* - \sum_{k=1}^K \wt R_{k,L,\wt \calp_*}$ in $H^{-1}$.
Taking  $n\to \9$ and using the fact that $F(u_{n}, \wt \calp(u_{{n}})) =0$
we obtain
$F(u_*, \wt \calp_*) = 0$,
and so $u_*\in B$.
Hence,
$B$ is also closed in $B_{\frac{\delta_*}{C_*}}(R_L)$.

Therefore, \eqref{B-BSL} is verified.
The geometrical decomposition \eqref{dec-u-Ur-app}
and the orthogonality conditions in \eqref{ortho-cond-R-ve-app} hold.
Moreover, estimate \eqref{3} follows from \eqref{R-wtR-PP0} and  \eqref{Pv1-Pv2}
by taking $u_1= u$ and $u_2 = R_L$.
The proof of Lemma \ref{Lem-Geo-Decom} is complete.
\end{proof}

{\bf Proof of Proposition \ref{Prop-dec-un-cri}.}
Since $u(T)=R(T)$,
by the local wellposedness theory,
there exists $T^*$ close to $T$, such that
$u(t)\in C^1([T^*,T]; H^{-1}) \bigcap C([T^*,T];H^1)$
and $\|u(t)-R(T)\|_{H^{1}} \in B_{\delta}(u(T))$ for all $t\in[T^*,T]$,
where $\delta>0$ is as in Lemma \ref{Lem-Geo-Decom}.

Hence, applying Lemma \ref{Lem-Geo-Decom} to $\{u(t)\}$
we obtain that for $T$ large enough,
there exist unique $C^1$ functions $(\al_k,\t_k,\omega_k) \in C^1([T^*,T]; \mathbb{X}^K)$,
$1\leq k\leq K$,
such that
for any $t\in [T^*,T]$,
$u(t)$ admits the decomposition \eqref{dec-u-Ur-app}
and the orthogonality conditions in \eqref{ortho-cond-R-ve-app} hold
with $t$ replacing $T$.

Then, taking into account $u(t)\in H^1$
and \eqref{dec-u-Ur-app}, the remainder $\ve(t)$ is indeed in the space $H^1$.
Thus, the parings between $H^{-1}$ and $H^1$ in \eqref{ortho-cond-R-ve-app} are exactly the $L^2$ inner products,
which yields the orthogonality conditions in \eqref{ortho-cond-R-ve-cri} for any $[T^*,T]$.
Therefore, the proof is complete.
\hfill $\square$

We close this section with the proof of  \eqref{4}. 
For this purpose, we set
$\wt S_k:= \sum_{j=k}^K \wt R_j$, $1\leq k\leq K$.
Then,
\begin{align}  \label{wtSk-wtSK1}
   \wt S_k = \wt R_k + \wt S_{k+1},\ \ 1\leq k\leq K-1.
\end{align}

\begin{lemma}   \label{Lem-Sk-Rk-q}
Let $0<q<\9$, we have
\be
\left||\wt S_k |^q-|\wt R_k|^q\right|
\leq C h(\wt S_{k+1}),
\ee
where $C>0$,
$h(\wt S_{k+1}) = |\wt S_{k+1}|^q$ if $0<q<1$,
and
$h(\wt S_{k+1}) = |\wt S_{k+1}|$ if $1\leq q<\9$.
\end{lemma}

{\bf Proof.}
The case where $0<q<1$ follows from the inequality
\begin{align*}
(a+b)^{q}\leq a^q+b^q, \ \ a,b \geq 0,
\end{align*}
while  the case $1\leq q<\9$ follows from the inequality
\begin{align*}
    ||\wt S_k |^q-|\wt R_k|^q |
   \leq C(|\wt S_{k+1}|^{q-1} + |\wt R_{k}|^{q-1}) |\wt S_{k+1}|
\end{align*}
and the uniform boundedness of $\wt R_j$, $1\leq j\leq K$.
\hfill $\square$

\begin{lemma}   \label{Lem-Skp1-Rkp1-exp}
There exists $\delta >0$ such that
\begin{align}  \label{Skp1-Rkp1-exp}
\int |\wt S_{k}|^{p+1}-|\wt R_k|^{p+1}-|\wt S_{k+1}|^{p+1}dx=\calo(e^{-\delta t}).
\end{align}
\end{lemma}

{\bf Proof.}
Using the expansion
\begin{align*}
    |\wt S_k|^2 = |\wt R_k|^2 + |\wt S_{k+1}|^2 + 2 {\rm Re} (\wt R_k \wt S_{k+1}),
\end{align*}
and Lemmas \ref{Lem-decoup} and \ref{Lem-Sk-Rk-q}
we have
\begin{align*}
   & \big| \int |\wt S_k|^{p+1}-|\wt R_k|^{p+1}-|\wt S_{k+1}|^{p+1}dx \big| \nonumber \\
   \leq& \int \big||\wt S_k|^{p-1}-|\wt R_k|^{p-1}\big| |\wt R_k|^2
            +\big||\wt S_k|^{p-1}-|\wt  S_{k+1}|^{p-1}\big||\wt S_{k+1}|^2
            +2|\wt  S_{k}|^{p-1}|\wt R_k\wt S_{k+1}|dx \\
    \leq& C \int    h(\wt S_{k+1}) |\wt R_k|^2
                  + h(\wt R_k) |\wt S_{k+1}|^2
+2|\wt  S_{k}|^{p-1} |\wt R_k \wt S_{k+1}|dx\\
    \leq& Ce^{-\delta t},
\end{align*}
which yields \eqref{Skp1-Rkp1-exp}.
\hfill $\square$

\begin{lemma}  \label{Lem-Sk1-Rk-ve-exp}
There exists $\delta >0$ such that
\begin{align}  \label{Sk1-Rk-ve-exp}
 \int (| \wt S_{k}|^{p-1}  \wt S_{k} - | \wt R_k|^{p-1}  \wt R_k-
        | \wt S_{k+1}|^{p-1}  \wt S_{k+1}) \ol{\ve}dx
        =\calo(e^{-\delta t} \|\ve\|_{L^2}).
\end{align}
\end{lemma}

{\bf Proof.}
By the expansion \eqref{wtSk-wtSK1},
Lemmas \ref{Lem-decoup} and  \ref{Lem-Sk-Rk-q}
and  H\"{o}lder's inequality,
\begin{align*}
&\big|\int (| \wt S_{k}|^{p-1}  \wt  S_{k} -| \wt R_k|^{p-1}  \wt R_k-| \wt  S_{k+1}|^{p-1}  \wt  S_{k+1})\ol{\ve}dx\big|\\
&\leq
\int \(\big|| \wt  S_{k}|^{p-1} -| \wt R_k|^{p-1}  \big||\wt R_k|
        +\big|| \wt S_{k}|^{p-1} -| \wt S_{k+1}|^{p-1} \big|| \wt S_{k+1}|\)|\ve|dx\\
&\leq  C (\|h(\wt  S_{k+1}) \wt R_k\|_{L^2}+ \|h(\wt R_k) \wt  S_{k+1}\|_{L^2}) \|\ve\|_{L^2}\\
&\leq Ce^{-\delta t} \|\ve\|_{L^2},
\end{align*}
which yields \eqref{Sk1-Rk-ve-exp}.
\hfill $\square$

\begin{lemma}  \label{Lem-Skp1-Rkp1-ve2-exp}
There exists $\delta >0$ such that
\begin{align}  \label{Skp1-Rkp1-ve2-exp}
\int (| \wt S_{k}|^{p-1} -| \wt R_k|^{p-1}-| \wt S_{k+1}|^{p-1} ) |\ve|^2dx
=\calo(e^{-\delta t} \|\ve\|_{L^2}^2).
\end{align}
\end{lemma}

{\bf Proof.}
Let $\Omega_k:=\{ x: |x-v_kt|\leq \frac 12 \min_{j\not =k} |v_k-v_j|t \}$.
By Lemma \ref{Lem-Sk-Rk-q},
\begin{align}   \label{Skp1-Rkp1-ve2-esti.0}
   &\bigg|\int (| \wt S_k|^{p-1} -| \wt R_k|^{p-1}-| \wt S_{k+1}|^{p-1} ) |\ve|^2dx\bigg|  \notag \\
  \leq& 2\int_{\Omega_k} (h(\wt S_{k+1}) + |\wt S_{k+1}|^{p-1})  |\ve|^2dx
        + 2\int_{\Omega_k^c}  (h(\wt R_k) + |\wt R_k|^{p-1}) |\ve|^2dx  \notag \\
  \leq& C \| h(\wt S_{k+1}) + |\wt S_{k+1}|^{p-1} \|_{L^\9(\Omega_k)}  \|\ve\|_{L^2}^2
        + C \|h(\wt R_k) + |\wt R_k|^{p-1}\|_{L^\9(\Omega_k^c)} \|\ve\|_{L^2}^2.
\end{align}
Note that, for $x\in \Omega_k$, for any $j\not =k$,
\begin{align*}
   |x-v_j t-\a_j| \geq |v_j-v_k|t - |x- v_kt| - |\a_j|
   \geq \frac 14 |v_j - v_k|t,
\end{align*}
and thus by the exponential decay of $Q$,
\begin{align}     \label{Skp1-Rkp1-ve2-esti.1}
   \|h(\wt S_{k+1}) + |\wt S_{k+1}|^{p-1}\|_{L^\9(\Omega_k^c)}
   \leq C e^{-\delta t}.
\end{align}
Similarly, for $x\in \Omega_k^c$,
there exists $c>0$ such that for $t$ large enough,
\begin{align*}
     |x-v_kt - \a_k| \geq \frac 12 \min_{j\not =k}\{|v_j-v_k|t\} - |\a_k|
     \geq ct,
\end{align*}
and thus
\begin{align}     \label{Skp1-Rkp1-ve2-esti.2}
   \|h(\wt R_k) + |\wt R_k|^{p-1}\|_{L^\9(\Omega_k^c)}
   \leq C  e^{-\delta t}.
\end{align}
Therefore, plugging \eqref{Skp1-Rkp1-ve2-esti.1} and \eqref{Skp1-Rkp1-ve2-esti.2} into
\eqref{Skp1-Rkp1-ve2-esti.0} we obtain \eqref{Skp1-Rkp1-ve2-exp}
and finish the proof.
\hfill $\square$

\begin{lemma}  \label{Lem-Skp3-Rkp3-ve2}
There exists $\delta >0$ such that
\begin{align} \label{Skp3-Rkp3-exp}
\int (|\wt S_k|^{p-3}\wt S_k^2 -| \wt R_k|^{p-3}\wt R_k^2
   -| \wt S_{k+1}|^{p-3}\wt S_{k+1}^2) \ol{\ve}^2dx
   = \calo(e^{-\delta t}).
\end{align}
\end{lemma}

{\bf Proof. }
Since
\begin{align*}
   |\wt S_k|^{p-1} \frac{\wt S_{k}^2}{|\wt S_k^2|}
    = |\wt R_k|^{p-1} \frac{\wt S_{k}^2}{|\wt S_k^2|}
      + |\wt S_{k+1}|^{p-1} \frac{\wt S_{k}^2}{|\wt S_k^2|}
      + \calo( | |\wt S_k|^{p-1} - |\wt R_k|^{p-1} - |\wt S_{k+1}|^{p-1} |),
\end{align*}
we have
\begin{align} \label{Skp3-Rkp3-esti}
    &\big|\int (|\wt S_k|^{p-3}\wt S_k^2 -| \wt R_k|^{p-3}\wt R_k^2
   -| \wt S_{k+1}|^{p-3}\wt S_{k+1}^2) \ol{\ve}^2dx\big|  \nonumber  \\
\leq&\int \bigg||\wt S_k|^{p-1}\frac{\wt S_k^2}{|\wt S_k|^2}
   -| \wt R_k|^{p-1}\frac{\wt R_k^2}{|\wt R_k|^2}
   -|\wt S_{k+1}|^{p-1}\frac{\wt S_{k+1}^2}{|\wt S_{k+1}|^2}\bigg| |\ve|^2dx \nonumber  \\
 \leq &  \int | \wt R_k|^{p-1}\bigg|\frac{\wt S_k^2}{|\wt S_k|^2}
     -\frac{\wt R_k^2}{|\wt R_k|^2} \bigg| |\ve|^2dx
    +\int | \wt S_{k+1}|^{p-1}\bigg|\frac{\wt S_k^2}{|\wt S_k|^2}
       -\frac{\wt  S_{k+1}^2}{|\wt  S_{k+1}|^2} \bigg| |\ve|^2dx     \nonumber \\
     & + \calo\(\int \big||\wt S_k|^{p-1}-| \wt S_k|^{p-1}-| \wt R_k|^{p-1}\big| |\ve|^2dx\) \nonumber  \\
= &  \int | \wt R_k|^{p-1}\bigg|\frac{\wt  S_{k}^2}{|\wt S_k|^2}
           -\frac{\wt R_k^2}{|\wt R_k|^2} \bigg| |\ve|^2dx
+\int | \wt  S_{k+1}|^{p-1}\bigg|\frac{\wt S_k^2}{|\wt S_k|^2}
         -\frac{\wt  S_{k+1}^2}{|\wt  S_{k+1} |^2} \bigg| |\ve|^2dx+ \calo(e^{-\delta t})  \nonumber \\
 =:& J_1 + J_2 + \calo(e^{-\delta t}).
\end{align}
where the last step is due to Lemma \ref{Lem-Skp1-Rkp1-ve2-exp}.

Below we estimate $J_1$ and $J_2$ separately.
For this purpose,
let us set
$d_*:= \min_{k\leq j\not =l\leq K}\{|v_jt + \a_j - v_lt-\a_l|\}$.
Similarly,
let $w_*:=\min_{k\leq j\leq K} w_j$, $w^*:=\max_{k\leq j\leq K} w_j$.
For every $k\leq j\leq K$, set
$$\Omega_j:= \bigg\{x\in \bbr^d: |x-v_jt -\a_j|
    \leq   \ve d_*  \bigg\}, $$
where $\ve$ is a small constant to be specified below.

$(i)$ Estimate of $J_1$.
We decompose
\begin{align}  \label{I-II}
   J_1 =&   \int_{\Omega_k^c} | \wt R_k|^{p-1}\bigg|\frac{\wt S_k^2}{|\wt S_k|^2}
        -\frac{\wt R_k^2}{|\wt R_k|^2} \bigg| |\ve|^2dx
        + \int_{\Omega_k} | \wt R_k|^{p-1}\bigg|\frac{\wt S_k^2}{|\wt S_k|^2}
       -\frac{\wt R_k^2}{|\wt R_k|^2} \bigg| |\ve|^2  dx
    :=J_{11}+J_{12}.
\end{align}

Note that,
for $x\in \Omega_k^c$,
since
\begin{align}
   |x-v_kt -\a_k|
   \geq& \ve  d_*
    > \frac{c}{2} t
\end{align}
for $t$ large enough, where $c>0$,
by \eqref{Q-decay}, there exist $C,\delta>0$ such that
\begin{align} \label{I-exp}
   J_{11}\leq C\|R_k\|^{p-1}_{L^\9(\Omega^c_k)}\|\ve\|_{L^2}^2\leq Ce^{-\delta t}\|\ve\|_{L^2}^2 .
\end{align}

Concerning the first term $J_{12}$ in \eqref{I-II},
since $Q(x)\sim e^{-\delta_0 |x|}$ (see \cite{BL83}),
we infer that
\begin{align} \label{wtRk-exp-lowbdd}
   |\wt R_{k}(t,x)| \geq C e^{-\delta_0 \frac{ \ve d_* }{w_k} }\geq C e^{-\delta_0 \frac{ \ve d_* }{w_*} },\ \ x\in \Omega_k.
\end{align}
On the other hand, for $x\in \Omega_k$ and any $j\not =k$,
\begin{align*}
   |x-v_jt -\a_j|
   \geq  |(v_k-v_j)t + \a_k-\a_j|
           - |x-v_kt-\a_k|
   \geq  (1-\ve)d_*,
\end{align*}
which yields that
\begin{align} \label{wtSk1-exp-upbdd}
   |\wt S_{k+1}(t,x)|
   \leq C \sum\limits_{j=k+1}^K
         e^{- \delta_0 \frac{(1-\ve)d_*}{w_j}}
   \leq C e^{-\delta_0 \frac{(1-\ve)d_*}{w^*}},\ \ x\in \Omega_k.
\end{align}
Hence, we obtain from \eqref{wtRk-exp-lowbdd} and \eqref{wtSk1-exp-upbdd} that,
for $\ve$ small enough such that
$$\ve<\frac{w_*}{w^*+w_*},$$
there exist $C,\delta>0$ such that
\begin{align}  \label{Sk1-Rk-exp}
   \bigg|\frac{\wt S_{k+1}(t,x)}{\wt R_k(t,x)}\bigg|
   \leq C e^{-\delta_0 d_* (\frac{(1-\ve)}{w^*}-\frac{\ve }{w_*})}
   \leq C e^{-\delta t},\ \ x\in \Omega_k.
\end{align}
Taking into account
\begin{align*}
   \frac{\wt S_k^2}{|\wt S_k^2|} - \frac{\wt R_k^2}{|\wt R_k^2|}
   =& \frac{\wt S_k^2 |\wt R_k|^2 -| \wt S_k|^2 \wt R_k^2
           + 2 \wt R_k \wt S_{k+1} |\wt R_k|^2 - 2 {\rm Re}(\wt R_k \wt S_{k+1}) \wt R_k^2 }
     {|\wt R_k + \wt S_{k+1}|^2 |\wt R_k|^2} \notag  \\
   \leq& \frac{|\wt R_k^{-1} \wt S_{k+1} |^2 + |\wt R_k^{-1} \wt S_{k+1}|}
           {|1+ \wt R_k^{-1} \wt S_{k+1}|^2}
\end{align*}
we thus lead to
\begin{align}  \label{wtSk2-wtRk2}
   \bigg| \frac{\wt S_k^2}{|\wt S_k^2|} - \frac{\wt R_k^2}{|\wt R_k^2|}    \bigg|
   \leq C e^{-\delta t}, \ \ x\in \Omega_k,
\end{align}
which yields that
\begin{align}  \label{II-exp}
       J_{12} \leq Ce^{-\delta t}\|\ve\|^2_{L^2}.
\end{align}

Thus, plugging  \eqref{I-exp} and \eqref{II-exp}  into \eqref{I-II}
we obtain
\begin{align} \label{wtRk-p-1-exp}
 J_2 \leq Ce^{-\delta t}.
\end{align}

$(ii)$ Estimate of $J_2$.
Set
\begin{align}
    \Omega  = \bigcup\limits_{j=k+1}^K \Omega_j
    = \bigcup\limits_{j=k+1}^K \bigg\{x\in \bbr^d: |x-v_jt -\a_j|
    \leq  \ve d_* \bigg\}
\end{align}
and decompose
\begin{align} \label{J2-J21-J22}
   J_2
   =& \int\limits_{\Omega} | \wt  S_{k+1}|^{p-1}\bigg|\frac{\wt S_k^2}{|\wt S_k|^2}
         -\frac{\wt  S_{k+1}^2}{|\wt  S_{k+1} |^2} \bigg| |\ve|^2dx
      + \int\limits_{\Omega^c} | \wt  S_{k+1}|^{p-1}\bigg|\frac{\wt S_k^2}{|\wt S_k|^2}
         -\frac{\wt  S_{k+1}^2}{|\wt  S_{k+1} |^2} \bigg| |\ve|^2dx \notag \\
   =& J_{21} + J_{22}.
\end{align}

Note that,
for every $k+1\leq j\leq K$,
since $Q(x)\sim e^{-\delta_0|x|}$,
\begin{align}
   |\wt R_j(t,x)|
   \geq C e^{-\delta_0 \frac{\ve d_*}{w_j}}
           \geq C e^{-\delta_0 \frac{\ve d_*}{w_*}},\ \ x\in \Omega_j.
\end{align}
Moreover, for $x\in \Omega/\Omega_j$,
there exists $j'\not = j$ such that $x\in \Omega_{j'}$ and so
\begin{align*}
   |x-v_jt-\a_j|
   \geq |v_{j'}t +\a_{j'} - v_jt -\a_j| - |x-v_{j'}t-\a_{j'}|
   \geq (1-\ve)d_*.
\end{align*}
This yields that
\begin{align*}
   |\wt R_j (t,x)|
     \leq C e^{-\delta_0\frac{(1-\ve)d_*}{w_j}}
   \leq C e^{-\delta_0 \frac{(1-\ve)d_*}{w^*}}, \ \ x\in \Omega/\Omega_j.
\end{align*}
Hence,
for $\ve$ very small such that
$$\ve<\frac{w_*}{w_*+w^*},$$
we obtain that
\begin{align}
    |\wt S_{k+1}|\geq C e^{-\delta_0  \frac{\ve d_*}{w_*}}
                - C' e^{-\delta_0 \frac{(1-\ve)d_*}{w^*}}
    \geq \frac 12 C e^{-\delta_0 \frac{\ve d_*}{w_*}}, \ \ x\in \Omega,\ k+1\leq j\leq K,
\end{align}
which yields that there exist $C,\delta >0$ such that
\begin{align}
   |\wt S_{k+1}|
    \geq  C e^{-\delta_0 \frac{\ve d_*}{w_*}}, \ \ x\in \Omega.
\end{align}

Moreover, for any $x\in \Omega$,
there exists $k+1\leq j\leq K$ such that $x\in \Omega_j$ and so
\begin{align}
   |x-v_kt -\a_k|
   \geq  |(v_j-v_k)t -(\a_j-\a_k)|  -  |x-v_jt -\a_j|
   \geq (1-\ve) d_*,
\end{align}
which yields that
\begin{align}
   |\wt R_k(t,x)| \leq e^{-\delta_0\frac{(1-\ve)d_*}{w_k}}\leq e^{-\delta_0\frac{(1-\ve)d_*}{w^*}}, \ \ x\in \Omega.
\end{align}

Thus, we infer that for $\ve$ possibly smaller such that
$$\ve< \frac{w_*}{w_*+w^*}, $$
then for $x\in \Omega$,
\begin{align}
   \bigg|\frac{\wt R_k(t,x)}{\wt S_{k+1}(t,x)}\bigg|
   \leq C e^{-\delta_0d_*(\frac{(1-\ve)}{w^*} - \frac{\ve}{w_*} )}
   \leq C e^{-\delta t},\ \ x\in \Omega.
\end{align}
Then, similar to \eqref{wtSk2-wtRk2}, we have
\begin{align}  \label{wtSk2-wtSk12}
   \bigg| \frac{\wt S_k^2}{|\wt S_k^2|} - \frac{\wt S_{k+1}^2}{|\wt S_{k+1}^2|}    \bigg|
   \leq C \frac{|\wt S_{k+1}^{-1}\wt R_k  |^2 + |\wt S_{k+1}^{-1}\wt R_k |}
           {|1+\wt S_{k+1}^{-1} \wt R_k |^2}
   \leq C e^{-\delta t}, \ \ x\in \Omega,
\end{align}
which yields that
\begin{align} \label{J21}
   J_{21} \leq C  e^{-\delta t} \|\ve\|^2_{L^2}.
\end{align}

Concerning $J_{22}$,
we see that for $x\in \Omega^c$,
for $k+1\leq j\leq K$,
\begin{align*}
   |x-v_jt -\a_j|
   \geq \ve d_*,
\end{align*}
and so
\begin{align}
   |\wt R_j(t,x)| \leq C e^{-\delta_0 \frac{\ve d_*}{w_j}},\ \ x\in \Omega^c.
\end{align}
This yields that there exist $C,\delta>0$ such that
\begin{align}
   |\wt S_{k+1}| \leq C \sum\limits_{j=k+1}^K |\wt R_j|
   \leq C e^{-\delta t},\ \ x\in \Omega^c,
\end{align}
and thus
\begin{align} \label{J22}
  J_{22} \leq  C e^{-\delta t} \|\ve\|^2_{L^2}.
\end{align}

Thus, we obtain from \eqref{J2-J21-J22}, \eqref{J21} and \eqref{J22} that
\begin{align} \label{wtSk-p-1-exp}
 J_2 \leq Ce^{-\delta  t}.
\end{align}

Therefore, plugging \eqref{wtRk-p-1-exp} and \eqref{wtSk-p-1-exp}
into \eqref{Skp3-Rkp3-esti}
we prove \eqref{Skp3-Rkp3-exp}
and thus finish the proof.
\hfill $\square$

Now, estimate \eqref{4} follows from Lemmas \ref{Lem-Skp1-Rkp1-exp},
\ref{Lem-Sk1-Rk-ve-exp}, \ref{Lem-Skp1-Rkp1-ve2-exp} and \ref{Lem-Skp3-Rkp3-ve2}.

\section*{Acknowledgements}

The authors thank Daomin Cao for valuable discussions.
M. R\"ockner and D. Zhang thank for the financial support
by the Deutsche Forschungsgemeinschaft
(DFG, German Science Foundation) through SFB 1283/2
2021-317210226 at Bielefeld University.
Y. Su is supported by NSFC (No. 11601482).
D. Zhang  is also grateful for the support by NSFC (No. 11871337)
and Shanghai Rising-Star Program 21QA1404500.
This work is also partially supported by the Key Laboratory of Scientific and Engineering Computing
(Ministry of Education).

\end{document}